\documentclass[10pt,reqno]{amsart}

\usepackage{amsthm,amsmath,amstext,amssymb,amscd,euscript, mathrsfs, dsfont,multicol,times,enumerate,subfig,sidecap}
\usepackage{enumerate}
\usepackage{amsmath, amssymb, amsthm}
\usepackage{mathrsfs}
\usepackage{esint}
\usepackage{xcolor}
\usepackage{mathtools}
\usepackage{bm}
\usepackage{bbm}

\usepackage[colorlinks=true, pdfstartview=FitV, linkcolor=blue, citecolor=red, urlcolor=blue]{hyperref} 

\newtheorem{thm}{Theorem}[section]
\newtheorem{corollary}[thm]{Corollary}
\newtheorem{lem}[thm]{Lemma}
\newtheorem{prop}[thm]{Proposition}
\theoremstyle{definition}
\newtheorem{defn}[thm]{Definition}
\newtheorem{example}[thm]{Example}
\newtheorem{assumption}[thm]{Assumption}
\newtheorem{rem}[thm]{Remark}

\newcommand\bC{\mathbb{C}}

\newcommand\bE{\mathbb{E}}
\newcommand\bF{\mathbb{F}}
\newcommand\bH{\mathbb{H}}

\newcommand\bL{\mathbb{L}}
\newcommand\bN{\mathbb{N}}

\newcommand\bP{\mathbb{P}}

\newcommand\bR{\mathbb{R}}

\newcommand\bZ{\mathbb{Z}}

\newcommand\cB{\mathcal{B}}

\newcommand\cD{\mathcal{D}}
\newcommand\cF{\mathcal{F}}

\newcommand\cH{\mathcal{H}}
\newcommand\cI{\mathcal{I}}

\newcommand\cL{\mathcal{L}}
\newcommand\cP{\mathcal{P}}

\newcommand\cS{\mathcal{S}}

\newcommand\cT{\mathcal{T}}

\newcommand\cpar{\text{$[$\kern-.38em$[$}}
\newcommand\cbrk{\text{$]$\kern-.15em$]$}}
\newcommand\opar{\text{\,\raise.2ex\hbox{${\scriptstyle
|}$}\kern-.34em$($}}
\newcommand\obrk{\text{$)$\kern-.34em\raise.2ex\hbox{${\scriptstyle |}$}}\,}

\newcommand{\mysection}[1]{\section{#1}
\setcounter{equation}{0}}

\def\XXint#1#2#3{{\setbox0=\hbox{$#1{#2#3}{\int}$ }
\vcenter{\hbox{$#2#3$ }}\kern-.58\wd0}}

\makeatletter
\def\@tocline#1#2#3#4#5#6#7{\relax
  \ifnum #1>\c@tocdepth 
  \else
    \par \addpenalty\@secpenalty\addvspace{#2}%
    \begingroup \hyphenpenalty\@M
    \@ifempty{#4}{%
      \@tempdima\csname r@tocindent\number#1\endcsname\relax
    }{%
      \@tempdima#4\relax
    }%
    \parindent\z@ \leftskip#3\relax \advance\leftskip\@tempdima\relax
    \rightskip\@pnumwidth plus4em \parfillskip-\@pnumwidth
    #5\leavevmode\hskip-\@tempdima
      \ifcase #1
       \or\or \hskip 1em \or \hskip 2em \else \hskip 3em \fi%
      #6\nobreak\relax
    \dotfill\hbox to\@pnumwidth{\@tocpagenum{#7}}\par
    \nobreak
    \endgroup
  \fi}
\makeatother

\begin{document}
\title[SRDAEs with infinitesimal generators of subordinate Brownian motions and colored noises]
{Sobolev regularity theory for stochastic reaction-diffusion-advection equations with spatially homogeneous colored noises and infinitesimal generators of subordinate Brownian motions}

\author[J.-H. Choi]{Jae-Hwan Choi}
\address[J.-H. Choi]{School of Mathematics, Korea Institute for Advanced Study, 85 Hoegiro Dongdaemun-gu, Seoul 02455, Republic of Korea}
\email{jhchoi@kias.re.kr}

\author[B.-S. Han]{Beom-Seok Han}
\address[B.-S. Han]{School of Mathematics, Statistics and Data Science, Sungshin Women's University, 2, Bomun-ro 34 da-gil, Seongbuk-gu, Seoul, 02844, Republic of Korea.} 
\email{b\_han@sungshin.ac.kr}

\author[D. Park]{Daehan Park}
\address[D. Park]{Department of Mathematics, Kangwon National University, 1 Kangwondaehakgil, Chucheon, Gangwon State, 24341, Republic of Korea.} 
\email{daehanpark@kangwon.ac.kr}

\thanks{}



\maketitle
\begin{abstract}
This article investigates the existence, uniqueness, and regularity of solutions to nonlinear stochastic reaction-diffusion-advection equations (SRDAEs) with spatially homogeneous colored noises and infinitesimal generators of subordinate Brownian motions in mixed norm $L_q(L_p)$-spaces.
We introduce a new condition—strongly reinforced Dalang's condition—on colored noise, which facilitates a deeper understanding of the complicated relation between nonlinearities and stochastic forces.
Additionally, we establish the space-time H\"older type regularity of solutions.
\end{abstract}
\tableofcontents

\mysection{Introduction and Main results}
\subsection{Goal and Settings}
Consider the following non-linear stochastic reaction-diffusion-advection equations (SRDAEs):
\begin{equation}
\label{24.03.13.13.13}
\partial_tu=\phi(\Delta)u+\zeta F(u)+\vec{b}\cdot\nabla_x(B(u))+\xi\varphi(u)\dot{W}\quad \text{on } (0,\infty)\times\bR^d;\quad u(0,\cdot)=u_0.
\end{equation}
Here, 
\begin{itemize}
    \item the operator $\phi(\Delta)$ represents a pseudo-differential operator defined by
    $$
    \phi(\Delta)f(x):=-\frac{1}{(2\pi)^d}\int_{\bR^d}\int_{\bR^d}\phi(|\xi|^2)f(y)\mathrm{e}^{\mathrm{i}(x-y)\cdot\xi}\mathrm{d}\xi\mathrm{d}y,
    $$
where $\phi$ is a Bernstein function,
    \item the random noise $\dot{W}$ is a spatially homogeneous colored noise with a nonnegative definite correlation $\pi$,
    \item we call non-linear terms $F(u)$, $B(u)$ and $\varphi(u)$ \emph{strong dissipative term}, \emph{generalized Burgers' term} and \emph{super-linear multiplicative noise term}, respectively.
\end{itemize}

\begin{defn}
    An infinitely differentiable function $\phi:(0,\infty)\to[0,\infty)$ is called a \emph{Bernstein function} if 
    $$
    (-1)^nD^n\phi(\lambda)\leq0,\quad \forall n\in\bN,\,\lambda\in(0,\infty).
    $$
\end{defn}

    \begin{defn}
    $(i)$ A tempered distribution $\pi$ is called a \emph{nonnegative definite} if for a Schwartz function $\phi$,
    $$
    (\pi,\phi\ast\bar{\tilde{\phi}})\geq0,\quad \tilde{\phi}(x):=\phi(-x).
    $$
    Here $\overline{z}$ is the complex conjugate of $z$.
        
    $(ii)$ We say that a mean zero Gaussian random field $W$ is a \emph{spatially homogeneous colored noise} with a nonnegative definite correlation $\pi$ if for Schwartz functions $\psi_1$ and $\psi_2$,
    $$
    \mathbb{E}[W(\psi_1)W(\psi_2)]=\int_0^{\infty}(\pi,\psi_1(t,\cdot)\ast\bar{\tilde{\psi}}_2(t,\cdot))\mathrm{d}t.
    $$
    \end{defn}

\begin{rem}
The concept of spatially homogeneous colored noise $W$ can be elucidated through an infinite series of Itô stochastic integrals. For any Schwartz function $\psi$, we express
    $$
    W(\psi)=\sum_{k=1}^{\infty}\int_0^{\infty}\int_{\bR^d}\psi(t,x)(\pi\ast e_k)(x)\mathrm{d}x\mathrm{d}w_t^k.
    $$
    Here, $\{w_{\cdot}^k\}_{k=1}^{\infty}$ denotes a sequence of independent one-dimensional Wiener processes, and $\{e_k\}_{k=1}^{\infty}$ forms a real-valued complete orthonormal basis of the Hilbert space $\mathfrak{H}$, which is derived from $\pi$.
    Consequently, the noise term $W$ can be extended to an $L_2(\Omega)$-valued martingale measure:
    $$
    W(\mathrm{d}t,\mathrm{d}x):=\sum_{k=1}^{\infty}(\pi\ast e_k)(x)\mathrm{d}x\mathrm{d}w_t^k.
    $$
    This formulation allows us to interpret equation \eqref{24.03.13.13.13} in an integral form: for any test function $\eta \in C_c^{\infty}(\mathbb{R}^d)$,
    \begin{align*}
        (u(t,\cdot),\eta)&=(u_0,\eta)+\int_0^t(u(s,\cdot),\phi(\Delta)\eta)+(\zeta(s,\cdot) F(s,u(s,\cdot)),\eta)\mathrm{d}s\\
        &\quad+\int_0^t(B(s,u(s,\cdot)),\vec{b}(s,\cdot)\cdot\nabla_{x}\eta)\mathrm{d}s+\sum_{k=1}^{\infty}\int_0^t(\xi(s,\cdot)\varphi(s,u(s,\cdot))\pi\ast e_k,\eta)\mathrm{d}w_s^k.
    \end{align*}
    Therefore, adopting It\^o's differential notation, the equation simplifies to
    $$
    \mathrm{d}u=(\phi(\Delta)u+\zeta F(u)+\vec{b}\cdot\nabla_x(B(u)))\mathrm{d}t+\xi\varphi(u)(\pi\ast e_k)\mathrm{d}w_t^k;\quad u(0,\cdot)=u_0.
    $$
    For more discussions on the Hilbert space $\mathfrak{H}$ and the basis ${e_k}$, we refer the reader to \cite{CH2021}.
\end{rem}

In this article, we investigate the existence, uniqueness, and regularity of solutions to \eqref{24.03.13.13.13} in Sobolev mixed norm spaces $\bH_{p,q}^{\phi,\gamma}(\tau):=L_q(\Omega\times (0,\tau],\mathcal{P},\mathbb{P}\times\mathrm{d}t;H_p^{\phi,\gamma}(\bR^d))$.
Here $(\Omega, \bF, \bP)$ is a complete probability space with a filtration $\{\bF_t\}_{t\geq0}$ satisfying the usual conditions and $\cP$ is the predictable $\sigma$-field related to $\bF_t$, $\tau$ is a $\mathbb{F}_{t}$-adapted stopping time, and the norm of $\bH_{p,q}^{\phi,\gamma}(\tau)$ is defined in the usual way, \textit{i.e.}, 
\begin{equation} \label{norm}
\| u \|^q_{\bH^{\phi,\gamma}_{p,q}(\tau)} := \bE\left[ \int^{\tau}_0 \| u(t,\cdot) \|^q_{H^{\phi,\gamma}_{p}(\bR^d)}\mathrm{d}t\right]. 
\end{equation}

\begin{defn}
For $p>1$ and $\gamma \in \mathbb{R}$,
\begin{enumerate}[(i)]
\item  let $H_p^{\phi,\gamma}(\bR^d)$ denote the class of tempered distributions $u$ on $\bR^d$ such that
$$ \| u \|_{H_p^{\phi,\gamma}(\bR^d)} := \| (1-\phi(\Delta))^{\gamma/2} u\|_{L_p(\bR^d)} = \| \cF^{-1}[ (1+\phi(|\cdot|^2))^{\gamma/2}\cF[u]]\|_{L_p(\bR^d)}<\infty,
$$

\item let $H_p^{\phi,\gamma}(\bR^d;\ell_2)$ denote the class of $\ell_2$-valued tempered distributions $g=(g^1,g^2,\cdots)$ on $\bR^d$ such that
$$ \|g\|_{H_{p}^{\phi,\gamma}(\bR^d;\ell_2)}:= \| | (1-\phi(\Delta))^{\gamma/2} g|_{\ell_2}\|_{L_p} = \| |\cF^{-1}[ (1+\phi(|\cdot|^2))^{\gamma/2}\cF[g]]|_{\ell_2} \|_{L_p(\bR^d)}<\infty.
$$	
\end{enumerate}
Here, $\cF$ denotes the $d$-dimensional Fourier transform and, $\cF^{-1}$ denotes the $d$-dimensional inverse Fourier transform.
\end{defn}

\subsection{Background}

The study of the solvability of SRDAEs poses significant challenges due to the irregular perturbation from noise.
As such, the solvability of SRDAEs has attracted substantial interest from the mathematical research community.
This section revisits various contributions regarding the non-linear SRDAEs.

\emph{Case 1: Stochastic reaction-diffusion equations with a solely stochastic non-linear reaction term}
$$
\partial_t u = \mathcal{L}u + |u|^{1+\lambda_{\text{s.m.}}}\dot{W}.
$$
When $d = 1$, $\mathcal{L}$ is a second-order elliptic operator, and $\dot{W}$ is the space-time white noise, C. Mueller \cite{Muller1991} established that for $0 \leq \lambda_{\text{s.m.}} < \frac{1}{2}$, a mild solution exists globally in time on a finite spatial interval with periodic boundary conditions.
Moreover, C. Mueller \cite{mueller2000critical} also demonstrated that for $\lambda_{\text{s.m.}} > \frac{1}{2}$, the mild solution exhibits finite-time explosion with non-zero probability under the same conditions.
In the critical scenario where $\lambda_{\text{s.m.}} = \frac{1}{2}$, M. Salins \cite{Salin2023} showed that solution explosions are precluded within a periodic domain.
In instances where non-explosive solutions are considered, for $0 \leq \lambda_{\text{s.m.}} < \frac{1}{2}$, significant results have emerged.
Notably, N.V. Krylov \cite[Section 8.4]{kry1999analytic} established the existence and Sobolev regularity of the strong solution over $\mathbb{R}$.
B.-S. Han \cite{HK2020} confirmed the existence, uniqueness, Sobolev regularity, and boundary behavior of strong solutions on a finite interval with Dirichlet boundary conditions.
Additionally, B.-S. Han \cite{Han2021} further substantiated the existence, uniqueness, and Sobolev regularity of strong solutions over $\mathbb{R}$ with $\mathcal{L} = a\Delta^{\alpha/2} + bD + c$.
These results predominantly address the dynamics influenced by space-time white noise.
J.-H. Choi and B.-S. Han \cite{CH2021} demonstrated the existence, uniqueness, and Sobolev regularity of strong solutions in $\mathbb{R}^d$ with spatially homogeneous colored noise $\dot{W}$ and $\mathcal{L} = a^{ij}D^{ij} + b^iD^i + c$.
For further discussions on the range $-1 < \lambda_{\text{s.m.}} < 0$, the reader is referred to \cite{KCE2010,CLE2014,LE2011}.

\emph{Case 2: Stochastic reaction-diffusion equations with deterministic and stochastic reaction terms}
$$
\partial_t u = \mathcal{L}u + F(u) + \varphi(u)\dot{W}.
$$
Assuming $\mathcal{L}$ as a second-order elliptic operator and $\dot{W}$ as spatially homogeneous colored noise, the scenario where $\varphi$ exhibits linear growth and $F$ has polynomial growth with a negative leading coefficient has been addressed by S. Cerrai, G. Da Prato, and F. Flandoli \cite{Cerrai2003,Cerrai2011,cerrai2013pathwise}.
They proved the existence and uniqueness of mild solutions over bounded open sets in $\mathbb{R}^d$ with regular boundaries.
Additionally, in the case where $|F(u)| \lesssim |u|\log|u|$ and $|\varphi(u)| \lesssim |u|(\log|u|)^{\kappa/2}$ for some $\kappa > 0$ and large $|u|$, R.C. Dalang, D. Khoshnevisan, and T. Zhang \cite{DKZ2019} established the existence and uniqueness of the mild solution with space-time white noise.
L. Chen and J. Huang \cite{chen2022superlinear} extended these works to settings with spatially homogeneous colored noise under reinforced Dalang's condition.
When $F(u)\text{sign}(u) \lesssim -|u|^{1+\beta}$ and $|\varphi(u)| \lesssim (1 + |u|)^{1+\gamma}$, M. Salins \cite{Salin2022} confirmed the existence of a global mild solution.
Inspired by the result of Salins \cite{Salin2022}, which suggests that larger values of $\gamma$ can be chosen if $\beta$ is substantial, B.-S. Han and J. Yi \cite{HY2024} validated the existence, uniqueness, and Sobolev regularity of strong solutions with $F(u) = -|u|^{1+\beta}$ and $\varphi(u) = |u|^{1+\gamma}$ in contexts with spatially homogeneous colored noise meeting reinforced Dalang's condition.
Recently, A. Agresti and M.C. Veraar \cite{AM2023,AM2024} demonstrated the existence, uniqueness, and regularity of solutions within critical spaces for stochastic reaction-diffusion equations with transport noise in periodic domains.

\emph{Case 3: Stochastic reaction-diffusion-advection equations}
$$
\partial_t u = \mathcal{L}u + F(u) + \vec{b}\cdot \nabla_x(B(u)) + \varphi(u)\dot{W}.
$$
Assuming that $\mathcal{L}$ is a second-order elliptic operator, the field has evolved significantly since the seminal contributions of J.D. Cole and E. Hopf \cite{Cole1950,Hopf1950}, who initially studied Burgers' equation.
Following their pioneering efforts, L. Bertini, N. Cancrini, and G. Jona-Lasinio \cite{BCJ1994} established the existence of a mild solution via the Cole-Hopf transformation over $\mathbb{R}$, considering $F=0$, $\varphi=1$, and space-time white noise $\dot{W}$.
Further developments were made by G. Da Prato and D. Gatarek \cite{DG1995}, who demonstrated the existence of a mild solution in the unit interval with Dirichlet boundary conditions, where $F=0$, $B(u)=u^2$, $\varphi$ was bounded and Lipschitz continuous, and influenced by space-time white noise $\dot{W}$.
Strong solutions have also been extensively investigated.
I. Gy\"ongy \cite{Gyongy1998} verified the existence and uniqueness of a strong solution within the unit interval with Dirichlet boundary conditions, where $B(u)=u^2$, both $F$ and $\varphi$ exhibited linear growth, and $\dot{W}$ was space-time white noise.
This analysis was expanded to the real line in \cite{GN1999}.
I. Gy\"ongy and C. Rovira \cite{gyongy1999stochastic} further extended these results, proving the existence and uniqueness of strong solutions in unit intervals with Dirichlet boundary conditions, white in time noise, and a polynomial growth function $B$.
Their findings were later generalized to bounded convex domains in $\mathbb{R}^d$ \cite{gyongy2000lp}.
J.A. Le\'on, D. Nualart, and R. Pettersson \cite{leon2000stochastic} built upon the results of \cite{BCJ1994} by considering bounded Lipschitz continuous $\varphi$ within the unit interval under Dirichlet boundary conditions, where $F=0$ and $B(u)=u^2$, also providing moment estimates. 
P. Lewis and D. Nualart \cite{LD2018} generalized these earlier studies \cite{BCJ1994,leon2000stochastic}, contributing to the understanding of the H\"older regularity of moment estimates of the solution.
Most recently, B.-S. Han \cite{Han2022,Han2023} has confirmed the existence, uniqueness, and Sobolev regularity of strong solutions across the entire space with a linear growth function $F$.
Further developments with the time-fractional derivative were reported in \cite{Han2024}.
For an abstract semigroup approach to SRDAEs, we refer the reader to \cite{RS2006}.

Despite significant advancements in the study of SRDAEs as evidenced by the results discussed previously, several fundamental and intriguing questions remain unaddressed.
These questions include:

\begin{enumerate}
    \item \textbf{Non-local operators}: Beyond the findings of B.-S. Han \cite{Han2021}, research has predominantly focused on second-order elliptic operators.
    This observation prompts the question: \emph{Is it feasible to establish solvability for SRDAEs that involve infinitesimal generators of subordinate Brownian motions?}
    
    \item \textbf{Regularity theory in mixed Norm $L_q(L_p)$-Spaces}: Aside from the contributions of A. Agresti and M.C. Veraar \cite{AM2023,AM2024}, existing research primarily addresses regularity within $L_p$-spaces.
    This situation leads to the question: \emph{Is it possible to extend solvability results to SRDAEs in mixed norm $L_q(L_p)$-spaces?}
    
    \item \textbf{Nonlinear terms}: Prior studies have not considered the presence of three polynomial growth nonlinear terms, $F$, $B$, and $\varphi$.
    Consequently, we ask: \emph{Is it feasible to establish solvability for SRDAEs featuring three polynomial growth nonlinear terms $F$, $B$, and $\varphi$?}
    
    \item \textbf{Spatially homogeneous colored noise}: Most results concerning spatially homogeneous colored noise rely on the reinforced Dalang's condition.
    This consideration leads to the questions: \emph{Can we characterize the reinforced Dalang's condition effectively?
    If so, how can this characterization facilitate the resolution of SRDAEs?}
\end{enumerate}
This article is dedicated to addressing the above questions, aiming to fill the gaps identified in current research on SRDAEs.
\emph{On (4),} rather than assuming the reinforced Dalang condition abstractly, we \emph{characterize} it via an explicit, verifiable integrability criterion on the correlation measure—Assumption \ref{main_ass_corr}$(\delta_0,\gamma,r)$ (equivalently, $\nu_{\delta_0,\gamma,r}<\infty$). 
This parameterized form, through the noise smoothness $r>1$ and the weak lower scaling index $\delta_0$ of $\phi$, enters directly into our a priori estimates and yields \emph{weaker} requirements on the noise measure together with a \emph{larger} admissible regularity range $\gamma$ (see Theorem \ref{main} and Remark \ref{25.08.13.15.15}).

\subsection{Conditions on spatially homogeneous colored noise}
This subsection details the conditions imposed on spatially homogeneous colored noise $W$ characterized by a nonnegative definite correlation $\pi$.
Here is our main assumption on correlation $\pi$.
Throughout this article, we always set 
$$
\delta_0\in(0,1], \quad \gamma\in(0,1)\quad \text{and}\quad r\in[1,d/(d-\delta_0(1-\gamma))).
$$
\begin{assumption}[$\delta_0,\gamma,r$] 
\label{main_ass_corr}
    The spatially homogeneous colored noise $W$ with a nonnegative definite correlation $\pi$ satisfies the following:
    \begin{enumerate}[(i)]
        \item The correlation $\pi$ is a nonnegative and nonnegative definite tempered distribution.
        \item The correlation $\pi$  satisfies that
$$
\begin{cases}
    \int_{|x|<1}|x|^{2\alpha_{\delta_0,\gamma,r}-d}\pi(\mathrm{d}x)<\infty,\quad &\text{if } 0<\alpha_{\delta_0,\gamma,r}<d/2,\\
    \int_{|x|< 1}\log(1/|x|)\pi(\mathrm{d}x)<\infty,\quad &\text{if } \alpha_{\delta_0,\gamma,r}=d/2,\\
    \text{no specific conditions on $\pi$,}\quad &\text{if } \alpha_{\delta_0,\gamma,r}>d/2,
\end{cases}
$$
where $\alpha_{\delta_0,\gamma,r}:=d-r(d-\delta_0(1-\gamma))\in(0,d)$.
    \end{enumerate}
\end{assumption}

When $r=1$, Assumption \ref{main_ass_corr} ($\delta_0,\gamma,1$) is called \emph{the reinforced Dalang’s condition} which is equivalent to
    $$
    \int_{\bR^d}(R_{\delta_0(1-\gamma)}\ast R_{\delta_0(1-\gamma)})(x)\pi(\mathrm{d}x)=\int_{\bR^d}R_{2\delta_0(1-\gamma)}(x)\pi(\mathrm{d}x)=:\nu_{\delta_0,\gamma,1}<\infty,
    $$
where
    \begin{equation}
    \label{24.03.04.15.14}
    R_{\beta}(x):=\frac{1}{(2\pi)^{d/2}}\int_{\bR^d}\frac{\mathrm{e}^{\mathrm{i}x\cdot\xi}}{(1+|\xi|^2)^{\beta/2}}\mathrm{d}\xi,\quad \beta>0.
    \end{equation}
The reinforced Dalang’s condition plays a crucial role in the study of the regularity of solutions to SPDEs driven by colored noise; see, for example, \cite{FS2006, HY2024, Sanzsole2002}.

For $r>1$, we say that Assumption \ref{main_ass_corr} ($\delta_0,\gamma,r$) is \emph{the strongly reinforced Dalang’s condition}.
We explore the connections between the reinforced and strongly reinforced Dalang's conditions through the following proposition.
\begin{prop}
\label{24.04.21.21.16}
    Let $\delta_0\in(0,1]$, $\gamma\in(0,1)$, $r_1,r_2\in[1,d/(d-\delta_0(1-\gamma)))$.
    Suppose that $\pi$ is a correlation satisfying Assumption \ref{main_ass_corr} $(\delta_0,\gamma,r_2)$.
    If $r_1< r_2$, then $\pi$ also satisfies Assumption \ref{main_ass_corr} $(\delta_0,\gamma,r_1)$.
\end{prop}
The proof of Proposition \ref{24.04.21.21.16} is presented at the end of this subsection.

To establish Proposition \ref{24.04.21.21.16}, we rely on specific properties of $R_{\beta}$, as detailed in \cite[Proposition 1.2.5]{grafakos2014modern}.
\begin{enumerate}[(1)]
    \item $R_{\beta}$ is strictly positive and radially decreasing.
    \item for all $\beta>0$,
    \begin{equation}
    \label{24.03.06.12.01}
    \|R_{\beta}\|_{L_1(\bR^d)}=1.
    \end{equation}
    \item for all $x\in\bR^d\setminus\{0\}$ and $\beta\in(0,d)$,
\begin{equation}
\label{24.04.09.15.21}
 R_{\beta}(x)\leq N(\beta,d)|x|^{\beta-d}1_{|x|\leq2}+N(\beta,d)\mathrm{e}^{-\frac{|x|}{2}}1_{|x|>2}.
\end{equation}
    \item for $|x|\leq 2$ and $\beta\in(0,d)$,
    \begin{equation}
    \label{25.08.06.17.55}
    N(\beta,d)|x|^{\beta-d}\leq R_{\beta}(x).
    \end{equation}
\end{enumerate}
Moreover, by synthesizing properties (1) and (4), we derive the extended property:

\vspace{1mm}

\noindent $(4')$ for $\kappa>0$ and $\beta\in(0,d)$, there exists $N=N(\beta,d,\kappa)>0$ such that for $|x|<\kappa$,
\begin{equation}
\label{24.04.21.14.50}
    N|x|^{\beta-d}\leq R_{\beta}(x).
\end{equation}
By combining (3) and (4), we ascertain that:
\begin{equation}
\label{24.04.21.20.52}
\|R_{\delta_0(1-\gamma)}\|_{L_r(\bR^d)}<\infty \Longleftrightarrow1\leq r<\frac{d}{d-\delta_0(1-\gamma)}\Longleftrightarrow \alpha_{\delta_0,\gamma,r}\in(0,d).
\end{equation}

Based on the well-defined properties of $R_{\delta_0(1-\gamma)}$, we can find conditions that correspond equivalently to Assumption \ref{main_ass_corr} ($\delta_0,\gamma,r$).
\begin{prop}
\label{24.04.21.21.05}
Let $\delta_0\in(0,1]$, $\gamma\in(0,1)$, $r\in[1,d/(d-\delta_0(1-\gamma)))$ and $\pi$ be a tempered measure defined on $\bR^d$.
Then, the following conditions are equivalent:
    \begin{enumerate}[(i)]
        \item Assumption \ref{main_ass_corr} ($\delta_0,\gamma,r$) holds.
        \item The convolution $(R_{\delta_0(1-\gamma)}^{r}\ast R_{\delta_0(1-\gamma)}^{r})$ belongs to $L_1(\bR^d,\pi)$;
        \begin{equation}
\label{24.04.21.12.31}
\nu_{\delta_0,\gamma,r}:=\int_{\bR^d}(R_{\delta_0(1-\gamma)}^{r}\ast R_{\delta_0(1-\gamma)}^{r})(x)\pi(\mathrm{d}x)<\infty.
\end{equation}
    \item The convolution $(R_{\delta_0(1-\gamma)}^{r}\ast R_{\delta_0(1-\gamma)}^{r})$ belongs to $L_1(B_1(0),\pi)$;
    \begin{equation}
\label{24.04.09.14.33}
    \int_{|x|< 1}(R_{\delta_0(1-\gamma)}^r\ast R_{\delta_0(1-\gamma)}^r)(x)\pi(\mathrm{d}x)<\infty.
\end{equation}
    \end{enumerate}
\end{prop}
\begin{proof}
    The implication from condition $(ii)$ to $(iii)$ is straightforward.
    By examining the exponential decay of $R_{\delta_0(1-\gamma)}$ as $|x| \to \infty$, as outlined in \eqref{24.04.09.15.21} and corroborated by \cite[Lemma 4.3]{CH2021}, we infer that the convolution $(R_{\delta_0(1-\gamma)}^r \ast R_{\delta_0(1-\gamma)}^r)(x)$ exhibits similar exponential decay for large $x$.
    Furthermore, the characterization of $\pi$ as a tempered measure, detailed in \cite[Definition 2.1]{CH2021}, stipulates the existence of some $k \geq 0$ such that
\begin{equation}
\label{24.04.21.15.46}
A_{\pi}:=\int_{\mathbb{R}^d}\frac{1}{(1+|x|^2)^{k/2}}\pi(\mathrm{d}x) < \infty.
\end{equation}
This integral condition on $\pi$ robustly supports the assertion that $(iii)$ implies $(ii)$.
We proceed by establishing the equivalence between conditions $(i)$ and $(iii)$ through three cases based on the value of $\alpha_{\delta_0,\gamma,r}$.

    \textbf{Case 1.} Assume $\alpha_{\delta_0,\gamma,r}\in(d/2,d)$.
    Given \eqref{24.04.21.15.46}, we demonstrate the existence of constants $c$ and $C$ such that for all $|x|<1$,
$$
c\leq(R_{\delta_0(1-\gamma)}^r\ast R_{\delta_0(1-\gamma)}^r)(x)\leq C.
$$
    When $\alpha_{\delta_0,\gamma,r} \in (d/2,d)$, this translates to $r < \frac{d}{2(d-\delta_0(1-\gamma))}$, allowing us to
\begin{align*}
(R_{\delta_0(1-\gamma)}^r\ast R_{\delta_0(1-\gamma)}^r)(x)&\leq \|R_{\delta_0(1-\gamma)}\|_{L_{2r}(\bR^d)}=:C<\infty.
\end{align*}
Conversely,
\begin{align*}
(R_{\delta_0(1-\gamma)}^r \ast R_{\delta_0(1-\gamma)}^r)(x) &> N \int_{1/2 < |y| < 3} R_{\delta_0(1-\gamma)}(x-y)^r R_{\delta_0(1-\gamma)}(y)^r , \mathrm{d}y,
\end{align*}
where the strict positivity and radial decrease of $R_{\delta_0(1-\gamma)}$ imply for $1/2 < |y| < 3$,
$$
R_{\delta_0(1-\gamma)}(y)\geq R_{\delta_0(1-\gamma)}(3,0,\cdots,0)>0.
$$
Thus,
\begin{align*}
    &\int_{1/2<|y|<3}R_{\delta_0(1-\gamma)}(x-y)^rR_{\delta_0(1-\gamma)}(y)^r\mathrm{d}y\\
    &\geq R_{\delta_0(1-\gamma)}(3,0,\cdots,0)\int_{3/2<|x-y|<2}R_{\delta_0(1-\gamma)}(x-y)^r\mathrm{d}y=:c>0.
\end{align*}

\textbf{Case 2.} Assume $\alpha_{\delta_0,\gamma,r}\in(0,d/2)$. The objective here is to establish that there exist positive constants $c$ and $C$ such that:
$$
c|x|^{2\alpha_{\delta_0,\gamma,r}-d}\leq(R_{\delta_0(1-\gamma)}^r\ast R_{\delta_0(1-\gamma)}^r)(x)\leq C|x|^{2\alpha_{\delta_0,\gamma,r}-d},\quad \forall |x|<1.
$$
From \eqref{24.04.09.15.21}, for any $x \in \mathbb{R}^d \setminus \{0\}$,
\begin{align*}
(R_{\delta_0(1-\gamma)}^r \ast R_{\delta_0(1-\gamma)}^r)(x) &\leq N \int_{\mathbb{R}^d} |y|^{\alpha_{\delta_0,\gamma,r} - d} |x-y|^{\alpha_{\delta_0,\gamma,r} - d} \mathrm{d}y.
\end{align*}
We analyze this integral by dividing $\mathbb{R}^d$ into two regions, $D_1(x):=\{y\in\bR^d:|y|\leq 2|x|\}$ and $D_2(x):=\{y\in\bR^d:|y|> 2|x|\}$, and evaluate
\begin{equation}
\label{24.04.09.19.59}
\begin{aligned}
    \int_{D_1(x)}|y|^{\alpha_{\delta_0,\gamma,r}-d}|x-y|^{\alpha_{\delta_0,\gamma,r}-d}\mathrm{d}y&\leq N|x|^{\alpha_{\delta_0,\gamma,r}-d}\int_{|y|\leq 3|x|}|y|^{\alpha_{\delta_0,\gamma,r}-d}\mathrm{d}y\\
    &=N|x|^{2\alpha_{\delta_0,\gamma,r}-d}.
\end{aligned}
\end{equation}
For $y\in D_2(x)$, note the inequalities $|x| \leq |y| - |x| \leq |x-y| \leq |x| + |y| \leq \frac{3|y|}{2}$, which ensure
$$
\int_{D_2(x)}|y|^{\alpha_{\delta_0,\gamma,r}-d}|x-y|^{\alpha_{\delta_0,\gamma,r}-d}\mathrm{d}y\leq N\int_{|x-y|\geq |x|}|x-y|^{2\alpha_{\delta_0,\gamma,r}-2d}\mathrm{d}y=N|x|^{2\alpha_{\delta_0,\gamma,r}-d}.
$$
Conversely, since $|x|<1$, utilizing \eqref{24.04.21.14.50},
\begin{align*}
    (R_{\delta_0(1-\gamma)}^r\ast R_{\delta_0(1-\gamma)}^r)(x)&\geq \int_{|x|/2<|y|<3|x|/2}R_{\delta_0(1-\gamma)}(x-y)^rR_{\delta_0(1-\gamma)}(y)^r\mathrm{d}y\\
    &\geq N|x|^{\alpha_{\delta_0,\gamma,r}-d}\int_{|x|/2<|y|<3|x|/2}R_{\delta_0(1-\gamma)}(x-y)^r\mathrm{d}y\\
    &\geq N|x|^{\alpha_{\delta_0,\gamma,r}-d}\int_{|x-y|<|x|/2}R_{\delta_0(1-\gamma)}(x-y)^r\mathrm{d}y=N|x|^{2\alpha_{\delta_0,\gamma,r}-d}.
\end{align*}

\textbf{Case 3.} Consider the case where $\alpha_{\delta_0,\gamma,r} = d/2$.
The objective in this scenario is to demonstrate the existence of positive constants $c$ and $C$ such that
$$
c\log(1/|x|)\leq(R_{\delta_0(1-\gamma)}^r\ast R_{\delta_0(1-\gamma)}^r)(x)\leq C(1+\log(1/|x|)),\quad \forall |x|<1.
$$
For $x\in\bR^d\setminus\{0\}$,
\begin{align*}
(R_{\delta_0(1-\gamma)}^r\ast R_{\delta_0(1-\gamma)}^r)(x)&\leq N\int_{|y|\leq2}|y|^{r(\delta_0(1-\gamma)-d)}|R_{\delta_0(1-\gamma)}(x-y)|^r\mathrm{d}y\\
&\quad + N\int_{|y|>2}\mathrm{e}^{-\frac{|y|}{2}}|R_{\delta_0(1-\gamma)}(x-y)|^r\mathrm{d}y\\
&\leq N\int_{|y|\leq2}|y|^{-d/2}|R_{\delta_0(1-\gamma)}(x-y)|^r\mathrm{d}y+N\|R_{\delta_0(1-\gamma)}\|_{L_r(\bR^d)}.
\end{align*}
Given that $\|R_{\delta_0(1-\gamma)}\|_{L_r(\mathbb{R}^d)} < \infty$, our analysis focuses predominantly on
$$
\int_{|y|\leq2}|y|^{-d/2}|R_{\delta_0(1-\gamma)}(x-y)|^r\mathrm{d}y.
$$
Since $|x| < 1$, we have $|x-y| \leq 3$, and thus
\begin{align*}
    &\int_{|y|\leq2}|y|^{-d/2}|R_{\delta_0(1-\gamma)}(x-y)|^r\mathrm{d}y\\
    &\leq N\int_{|y|\leq2}|y|^{-d/2}|x-y|^{-d/2}\mathrm{d}y\\
    &=N\int_{|y|\leq|x|/2}\cdots\mathrm{d}y+N\int_{|x|/2<|y|\leq2|x|}\cdots\mathrm{d}y+N\int_{2|x|<|y|\leq2}\cdots\mathrm{d}y.
\end{align*}
By \eqref{24.04.09.19.59},
$$
\int_{|y|\leq|x|/2}|y|^{-d/2}|x-y|^{-d/2}\mathrm{d}y+\int_{|x|/2<|y|\leq2|x|}|y|^{-d/2}|x-y|^{-d/2}\mathrm{d}y\leq N,
$$
where $N$ remains independent of $x$.
For $2|x| < |y| \leq 2$, it holds that
$$
\frac{|y|}{2}\leq|y|-|x|\leq|x-y|\leq |x|+|y|<\frac{3|y|}{2}.
$$
Thus,
\begin{equation}
\label{24.04.21.20.38}
\int_{2|x|<|y|\leq2}|y|^{-d/2}|x-y|^{-d/2}\mathrm{d}y\simeq \int_{2|x|<|y|\leq2}|y|^{-d}\mathrm{d}y=N\log(1/|x|).
\end{equation}
Consequently, for $|x| < 1$,
$$
(R_{\delta_0(1-\gamma)}^r\ast R_{\delta_0(1-\gamma)}^r)(x)\leq N(1+\log(1/|x|)).
$$
Conversely, employing \eqref{24.04.21.14.50} and \eqref{24.04.21.20.38}, we find
\begin{align*}
    (R_{\delta_0(1-\gamma)}^r\ast R_{\delta_0(1-\gamma)}^r)(x)&\geq \int_{|x|/2<|y|\leq 2}R_{\delta_0(1-\gamma)}(x-y)^rR_{\delta_0(1-\gamma)}(y)^r\mathrm{d}y\\
    &\geq N\int_{|x|/2<|y|\leq 2}|y|^{\alpha_{\delta_0,\gamma,r}-d}|x-y|^{\alpha_{\delta_0,\gamma,r}-d}\mathrm{d}y\simeq \log(1/|x|).
\end{align*}
The proposition is proved.
\end{proof}

We conclude this subsection by demonstrating the validity of Proposition \ref{24.04.21.21.16}.
\begin{proof}[Proof of Proposition \ref{24.04.21.21.16}]
Using H\"older's inequality, we obtain for $x\in\bR^d\setminus\{0\}$,
\begin{align*}
&(R_{\delta_0(1-\gamma)}^{r_1}\ast R_{\delta_0(1-\gamma)}^{r_1})(x)\\
&\leq \|R_{\delta_0(1-\gamma)}\|_{L_{r_1}(\bR^d)}^{r_1\left(1-\frac{r_1}{r_2}\right)}(R_{\delta_0(1-\gamma)}^{r_2}\ast R_{\delta_0(1-\gamma)}^{r_1})(x)^{\frac{r_1}{r_2}}\\
&\leq \|R_{\delta_0(1-\gamma)}\|_{L_{r_1}(\bR^d)}^{r_1\left(1-\frac{r_1}{r_2}\right)}\|R_{\delta_0(1-\gamma)}\|_{L_{r_2}(\bR^d)}^{r_1\left(1-\frac{r_1}{r_2}\right)}(R_{\delta_0(1-\gamma)}^{r_2}\ast R_{\delta_0(1-\gamma)}^{r_2})(x)^{\frac{r_1^2}{r_2^2}}.
\end{align*}
Given the finite norms described in \eqref{24.04.21.20.52},
$$
\|R_{\delta_0(1-\gamma)}\|_{L_{r_1}(\bR^d)}^{r_1\left(1-\frac{r_1}{r_2}\right)}\|R_{\delta_0(1-\gamma)}\|_{L_{r_2}(\bR^d)}^{r_1\left(1-\frac{r_1}{r_2}\right)}=:C_0<\infty.
$$
Thus, we can estimate
\begin{equation}
\label{24.04.09.12.50}
\begin{aligned}
    \nu_{\delta_0,\gamma,r_1}&:=\int_{\bR^d}(R_{\delta_0(1-\gamma)}^{r_1}\ast R_{\delta_0(1-\gamma)}^{r_1})(x)\pi(\mathrm{d}x)\\
    &\leq C_0\int_{\bR^d}(R_{\delta_0(1-\gamma)}^{r_2}\ast R_{\delta_0(1-\gamma)}^{r_2})(x)^{\frac{r_1^2}{r_2^2}}\pi(\mathrm{d}x).
\end{aligned}
\end{equation}
Given that $(R_{\delta_0(1-\gamma)}^r\ast R_{\delta_0(1-\gamma)}^r)(x)$ decays exponentially as $|x|\to\infty$ and considering  $\pi$ as a tempered measure,
\begin{equation}
\label{24.04.21.21.11}
    \int_{\bR^d}(R_{\delta_0(1-\gamma)}^{r_2}\ast R_{\delta_0(1-\gamma)}^{r_2})(x)^{\frac{r_1^2}{r_2^2}}\pi(\mathrm{d}x)<\infty \Longleftrightarrow \int_{|x|\leq 1}(R_{\delta_0(1-\gamma)}^{r_2}\ast R_{\delta_0(1-\gamma)}^{r_2})(x)^{\frac{r_1^2}{r_2^2}}\pi(\mathrm{d}x)<\infty.
\end{equation}
Employing H\"older's inequality once more, we have
\begin{equation}
\label{24.04.21.21.10}
\begin{aligned}
    &\int_{|x|\leq 1}(R_{\delta_0(1-\gamma)}^{r_2}\ast R_{\delta_0(1-\gamma)}^{r_2})(x)^{\frac{r_1^2}{r_2^2}}\pi(\mathrm{d}x)\\
    &\leq A_{\pi}^{1-\frac{r_1^2}{r_2^2}}\left(\int_{|x|\leq 1}(R_{\delta_0(1-\gamma)}^{r_2}\ast R_{\delta_0(1-\gamma)}^{r_2})(x)(1+|x|^2)^{\frac{k}{2}\left(\frac{r_2^2}{r_1^2}-1\right)}\pi(\mathrm{d}x)\right)^{\frac{r_1^2}{r_2^2}}\\
    &\leq N(A_{\pi},k,r_1,r_2)\nu_{\delta_,\gamma,r_2}^{\frac{r_1^2}{r_2^2}},
\end{aligned}
\end{equation}
where $A_{\pi}$ is the constant defined in \eqref{24.04.21.15.46}. 
Consequently,
\begin{align*}
    \text{Assumption \ref{main_ass_corr} $(\delta_0,\gamma,r_2)$ holds}&\underset{\text{Proposition \ref{24.04.21.21.05}}}{\Longleftrightarrow} \nu_{\delta_0,\gamma,r_2}<\infty \\
    &\underset{\text{\eqref{24.04.09.12.50},\eqref{24.04.21.21.11},\eqref{24.04.21.21.10}}}{\implies} \nu_{\delta_0,\gamma,r_1}<\infty\\
    &\underset{\text{Proposition \ref{24.04.21.21.05}}}{\Longleftrightarrow} \text{Assumption \ref{main_ass_corr} $(\delta_0,\gamma,r_1)$ holds}
\end{align*}
The proposition is thereby proven.
\end{proof}

\subsection{Main results}

Recall that $(\Omega, \bF, \bP)$ is a complete probability space with a filtration $\{\bF_t\}_{t\geq0}$ satisfying the usual conditions and $\cP$ is the predictable $\sigma$-field related to $\bF_t$. 
First, we introduce our solution spaces.

\begin{defn} \label{def:sto-banach}
For a bounded stopping time $\tau\leq T$, let us denote 
$$
\opar0,\tau\cbrk:=\{ (\omega,t):0<t\leq \tau(\omega) \},\quad\cpar0,\tau\cbrk:=\{ (\omega,t):0\leq t\leq \tau(\omega) \}.
$$
Define {\it{stochastic Banach spaces}} as
\begin{gather*}
\mathbb{H}_{p,q}^{\phi,\gamma}(\tau) := L_q(\opar0,\tau\cbrk, \mathcal{P}, \bP \times \mathrm{d}t ; H_{p}^{\phi,\gamma}(\bR^d)),\\
\mathbb{H}_{p,q}^{\phi,\gamma}(\tau,\ell_2) := L_q(\opar0,\tau\cbrk,\mathcal{P}, \bP \times \mathrm{d}t;H_{p}^{\phi,\gamma}(\bR^d;\ell_2)),\\
U_{p,q}^{\phi,\gamma} :=  L_q(\Omega,\bF_0, \bP ; (H_p^{\phi,\gamma}(\bR^d),H_p^{\phi,\gamma-2}(\bR^d))_{1/q,q}),
\end{gather*}
where $(X_0,X_1)_{\theta,q}$ is the real interpolation space between the interpolation couple $(X_0,X_1)$ with parameters $(\theta,q)\in(0,1)\times(1,\infty)$.
For convenience, we write $\bL_{p,q}(\tau):=\bH^{\phi,0}_{p,q}(\tau)$ and $\bL_{p,q}(\tau,\ell_2):=\bH^{\phi,0}_{p,q}(\tau,\ell_2)$.

\end{defn}

\begin{defn}[Global solution space] \label{def_of_sol_1}
Let $\tau$ be a bounded stopping time and $u \in \bH_{p,q}^{\phi,\gamma}(\tau)$.
\begin{enumerate}[(i)]
\item 
We write $u\in\cH^{\phi,\gamma}_{p,q}(\tau)$ if there exist $(u_0,f,g)\in U_{p,q}^{\phi,\gamma}\times
\bH_{p,q}^{\phi,\gamma-2}(\tau)\times\bH_{p,q}^{\phi,\gamma-1}(\tau,\ell_2)$ such that
\begin{equation*}
\mathrm{d}u = f\,\mathrm{d}t+\sum_{k=1}^{\infty} g^k\, \mathrm{d}w_t^k,\quad   t\in (0, \tau]\,; \quad u(0,\cdot) = u_0
\end{equation*}
in the sense of distributions, \textit{i.e.}, for any $\varphi\in \cS(\bR^d)$, the equality
\begin{equation} \label{def_of_sol_2}
(u(t,\cdot),\varphi) = (u_0,\varphi) + \int_0^t(f(s,\cdot),\varphi)\mathrm{d}s + \sum_{k=1}^{\infty} \int_0^t(g^k(s,\cdot),\varphi)\mathrm{d}w_s^k
\end{equation}
holds for all $t\in [0,\tau]$ almost surely.
In this case, we write
\begin{equation*}
u(0):=u_0,\quad\mathbb{D}u:= f,\quad \mathbb{S}u:=g.
\end{equation*}

\item
The norm of $\cH_{p,q}^{\phi,\gamma}(\tau)$ is defined by
\begin{equation*}
\| u \|_{\cH_{p,q}^{\phi,\gamma}(\tau)} :=  \| u \|_{\mathbb{H}_{p,q}^{\phi,\gamma}(\tau)} + \| \mathbb{D}u \|_{\mathbb{H}_{p,q}^{\phi,\gamma-2}(\tau)} + \| \mathbb{S}u \|_{\mathbb{H}_{p,q}^{\phi,\gamma-1}(\tau,\ell_2)} + \| u(0,\cdot) \|_{U_{p,q}^{\phi,\gamma}}.
\end{equation*}
\end{enumerate}
\end{defn} 

\begin{rem}
The pair $(f,g)$ in \eqref{def_of_sol_2} is uniquely determined by $u$.  
Indeed, if two pairs $(f_1,g_1)$ and $(f_2,g_2)$ satisfy \eqref{def_of_sol_2}, then for all $\varphi \in \mathcal{S}(\mathbb{R}^d)$,
\begin{equation}
\label{25.08.12.11.31}
    \int_0^t (f_2(s,\cdot) - f_1(s,\cdot), \varphi) \,\mathrm{d}s
    = \sum_{k=1}^\infty \int_0^t (g_1^k(s,\cdot) - g_2^k(s,\cdot), \varphi) \,\mathrm{d}w_s^k,
\end{equation}
for all $t \in [0,\tau]$ almost surely.  
Here, the left-hand side is a process of finite variation, while the right-hand side is a continuous local martingale.  
By \cite[Proposition~1.2, Chapter~IV]{RY1999}, both processes must vanish identically.  
Thus $(f,g)$ is uniquely determined by $u$, justifying the notation $\mathbb{D}u$ and $\mathbb{S}u$.
\end{rem}

\begin{defn}[Local solution space]
Let $\tau$ be a stopping time. 
We write $u\in\cH_{p,q,loc}^{\phi,\gamma}(\tau)$ if there exists a sequence of stopping times $\tau_n\uparrow \tau$ almost surely such that
\begin{enumerate}[(i)]
    \item each $\tau_n$ is bounded, \textit{i.e.}
    $$
    \sup_{\omega\in\Omega}\tau_n(w)<\infty,
    $$
    \item $u\in \cH_{p,q}^{\phi,\gamma}(\tau_n)$ for each $n$.
\end{enumerate} 
We say $u = v$ in $\cH_{p,q,loc}^{\phi,\gamma}(\tau)$ if there is a sequence of stopping times $\tau_n\uparrow\tau$ almost surely such that 
\begin{enumerate}[(i)]
    \item each $\tau_n$ is bounded, \textit{i.e.}
    $$
    \sup_{\omega\in\Omega}\tau_n(w)<\infty,
    $$
    \item $u = v$ in $\cH_{p,q}^{\phi,\gamma}(\tau_n)$ for each $n$.
\end{enumerate} 
For convenience, $\tau$ is omitted when $\tau = \infty$. 
\end{defn}

Here are the main results and assumptions of this article.
The proof will be given in Section \ref{24.03.21.14.39}.
\begin{assumption}[$\delta_0$]
    \label{24.03.14.11.58}
    The function $\phi:(0,\infty)\to[0,\infty)$ is a Bernstein function satisfying
\begin{equation}
\label{23.06.28.17.44}
    c_0\left(\frac{R}{r}\right)^{\delta_0}\leq\frac{\phi(R)}{\phi(r)},\quad \forall 0<r\leq R<\infty,\quad (c_0>0).
\end{equation}
\end{assumption}
This condition is a \emph{weak lower scaling condition} for the Bernstein function $\phi$ with index $\delta_0 \in (0,1]$. 
It ensures that $\phi$ grows at least polynomially at all scales and is a standard hypothesis in the study of nonlocal operators generated by subordinate Brownian motions.
\begin{example}
The following Bernstein functions satisfy Assumption \ref{24.03.14.11.58} $(\delta_0)$.  
For additional examples and details, see \cite[Chapter 16]{SSV2012}.
\begin{enumerate}[(1)]
    \item Stable subordinators: $\phi(\lambda)=\lambda^{\beta}$, $0<\beta\leq1$ $\implies$ Assumption \ref{24.03.14.11.58} $(\beta)$ holds.
    \item Sum of stable subordinators: $\phi(\lambda)=\lambda^{\beta_1}+\lambda^{\beta_2}$, $0<\beta_1,\beta_2\leq 1$ $\implies$ Assumption \ref{24.03.14.11.58} $(\min(\beta_1,\beta_2))$ holds.
    \item Stable with logarithmic correction: $\phi(\lambda)=\lambda^{\beta}(\log(1+\lambda))^{\gamma}$, $\beta\in(0,1)$, $\gamma\in(-\beta,1-\beta)$ $\implies$ Assumption  \ref{24.03.14.11.58} $(\beta)$ holds.
    \item Relativistic stable subordinators: $\phi(\lambda)=(\lambda+m^{1/\beta})^{\beta}-m$, $\beta\in(0,1)$, $m>0$ $\implies$ Assumption  \ref{24.03.14.11.58} $(\beta)$ holds.
    \item Conjugate geometric stable subordinators: $\phi(\lambda)=\frac{\lambda}{\log(1+\lambda^{\beta/2})}$, $\beta\in(0,2)$ $\implies$ Assumption  \ref{24.03.14.11.58} $(1-\frac{\beta}{2})$ holds.
\end{enumerate}
\end{example}

\begin{rem}
The condition \eqref{23.06.28.17.44} gives us to the characterization of $U_{p,q}^{\phi,\gamma}$.
By \cite[Proposition A.3]{CLSW2023trace}, we have
    $$
    (H_p^{\phi,\gamma}(\bR^d),H_{p}^{\phi,\gamma-2}(\bR^d))_{1/q,q}=B_{p,q}^{\phi,\gamma-2/q}(\bR^d),
    $$
    where $B_{p,q}^{\phi,\gamma-2/q}(\bR^d)$ is a set of tempered distrubution $f$ satisfying
    $$
    \|f\|_{B_{p,q}^{\phi,\gamma-2/q}(\bR^d)}:=\|S_0f\|_{L_p(\bR^d)}+\left(\sum_{j=1}^{\infty}\phi(2^{2j})^{(q\gamma-2)/2}\|\Delta_jf\|_{L_p(\bR^d)}^q\right)^{1/q}<\infty.
    $$
    Here $\Delta_j$ is the $j$-th Littlewood-Paley projection operator defined by
    $$
    \Delta_jf(x):=(\Psi_j\ast f)(x),\quad \Psi_j(x):=2^{jd}\Psi(2^jx),
    $$
    and $S_0:=\sum_{j\leq0}\Delta_j$, where $\Psi$ is a function in $\cS(\bR^d)$ which satisfies $supp(\cF[\Psi])\subseteq \{\xi\in\bR^d:2^{-1}\leq|\xi|\leq 2\}$ and
    $$
    \sum_{j\in\bZ}\Psi(2^{-j}\xi)=1,\quad \forall \xi\in\bR^d.
    $$
Here, $\mathcal{F}[\Psi]$ is the Fourier transform of $\Psi$.    If $\gamma>2/q$, then we also have
    $$
    \|f\|_{B_{p,q}^{\phi,\gamma-2/q}(\bR^d)}\simeq \|f\|_{L_p(\bR^d)}+\left(\int_{\bR^d}\phi(|h|^{-1})^{q\gamma-2}\|\mathcal{D}_h^Lf\|_{L_p(\bR^d)}^q\frac{\mathrm{d}h}{|h|^d}\right)^{1/q},
    $$
    where $\mathcal{D}_hf(x):=f(x+h)-f(x)$, $\mathcal{D}_h^nf(x):=\mathcal{D}_h(\mathcal{D}_h^{n-1}f)(x)$, and $L$ is the integer greater than $\gamma-2/q$.
\end{rem}

\begin{assumption}
\label{ass_nl_term}
    Let $(F,B,\varphi)=(F(\omega,t,x,u),B(\omega,t,x,u),\varphi(\omega,t,x,u))$ be $\cP\times\cB(\bR^d)\times\cB(\bR)$-measurable functions satisfying the followings:
    \begin{enumerate}[(i)]
        \item For $m\in(0,\infty)$,
        $$
        \sup_{|u|\leq 2m}\left(\left|\frac{\partial F}{\partial u}\right|+\left|\frac{\partial B}{\partial u}\right|+\left|\frac{\partial\varphi}{\partial u}\right|\right)\leq C_m,\quad \forall (\omega,t,x)\in\Omega\times[0,\infty)\times\bR^d.
        $$
        \item There exist $c_{\mathrm{s.d.}},c_{\mathrm{b.}},c_{\mathrm{s.m.}}>0$ and $\lambda_{\mathrm{s.d.}},\lambda_{\mathrm{b.}},\lambda_{\mathrm{s.m.}}\geq0$ such that
        $$
        -F(u)\geq c_{\mathrm{s.d.}}|u|^{1+\lambda_{\mathrm{s.d.}}},\quad|B(u)|\leq c_{\mathrm{b.}}^1|u|+c_{\mathrm{b.}}^2|u|^{1+\lambda_{\mathrm{b.}}},\quad |\varphi(u)|\leq c_{\mathrm{s.m.}}^1|u|+c_{\mathrm{s.m.}}^2|u|^{1+\lambda_{\mathrm{s.m.}}},
        $$
        for all $(\omega,t,x)\in\Omega\times[0,\infty)\times\bR^d$.
    \end{enumerate}
\end{assumption}
\begin{assumption}[$\delta_0,\gamma$]
\label{ass_coeff}
    Let $(\vec{b},\zeta,\xi)=(\vec{b}(\omega,t,x),\zeta(\omega,t,x),\xi(\omega,t,x))$ be $\cP\times\cB(\bR^d)$-measurable functions satisfying
    \begin{enumerate}[(i)]
        \item $\zeta$ is nonnegative and bounded away from zero.
        \item $\vec{b}=(b^1,\cdots,b^d)\in\bR^d$ and $b^i$ is independent of $x^i$.
        \item $\vec{b}\equiv0$ if $\delta_0(2-\gamma)\leq1$.
        \item for all $(\omega,t)\in \Omega\times[0,\infty)$, 
        $$        \|\vec{b}(\omega,t,\cdot)\|_{C^2(\bR^d;\mathbb{R}^d)}+\||\zeta(\omega,t,\cdot)|+|\xi(\omega,t,\cdot)|\|_{L_{\infty}(\bR^d)}\leq K.
        $$
    \end{enumerate}
\end{assumption}

\begin{thm}
\label{main}
Let $p,q\in(2,\infty)$, $\delta_0\in(0,1]$, $\gamma\in(0,1)$ and 
\begin{equation*}
1\leq r < \min\left(\frac{p}{2}, \frac{d}{d-\delta_0(1-\gamma)}\right).
\end{equation*}
Suppose that Assumptions \ref{main_ass_corr} ($\delta_0,\gamma,r$), \ref{24.03.14.11.58} ($\delta_0$), \ref{ass_nl_term} and \ref{ass_coeff}  ($\delta_0,\gamma$) hold.
Suppose also that
\begin{align}
        0<&\lambda_{\mathrm{b.}}+\frac{d+1-\delta_0(2-\gamma)}{d}-\frac{1}{p}-\frac{\lambda_{\mathrm{s.d.}}}{p\vee q}<1, \label{24.04.24.19.45}\\
        0<&\lambda_{\mathrm{s.m.}}+\frac{1}{2r}+\frac{1}{2}-\frac{1}{p}-\frac{\lambda_{\mathrm{s.d.}}}{p\vee q}<1, \label{24.04.24.19.46}
\end{align}
and
\begin{align}
\label{24.04.06.19.26}
\frac{2}{q}+\frac{d}{\delta_0 p}<\gamma.
\end{align}
\begin{enumerate}[(i)]
    \item Then, for any nonnegative $u_0\in U_{p,q}^{\phi,\gamma}\cap L_1(\Omega\times\bR^d)$, there exists a unique nonnegative solution $u\in\cH_{p,q,loc}^{\phi,\gamma}$ to
    \begin{equation}
\label{23.06.28.17.21}
    \partial_tu=\phi(\Delta)u+\zeta F(u)+\vec{b}\cdot\nabla_{x}(B(u))+\xi\varphi(u)\dot{W};\quad u(0)=u_0.
\end{equation}
    \item For any $\alpha,\beta>0$ satisfying
\begin{equation}
\label{25.08.12.17.42}
\frac{1}{q}\leq \alpha<\beta\leq\frac{\gamma}{2}-\frac{d}{2\delta_0 p},
\end{equation}
the solution $u$ in $(i)$ satisfies
$$
    |u|_{C_{t,x}^{\alpha-\frac{1}{q},\phi_{\gamma-2\beta}-\frac{d}{p}}([0,\tau]\times\bR^d)}<\infty
    $$
    for all bounded stopping time $\tau$ (a.s.).
    Here,
    $$
    |u|_{C_{t,x}^{\alpha-\frac{1}{q},\phi_{\gamma-2\beta}-\frac{d}{p}}([0,\tau]\times\bR^d)}:=|u|_{C([0,\tau]\times\bR^d)}+[u]_{C_{t,x}^{\alpha-\frac{1}{q},\phi_{\gamma-2\beta}-\frac{d}{p}}([0,\tau]\times\bR^d)},
    $$
    where 
    $$
    [u]_{C_{t,x}^{\alpha-\frac{1}{q},\phi_{\gamma-2\beta}-\frac{d}{p}}([0,\tau]\times\bR^d)}:=\sup_{(t,x),(s,y)\in[0,\tau]\times\mathbb{R}^d}\frac{|u(t,x)-u(s,y)|}{|t-s|^{\alpha-\frac{1}{q}}+\phi(|x-y|^{-2})^{-\frac{\gamma}{2}+\beta}|x-y|^{-\frac{d}{p}}}.
    $$
    If $\alpha=1/q$, then
    $$
    |u|_{C_{t,x}^{\alpha-\frac{1}{q},\phi_{\gamma-2\beta}-\frac{d}{p}}([0,\tau]\times\bR^d)}:=|u|_{C([0,\tau]\times\bR^d)}+\sup_{t\in[0,\tau]}\sup_{x,y\in\mathbb{R}^d}\frac{|u(t,x)-u(t,y)|}{\phi(|x-y|^{-2})^{-\frac{\gamma}{2}+\beta}|x-y|^{-\frac{d}{p}}}.
    $$
\end{enumerate} 
\end{thm}

\begin{rem}
    The solution in Theorem \ref{main} means that there exists a sequence of stopping times $\tau_n\uparrow\infty$ almost surely,
    $$
    \sup_{\omega\in\Omega}\tau_n(\omega)<\infty,\quad u\in\cH_{p,q}^{\phi,\gamma}(\tau_n),
    $$
    and for any $\eta\in \cS(\bR^d)$, the following equality holds for all $t\in[0,\tau_n]$ almost surely;
    \begin{align*}
        (u(t,\cdot),\eta)&=(u_0,\eta)+\int_0^t(u(s,\cdot),\phi(\Delta)\eta)+(\zeta(s,\cdot) F(s,u(s,\cdot)),\eta)\mathrm{d}s\\
        &\quad+\int_0^t(B(s,u(s,\cdot)),\vec{b}(s,\cdot)\cdot\nabla_{x}\eta)\mathrm{d}s+\sum_{k=1}^{\infty}\int_0^t(\xi(s,\cdot)\varphi(s,u(s,\cdot))\pi\ast e_k,\eta)\mathrm{d}w_s^k.
    \end{align*}
\end{rem}

\begin{rem}
\label{25.08.13.15.15}
    We make several remarks on Theorem \ref{main}.
The parameters appearing in \eqref{24.04.24.19.45}--\eqref{24.04.06.19.26} have the following interpretation:
\begin{itemize}
    \item $p,q$: integrability exponents for the solution $u$ (space and time, respectively);
    \item $\gamma$: spatial regularity index of $u$;
    \item $\delta_0$: weak lower scaling exponent of $\phi$;
    \item $r$: spatial smoothness parameter of the noise $\dot{W}$ (via its correlation class);
    \item $\lambda_{\mathrm{s.d.}},\, \lambda_{\mathrm{b.}},\, \lambda_{\mathrm{s.m.}}$: nonlinearity exponents of the strong dissipative term, the generalized Burgers-type term, and the superlinear multiplicative noise term, respectively.
\end{itemize}
    
    \begin{enumerate}[(i)]
        \item \emph{(Effect of $\phi$ on $\lambda_{\mathrm{b.}}$ and $\lambda_{\mathrm{s.m.}}$)}
        By the Mikhlin multiplier theorem and Assumption~\ref{24.03.14.11.58} $(\delta_0)$, the Fourier multiplier
$$
m(\xi):=\frac{(1+|\xi|^2)^{\delta_0/2}}{(1+\phi(|\xi|^2))^{1/2}}
$$
is bounded on $L_p(\mathbb{R}^d)$ for $1<p<\infty$.
Hence there exists $C=C(d,p,\delta_0,\phi)$ such that
$$
\|u\|_{L_p(\mathbb{R}^d)}+\|\Delta^{\delta_0/2}u\|_{L_p(\mathbb{R}^d)} \leq C\,(\|u\|_{L_p(\mathbb{R}^d)}+\|\phi(\Delta)^{1/2}u\|_{L_p(\mathbb{R}^d)}).
$$
In particular, $\delta_0$ provides a \emph{baseline spatial regularity} for every $u\in H_p^{\phi,1}(\mathbb{R}^d)$. 
Thus, as $\delta_0\uparrow 1$ (\textit{i.e.}, as $\phi(\Delta)$ becomes more regularizing), the operator $\phi(\Delta)$ enforces stronger smoothing.

Consequently, the constraint \eqref{24.04.24.19.45} becomes less restrictive for larger $\delta_0$, thereby enlarging the admissible range of the nonlinear exponents—most notably $\lambda_{\mathrm{b.}}$.

Similarly, for the multiplicative noise constraint \eqref{24.04.24.19.46}, note that the admissible window for $r$ contains
$$
1\leq r \;<\; \frac{d}{\,d-\delta_0(1-\gamma)\,}.
$$
Since the upper bound on $r$ increases with $\delta_0$, choosing $r$ close to this threshold makes \eqref{24.04.24.19.46} less restrictive as $\delta_0$ grows, which in turn enlarges the admissible range of $\lambda_{\mathrm{s.m.}}$.

\item \emph{(Effect of $r$ on $\lambda_{\mathrm{s.m.}}$)} By the multiplicative-noise constraint \eqref{24.04.24.19.46}, choosing smoother noise (\textit{i.e.}, larger $r$) makes the constraint less restrictive, thereby yielding a broader admissible range for $\lambda_{\mathrm{s.m.}}$ (with all other parameters fixed).

\item The condition \eqref{24.04.06.19.26} guarantees the existence of parameters $\alpha,\beta$ such that \eqref{25.08.12.17.42} holds.
    \end{enumerate}
\end{rem}

We conclude this subsection with a comparison of Theorem \ref{main} to existing results. 
In particular, Theorem \ref{main} provides a more general framework than the following:
$$
\bullet\text{\cite[Theorem 3.10]{CH2021}},\qquad
    \bullet\text{\cite[Theorem 3.2]{Han2021}},\qquad
    \bullet\text{\cite[Theorem 3.5]{HY2024}}.
$$
For a concrete side-by-side comparison, we adopt the parameter choice
\begin{equation}
\label{25.08.12.19.19}
d=2,\qquad \phi(\lambda)=\lambda^{\delta_0},\qquad  p=q=16.
\end{equation}

\emph{1. Comparison with \cite[Theorem 3.10]{CH2021}.}

\textbf{Case 1 (fixed $\gamma$ and $\lambda_{\mathrm{s.m.}}$).}
We adopt the following parameter configuration:
\begin{equation}
\label{25.08.12.19.43}
\delta_0=1,\qquad \gamma=\frac{1}{4},\qquad \lambda_{\mathrm{s.m.}}=\frac{1}{8},\qquad 
\lambda_{\mathrm{s.d.}}=0,\qquad \lambda_{\mathrm{b.}}=0,\qquad r=\frac{4}{3}.
\end{equation}
\begin{itemize}
    \item \textbf{Assumption \ref{main_ass_corr} $(\delta_0,\gamma,r)$.}
    Here
    $$
        \alpha_{\delta_0,\gamma,r}=d-r\big(d-\delta_0(1-\gamma)\big)
        =2-\frac{4}{3}\Big(2-\frac{3}{4}\Big)=\frac{1}{3}<\frac{d}{2}=1,
    $$
    hence the correlation $\pi$ must satisfy
    \begin{equation}
    \label{24.05.12.16.35}
        \int_{|x|<1}|x|^{-4/3}\,\pi(\mathrm{d}x)<\infty.
    \end{equation}
    \item \textbf{Assumption in \cite[Theorem 3.10]{CH2021}.}
    To apply \cite[Theorem 3.10]{CH2021}, one requires
    \begin{equation}
    \label{25.08.12.19.33}
        \int_{|x|<1}|x|^{-5/3}\,\pi(\mathrm{d}x)<\infty.
    \end{equation}
\end{itemize}

Under \eqref{25.08.12.19.19}, \eqref{25.08.12.19.43}, and \eqref{24.05.12.16.35}, all hypotheses of Theorem~\ref{main} are satisfied. 
Likewise, \eqref{25.08.12.19.19}, \eqref{25.08.12.19.43}, and \eqref{25.08.12.19.33} satisfy the assumptions of \cite[Theorem 3.10]{CH2021}.

If, moreover, $\pi(\mathrm{d}x)=|x|^{-\alpha}\,\mathrm{d}x$, then near the origin
$$
\int_{|x|<1}|x|^{-\beta}\,\pi(\mathrm{d}x)<\infty
\quad\Longleftrightarrow\quad 
\beta+\alpha<d\ (=2).
$$
Consequently,
$$
\text{Theorem \ref{main}} \;\Rightarrow\; 0<\alpha<\frac{2}{3},
\qquad
\text{\cite[Theorem 3.10]{CH2021}} \;\Rightarrow\; 0<\alpha<\frac{1}{3}.
$$
Thus, Theorem \ref{main} applies to a strictly broader class of noises $\dot W$ than \cite[Theorem 3.10]{CH2021}.

\textbf{Case 2 (fixed $\pi$ and $\lambda_{\mathrm{s.m.}}$).}
We adopt the following parameter configuration:
\begin{equation}
\label{25.08.12.20.12}
\delta_0=1,\qquad \lambda_{\mathrm{s.m.}}=\frac{1}{8},\qquad 
\lambda_{\mathrm{s.d.}}=0,\qquad \lambda_{\mathrm{b.}}=0,\qquad r=1.
\end{equation}
Suppose $\pi(\mathrm{d}x)=|x|^{-1/4}\,\mathrm{d}x$ and $\gamma\in(0,1)$.
\begin{itemize}
    \item \textbf{Condition \eqref{24.04.06.19.26}.}
    Here
    $$
        \frac{2}{q}+\frac{d}{\delta_0 p}
        = \frac{2}{16}+\frac{2}{16}
        = \frac{1}{4} < \gamma .
    $$
    \item \textbf{Assumption in \cite[Theorem 3.10]{CH2021}.}
    One requires
    $$
        \int_{|x|<1} |x|^{-\frac{4}{3}(1+\gamma)}\,\pi(\mathrm{d}x)<\infty 
        \quad\Longleftrightarrow\quad 
        \frac{4}{3}(1+\gamma)+\frac{1}{4} < 2
        \ \Longleftrightarrow\ 
        \gamma < \frac{5}{16},
    $$
    and
    $$
        \frac{d+2}{\gamma} = \frac{4}{\gamma} < p = 16
        \quad\Longleftrightarrow\quad \gamma > \frac{1}{4}.
    $$
\end{itemize}
Therefore, the admissible ranges of $\gamma$ are
$$
\text{Theorem \ref{main}} \;\Rightarrow\; \frac{1}{4}<\gamma<1,
\qquad
\text{\cite[Theorem 3.10]{CH2021}} \;\Rightarrow\; \frac{1}{4}<\gamma<\frac{5}{16}.
$$
Thus, Theorem \ref{main} applies to a strictly broader range of regularity indices $\gamma$ than \cite[Theorem 3.10]{CH2021}.

\medskip
        
 \emph{2. Comparison with \cite[Theorem 3.5]{HY2024}.}

Let us assume \eqref{25.08.12.19.19} and
$$
\delta_0=1,\qquad \gamma=\frac{1}{4},\qquad  
\lambda_{\mathrm{s.d.}}=1,\qquad \lambda_{\mathrm{b.}}=0,\qquad r=\frac{4}{3}.
$$
If the correlation $\pi(\mathrm{d}x):=|x|^{-1/3}\mathrm{d}x$, then $\pi$ satisfies \eqref{24.05.12.16.35} and \cite[Assumption 2.3]{HY2024}, simultaneously.
Since $\lambda_{\mathrm{s.d.}}=1$, the range of $\lambda_{\mathrm{s.m.}}$ allowed by \cite[Theorem~3.5]{HY2024} is
$$
\lambda_{\mathrm{s.m.}} < \frac{1}{8},
$$
whereas, under \eqref{24.04.24.19.46}, it extends to
$$
\lambda_{\mathrm{s.m.}} < \frac{1}{8} + \frac{2}{p}.
$$

\medskip

\emph{3. Comparison with \cite[Theorem 3.2]{Han2021}.}

Assume
\begin{equation}
\label{25.08.13.11.50}
d=1,\qquad \phi(\lambda)=\lambda^{\delta_0},
\end{equation}
and take
$$
p=q=16,\qquad \delta_0=\frac{1}{2},\qquad \gamma=\frac{1}{2},\qquad  
\lambda_{\mathrm{s.m.}}=\frac{1}{32},\qquad \lambda_{\mathrm{b.}}=\lambda_{\mathrm{s.d.}}=0,\qquad r=1.
$$
Let $\pi(\mathrm{d}x)=\varepsilon_0(\mathrm{d}x)$ denote the centered Dirac measure.  
Then all assumptions of Theorem~\ref{main} are satisfied, so the theorem applies in this setting.  
By contrast, \cite[Theorem~3.2]{Han2021} excludes the case $\delta_0=\frac{1}{2}$.

Next, still under \eqref{25.08.13.11.50} with $\pi(\mathrm{d}x)=\varepsilon_0(\mathrm{d}x)$, consider
$$
p=q=64,\qquad \delta_0=\frac{2}{3},\qquad \gamma=\frac{1}{12},\qquad 
\lambda_{\mathrm{b.}}=\lambda_{\mathrm{s.d.}}=0,\qquad r=\frac{4}{3}.
$$
In this case, the admissible range for $\lambda_{\mathrm{s.m.}}$ under Theorem \ref{main} is
$$
0<\lambda_{\mathrm{s.m.}}<\frac{9}{64},
$$
whereas the corresponding range under \cite[Theorem 3.2]{Han2021} is
$$
0<\lambda_{\mathrm{s.m.}}<\frac{1}{12}.
$$
Since $\frac{9}{64}>\frac{1}{12}$, Theorem \ref{main} permits strictly larger superlinear growth in the multiplicative noise term than \cite[Theorem 3.2]{Han2021}.

\subsection{Strategy and Organization}

We begin by outlining the proof concepts for Theorem \ref{main}.
Our initial step involves constructing a local solution $ u_m $ as detailed in Theorem \ref{local}-$(i)$, applying the established solvability of semilinear stochastic integro-differential equations in mixed norm $ L_q(L_p) $-spaces (Theorem \ref{23.07.04.14.29}).
This solvability derives from the $ H^{\infty} $-calculus discussed in Section \ref{24.05.06.12.50}.
Following this, we establish the nonnegativity of the local solution $ u_m $ using a modified version of Krylov's maximum principle for stochastic integro-differential equations (Theorem \ref{local}-($ii$)).
The continuity of $ u_m $ is then addressed in Theorem \ref{local}-($iii$).
By leveraging the established nonnegativity and continuity, we demonstrate that the explosion does not occur for the local solution $ u_m $ (Theorem \ref{local}-($iv$)).
Using the non-explosivity of the local solution, we proceed to derive the main results as formulated in Theorem \ref{main}.

The organization of this paper is structured as follows:
\begin{itemize}
    \item In Section \ref{24.05.06.13.03}, we introduce the necessary preliminaries required for the proof of Theorem \ref{main}. This includes an exploration of function space properties, semigroups and generators associated with $ \phi(\Delta) $, functional calculus, and the solvability of semilinear stochastic integro-differential equations.
    \item In Section \ref{24.03.21.14.39}, we provide the detailed proof of Theorem \ref{main}.
\end{itemize}

\mysection{Preliminaries}
\label{24.05.06.13.03}
We emphasize that there are no assumptions on the Bernstein function $\phi$ in this section.

\subsection{Properties of function spaces}
We explore several fundamental properties of the function spaces delineated in Definitions \ref{def:sto-banach} and \ref{def_of_sol_1}.
\begin{thm}
\label{23.06.25.12.37}
Let $p \in [2, \infty)$, $q \in (2, \infty)$, and $\gamma \in \mathbb{R}$.
Then, the following statements hold:

\begin{enumerate}[(i)]
    \item For every bounded stopping time $\tau\leq T$, the space $\cH_{p,q}^{\phi,\gamma}(\tau)$ is a Banach space with the norm $\|\cdot\|_{\cH_{p,q}^{\phi,\gamma}(\tau)}$.
    \item For every bounded stopping time $\tau \leq T$,
\begin{equation} 
\label{stochastic_gronwall}
\|u\|_{\mathbb{H}^{\phi,\gamma}_{p,q}(t \wedge \tau)}^q \leq N(d,p,T)\int_0^{t\wedge\tau} \|u\|_{\mathcal{H}_{p,q}^{\phi,\gamma+1}(s)}^q\,\mathrm{d}s.
\end{equation}

\item Let $\boldsymbol{\mathrm{H}}_{c,0}^{\infty}$ denote the set of functions
$$
g(\omega,x) := \sum_{i=1}^{n} 1_{A_i}(\omega)g_i(x),
$$
where $A_i \in \mathbb{F}_0$ are pairwise disjoint, and each $g_i \in \mathcal{S}(\mathbb{R}^d)$. 
Then, $\boldsymbol{\mathrm{H}}_{c,0}^{\infty}$ is dense in $U_{p,q}^{\phi,\gamma}$.

\item Let $\boldsymbol{\mathrm{H}}_{c}^{\infty}$ denote the set of functions
$$
g(\omega,t,x) := \sum_{i=1}^{n} 1_{(\tau_{i-1},\tau_i]}(\omega,t)g_i(x),
$$
where $\tau_i$ are bounded stopping times such that $\tau_{i-1} \leq \tau_i$ and  $g_i \in \mathcal{S}(\mathbb{R}^d)$. Then, $\boldsymbol{\mathrm{H}}_{c}^{\infty}$ is dense in $\mathbb{H}_{p,q}^{\phi,\gamma}$.

\item Let $\boldsymbol{\mathrm{H}}_{c}^{\infty}(\ell_2)$ the set of functions $g = (g^1, g^2, \cdots)$,
$$
g^k(\omega,t,x) := \begin{cases}
\sum_{i=1}^{n} 1_{(\tau_{i-1}^n,\tau_i^n]}(\omega,t)g_i^{k}(x), & \text{if } k \leq n,\\
0, & \text{if } k > n,
\end{cases}
$$
where $\tau_i^n$ are bounded stopping times such that $\tau_{i-1}^n \leq \tau_i^n$ and  $g_i^{k} \in \mathcal{S}(\mathbb{R}^d)$.
Then $\boldsymbol{\mathrm{H}}_{c}^{\infty}(\ell_2)$ is dense in $\mathbb{H}_{p,q}^{\phi,\gamma}(\ell_2)$.
\end{enumerate}
\end{thm}
\begin{proof}
    Assertions $(i)$ and $(ii)$ follow directly as corollaries of Theorem \ref{prop}, as detailed in \cite[Theorem 6.4]{KKK2013}.
    For assertions $(iii)$ through $(v)$, due to similarity, we demonstrate the proof explicitly for $(iii)$ only.
   Considering the isometry 
   $$
   (1-\phi(\Delta))^{\gamma/2}: U_{p,q}^{\phi,\gamma+2} \to U_{p,q}^{\phi,2} \subseteq L_q(\Omega, \mathbb{F}_0; L_p(\mathbb{R}^d)),
   $$
   and acknowledging that $\mathcal{S}(\mathbb{R}^d)$ is dense in the interpolation space $(H_p^{\phi,2}(\mathbb{R}^d), L_p(\mathbb{R}^d))_{1/q,q}$, it is sufficient to establish that $\boldsymbol{\mathrm{H}}_{c,0}^{\infty}(\mathbb{R}^d)$ is dense in $L_q(\Omega, \mathbb{F}_0, \mathbb{P}; L_p(\mathbb{R}^d))$.

    Assuming the contrary, by the Hahn-Banach theorem (\textit{e.g.} \cite[Theorem 5.19]{rudin2006real}), there exists a nonzero element $h \in L_{q'}(\Omega, \mathbb{F}_0, \mathbb{P}; L_{p'}(\mathbb{R}^d))$ for which
    $$
    \int_{\Omega}\int_{\bR^d}h(\omega,x)g(\omega,x)\mathrm{d}x\bP(\mathrm{d}\omega)=0,\quad \forall g\in \boldsymbol{\mathrm{H}}_{c,0}^{\infty}(\bR^d).
    $$
    Here $q':=q/(q-1)$ and $p':=p/(p-1)$.
    Specifically,
    $$
    \int_{\Omega}1_A(\omega)\int_{\bR^d}h(\omega,x)g(x)\mathrm{d}x\bP(\mathrm{d}\omega)=0,\quad \forall g\in \cS(\bR^d),\,A\in\mathbb{F}_0.
    $$
    Given that $\int_{\mathbb{R}^d} h(\omega, x) g(x) \mathrm{d}x$ yields an $\mathbb{F}_0$-measurable function, it implies that
    $$
    \int_{\bR^d}h(\omega,x)g(x)\mathrm{d}x=0,\quad\forall g\in \cS(\bR^d)
    $$
    almost surely.
    Consequently, this infers that $h(\omega, x) = 0$ for almost every $(\omega, x)$, contradicting the assumption that $h \neq 0$.
    The proof is completed.
\end{proof}

\subsection{Semigroups and generators of Subordinate Brownian Motions}
\label{23.09.18.16.30}

This subsection introduces the concept of Subordinate Brownian Motions (SBMs), and discusses some properties of their semigroups and generators.
We start with the definition of subordinators.
\begin{defn}
A real-valued Lévy process $S = (S_t)_{t \ge 0}$, defined on a probability space $(\Omega, \mathcal{F}, \mathbb{P})$, is termed a \textit{subordinator} if
$$
\bP(S_t\geq0,\,\forall t\geq0)=1.
$$
\end{defn}
It is widely recognized that there is a one-to-one correspondence between Bernstein functions and subordinators.
Thus, given either a subordinator $S=(S_t)_{t\geq0}$ or a Bernstein function $\phi$, there also exists the other object such that the following holds:
\begin{equation}
\label{sub_laplace}
\left( \int_{\Omega} \mathrm{e}^{-\lambda S_t(\omega)}  \mathbb{P}(\mathrm{d}\omega) =: \right) \mathbb{E}[\mathrm{e}^{-\lambda S_t}] =\mathrm{e}^{-t\phi(\lambda)},\quad \forall (t,\lambda)\in[0,\infty)\times\mathbb{R}_+.
\end{equation}
A more detailed explanation of the relationship between Bernstein functions and SBMs is found in \textit{e.g.} \cite{CLSW2023trace,SSV2012}.
We say that $S$ is a subordinator with Laplace exponent $\phi$ if \eqref{sub_laplace} is satisfied.

We now present the concept of SBMs.

\begin{defn}
The process $X=(\mathfrak{B}_{S_t})_{t\geq0}$ is termed a \emph{Subordinate Brownian Motion} (SBM) with the subordinator $S=(S_t)_{t\geq0}$ if $\mathfrak{B}=(\mathfrak{B}_t)_{t\geq0}$ is a $d$-dimensional Brownian motion that is independent of the subordinator $S$.
\end{defn}

\emph{Characteristic exponent of SBM.} It is well known (\textit{e.g.} \cite[Theorem 3.2]{SSV2012}) that Bernstein function $\phi$ has a unique representation
$$
\phi(\lambda)=b_{\phi}\lambda+\int_{0}^{\infty}\left(1-\mathrm{e}^{-\lambda t}\right)\,m_{\phi}(\mathrm{d}t),
$$
where $b_{\phi}\geq0$ and
$m_{\phi}$ is a nonnegative measure on $\bR_+$ satisfying
\begin{align}\label{2211041216}
\int_0^{\infty}\left(1\wedge t\right)m_{\phi}(\mathrm{d}t)<\infty\,.
\end{align}
We say that $b_{\phi}$ is the \emph{drift} and the nonnegative measure $m_{\phi}$ satisfying \eqref{2211041216} is the \emph{L\'evy measure} of $\phi$.
According to \cite[Theorem 30.1]{sato1999levy}, the SBM $X$ is a $\mathbb{R}^d$-valued L\'evy process, and its L\'evy triplet $(a,D,\nu)$ is as follows:
\begin{equation*}
\begin{gathered}
a=b_{\phi}\cdot 0+\int_{(0,\infty)}\int_{\mathbb{R}^d}xp(t,|x|)\mathrm{d}x\,m_{\phi}(\mathrm{d}t)=0,\\
D=b_{\phi}I_{d\times d},\\
\nu(B)=b_{\phi}\cdot0+\int_{(0,\infty)}\int_{B}p(t,|x|)\mathrm{d}x\,m_{\phi}(\mathrm{d}t)=\int_Bj(|x|)\mathrm{d}x,
\end{gathered}
\end{equation*}
where $I_{d\times d}$ is the $d\times d$ identity matrix,
\begin{equation}
    \label{eqn 03.30.10:15}
    p(t,r):=\frac{1}{(2\pi t)^{d/2}}\mathrm{e}^{-\frac{r^2}{2t}} ,\quad j(r):=\int_{(0,\infty)}p(t,r)m_{\phi}(\mathrm{d}t),
\end{equation}
$b_{\phi}$ and $m_{\phi}$ are the drift and the L\'evy measure of $\phi$, respectively.
By making use of \cite[Theorem 8.1]{sato1999levy}, we get
\begin{equation}\label{eqn 02.01.13:39}
\mathbb{E}[\mathrm{e}^{\mathrm{i}X_t\cdot\xi}] = \exp\left(-tb_{\phi}|\xi|^2+t\int_{\mathbb{R}^d}(\mathrm{e}^{\mathrm{i}\xi\cdot y}-1-\mathrm{i}y\cdot\xi1_{|y|\leq1})j(|y|)\mathrm{d}y\right).
\end{equation}
Applying Fubini's theorem, we can simplify the above as
\begin{equation}
\label{eqn 03.30.10:39}
\begin{aligned}
&\int_{\mathbb{R}^d}(\mathrm{e}^{\mathrm{i}\xi\cdot y}-1-\mathrm{i}y\cdot\xi1_{|y|\leq1})j(|y|)\mathrm{d}y\\
&=\int_{\mathbb{R}^d}(\cos(\xi\cdot y)-1)j(|y|)\mathrm{d}y\\
&=\int_{(0,\infty)}\int_{\mathbb{R}^d}(\mathrm{e}^{\mathrm{i}\xi\cdot y}-1-\mathrm{i}y\cdot\xi1_{|y|\leq1})p(s,|y|)\mathrm{d}y\,m_{\phi}(\mathrm{d}s)\\
&=\int_{(0,\infty)}(\mathrm{e}^{-s|\xi|^2}-1)m_{\phi}(\mathrm{d}s),
\end{aligned}
\end{equation}
which conclusively leads to:
\begin{equation}
\label{23.04.25.17.19}
\mathbb{E}[\mathrm{e}^{\mathrm{i}X_t\cdot\xi}]=\mathrm{e}^{-t\phi(|\xi|^2)}.
\end{equation}
We say that $X$ is a SBM with characteristic exponent $\phi$ if \eqref{23.04.25.17.19} is satisfied.

\emph{Semigroup and Infinitesimal generator of SBM.} Let $X$ be a SBM with characteristic exponent $\phi$,
$$
T_tf(x):=\mathbb{E}[f(x+X_t)],
$$
and $\cL_X$ be the infinitesimal generator of $T_t$:
$$
\cL_Xf(x):=\lim_{t \downarrow 0} \frac{ T_tf(x) -f(x)}{t}.
$$
According to \cite[Theorem 31.5]{sato1999levy}, if $X$ is a SBM with characteristic exponent $\phi$, then $\cL_X$ is equal to $\phi(\Delta)$ and
has the following integro-differential operator representation:
\begin{align*}
\cL_Xf(x)=\phi(\Delta)f(x) = b_{\phi}\Delta f + \int_{\mathbb{R}^d}\left(f(x+y)-f(x)-\nabla f(x)\cdot y1_{|y|\leq1}\right)J(y)\mathrm{d}y.
\end{align*}

We conclude this subsection by outlining some properties of $T_t$ and $\phi(\Delta)$, which will be pertinent for the functional calculus discussions in the subsection.
\begin{defn}
A family $L=\{L(t)\}_{t\geq0}$ of bounded linear operators acting on a Banach space $V$ is called a $C_0$-semigroup if the following three properties are satisfied:
\begin{enumerate}[(i)]
    \item $L(0)=I$,
    \item $L(t+s)=L(t)L(s)$ for all $t,s\geq0$,
    \item $\lim_{t\downarrow0}\|L(t)x-x\|_{V}=0$ for all $x\in V$.
\end{enumerate}
\end{defn}
\begin{lem}
\label{23.09.19.12.43}
Let $p\in(1,\infty)$.
Then, the following properties hold:
\begin{enumerate}[(i)]
    \item for $f\in C_c^{\infty}(\bR^d)$, $u_1(t,x):=T_tf(x)$ is a solution to
    $$
    u_1(t,x)=f(x)+\int_0^t\phi(\Delta)u_1(s,x)\mathrm{d}s.
    $$
    \item for $h\in C_c^{\infty}([0,\infty)\times\bR^d)$, $u_2(t,x):=\int_0^tT_{t-s}h(s,x)\mathrm{d}s$ is a solution to
    $$
    u_2(t,x)=\int_0^t(\phi(\Delta)u_2(s,x)+h(s,x))\mathrm{d}s.
    $$
    \item the semigroup $(T_t)_{t\geq0}$ is a contraction $C_0$-semigroup on $L_p(\bR^d)$.
    \item the operator $\phi(\Delta)$ acts as the generator of $(T_t)$ on $L_p(\mathbb{R}^d)$.
\end{enumerate}
\end{lem}
\begin{proof}
We start by invoking Minkowski's inequality to establish that
\begin{equation}
\label{24.03.17.12.31}
\|T_t\|_{L_p(\mathbb{R}^d) \to L_p(\mathbb{R}^d)} \leq 1.
\end{equation}

$(i)$  For any function $f \in C_c^{\infty}(\mathbb{R}^d)$,
\begin{equation}
\label{23.04.27.15.01}
\begin{aligned}
    u_1(t,x)-f(x)&=\cF^{-1}[(\mathrm{e}^{-t\phi(|\cdot|^2)}-1)\cF[f]](x)\\
    &=-\cF^{-1}\left[\left(\int_{0}^t\phi(|\cdot|^2)\mathrm{e}^{-s\phi(|\cdot|^2)}\mathrm{d}s\right)\cF[f]\right](x)\\
    &=\int_0^t\phi(\Delta)u_1(s,x)\mathrm{d}s.
\end{aligned}
\end{equation}

$(ii)$ By \eqref{23.04.27.15.01},
\begin{align*}
    u_2(t,x)-\int_0^th(s,x)\mathrm{d}s&=\int_0^t\cF^{-1}\left[(\mathrm{e}^{-(t-s)\phi(|\cdot|^2)}-1)\cF[h(s,\cdot)]\right](x)\mathrm{d}s\\
    &=-\int_0^t\cF^{-1}\left[\left(\int_s^t\phi(|\cdot|^2)\mathrm{e}^{-(l-s)\phi(|\cdot|^2)}\mathrm{d}l\right)\cF[h(s,\cdot)]\right](x)\mathrm{d}s\\
    &=\int_0^t\phi(\Delta)u_2(l,x)\mathrm{d}l.
\end{align*}

$(iii)$  Given that $X_0 = 0$ with almost certainty, it implies $T_0f(x) = f(x)$.
For $f \in C_c^{\infty}(\mathbb{R}^d)$, the Fourier transform reveals
$$
\cF[T_{t+s}f](\xi)=\bE[\mathrm{e}^{\mathrm{i}\xi\cdot X_{t+s}}]\cF[f](\xi)=\mathrm{e}^{-(t+s)\phi(|\xi|^2)}\cF[f](\xi)=\cF[T_tT_sf](\xi).
$$
Minkowski's inequality, \eqref{24.03.17.12.31} and \eqref{23.04.27.15.01} guarantee
$$
\|T_tf-f\|_{L_p(\bR^d)}\leq t\|\phi(\Delta)f\|_{L_p(\bR^d)},
$$
leading to
$$
\lim_{t\downarrow0}\|T_tf-f\|_{L_p(\bR^d)}=0.
$$
For $f\in L_p(\bR^d)$, by taking $\{f_n\}_{n\in\bN}\subseteq C_c^{\infty}(\bR^d)$ such that $f_n\to f$ in $L_p(\bR^d)$, we have
$$
\|T_{t+s}f-T_tT_sf\|_{L_p(\bR^d)}\leq \|T_{t+s}(f-f_n)\|_{L_p(\bR^d)}+\|T_tT_s(f-f_n)\|_{L_p(\bR^d)}
$$
and
$$
\|T_tf-f\|_{L_p(\bR^d)}\leq \|T_tf-T_tf_n\|_{L_p(\bR^d)}+\|T_tf_n-f_n\|_{L_p(\bR^d)}+\|f_n-f\|_{L_p(\bR^d)}.
$$
Due to \eqref{24.03.17.12.31}, $(iii)$ is proved.

$(iv)$ For $f \in C_c^{\infty}(\mathbb{R}^d)$, leveraging \eqref{23.04.27.15.01}, we find:
\begin{align*}
    \left\|\frac{T_tf-f}{t}-\phi(\Delta)f\right\|_{L_p(\bR^d)}\leq \frac{1}{t}\int_0^t\|\phi(\Delta)(T_s-I)f\|_{L_p(\bR^d)}\mathrm{d}s\to0
\end{align*}
as $t\downarrow0$.
This result, supported by a density argument, verifies $(iv)$, completing the lemma's proof.
\end{proof}

\subsection{Functional calculus}
\label{24.05.06.12.50}
First, we introduce the concept of $H^{\infty}$-calculus.
For $\eta\in(0,\pi)$, denote
$$
\Sigma_{\eta}:=\{z\in\bC\setminus\{0\}:|\arg(z)|<\eta\}.
$$
We denote $H^p(\Sigma_{\eta})$ the complex Banach space of all holomorphic functions $f:\Sigma_{\eta}\to\bC$ satisfying
$$
\|f\|_{H^p(\Sigma_{\eta})}:=\sup_{|\nu|<\eta}\|f(\mathrm{e}^{\mathrm{i}\nu}t)\|_{L_p(\bR_+,\mathrm{d}t/t)}<\infty.
$$
Let $A$ be a linear operator on a Banach space $V$.
\begin{enumerate}[(i)]
    \item (Resolvent of $A$) We say that $z$ is in the resolvent set $\rho(A)$ of $A$ if
\begin{itemize}
    \item the range of $A_z:=z-A$ is dense in $V$.
    \item $A_z$ has a continuous inverse.
\end{itemize}
Here, for $z\in\rho(A)$, we can define $R(z,A):=(z-A)^{-1}$.
    \item (Sectorial operator) We say that a linear operator $A$ is sectorial if
\begin{itemize}
    \item there exists $\eta\in(0,\pi)$ such that the spectrum $\sigma(A):=\bC\setminus\rho(A)$ is contained in $\overline{\Sigma_{\eta}}$.
    \item the operator $zR(z,A)$ is uniformly bounded on $\bC\setminus\overline{\Sigma_{\eta}}$, \textit{i.e.},
    $$
    \sup_{z\in\bC\setminus\overline{\Sigma_{\eta}}}\|zR(z,A)\|_{V\to V}<\infty.
    $$
\end{itemize}
In this case, we say $A$ is $\eta$-sectorial.
We say that the constant
$$
\eta(A):=\inf\{\eta\in(0,\pi):A\text{ is $\eta$-sectorial}\}
$$
is called the angle of sectoriality of $A$.
    \item (Bounded $H^{\infty}$-calculus) Let $A$ be a sectorial operator with the angle of sectoriality $\eta(A)$.
For functions $f\in H^1(\Sigma_{\nu})$, denote
$$
f(A):=\frac{1}{2\pi \mathrm{i}}\int_{\partial\Sigma_{\nu}}f(z)R(z,A)\mathrm{d}z,
$$
where $\eta(A)<\nu<\sigma$ is chosen arbitrarily.
We remark that the definition of $f(A)$ is independent of the choice of $\nu$ (see \textit{e.g.} \cite[Chapter 10.2]{hytonen2018analysis}).
For a constant $\sigma\in(\eta(A),\pi)$, we say that the operator $A$ has a bounded $H^{\infty}(\Sigma_{\sigma})$-calculus if there exists a constant $C>0$ such that
$$
\|f(A)\|_{V\to V}\leq C\|f\|_{L_{\infty}},\quad f\in H^1(\Sigma_{\sigma})\cap H^{\infty}(\Sigma_{\sigma}).
$$
We define
$$
\eta_{H^{\infty}}(A):=\inf\{\sigma\in(\eta(A),\pi):A\text{ has a bounded $H^{\infty}(\Sigma_{\sigma})$-calculus}\},
$$
and we say that $A$ has a bounded $H^{\infty}$-calculus of angle $\eta_{H^{\infty}}(A)$.
\end{enumerate}

\begin{defn}[Analytic semigroup]
    Let $(L(t))_{t\geq0}$ be a $C_0$-semigroup on a Banach space $V$.
    \begin{enumerate}[(i)]
        \item $(L(t))_{t\geq0}$ is called analytic on $\Sigma_{\eta}$ if for all $x\in V$, the function $t\mapsto L(t)x$ extends analytically to $\Sigma_{\eta}$ and satisfies
$$
\lim_{z\in\Sigma_{\eta},z\to0}L(z)x=x.
$$
We call $(L(t))_{t\geq0}$ an analytic $C_0$-semigroup if $(L(t))_{t\geq0}$ is analytic on $\Sigma_{\eta}$ for some $\eta\in(0,\pi)$.
    \item We say that $\{L(t)\}_{t\geq0}$ a bounded analytic $C_0$-semigroup  if $\{L(t)\}_{t\geq0}$ is analytic and uniformly bounded on $\Sigma_{\eta}$ for some $\eta\in(0,\pi)$.
    \end{enumerate}
\end{defn}

We are now prepared to employ functional calculus within the framework 
$$
(V, A, L(t)) = (L_p(\mathbb{R}^d), -\phi(\Delta), T_t).
$$
\begin{lem}
\label{23.09.19.12.41}
Let $\phi$ be a Bernstein function, $X$ be a SBM with characteristic exponent $\phi$, and
$$
T_tf(x):=\bE[f(x+X_t)].
$$
Then,
\begin{enumerate}[(i)]
    \item the operator $-\phi(\Delta)$ has a bounded $H^{\infty}$-calculus on $L_p(\bR^d)$ of angle $0$.
    \item the semigroup $(T_t)_{t\geq0}$ is a bounded analytic $C_0$-semigroup on $L_p(\bR^d)$.
\end{enumerate}
\end{lem}
\begin{proof}
For $(i)$, we refer to \cite[Theorem 3.12]{KPR2022} for a comprehensive proof.
For part $(ii)$, consider a function $f \in C_c^{\infty}(\mathbb{R}^d)$ and examine the operator $t(-\phi(\Delta))T_tf(x)$ as follows:
$$
t(-\phi(\Delta))T_tf(x)=\cF^{-1}[t\phi(|\cdot|^2)\mathrm{e}^{-t\phi(|\cdot|^2)}\cF[f]](x)=:\cF^{-1}[m_t\cF[f]](x).
$$
Since
$$
|D^{n}\phi(\lambda)|\leq N(n)\lambda^{-n}\phi(\lambda),\quad \forall n\in\{0,1,2,\cdots\},
$$
for all multi-index $\alpha$, it holds that
$$
|D^{\alpha}_{\xi}m_t(\xi)|\leq N|\xi|^{-|\alpha|}.
$$
Applying Mihlin's multiplier theorem, coupled with a density argument, leads to
\begin{equation}
\label{23.08.23.12.33}
t\|\phi(\Delta)T_tf\|_{L_p(\mathbb{R}^d)} \leq N\|f\|_{L_p(\mathbb{R}^d)},
\end{equation}
where the constant $N$ does not depend on $t$ or $f$.
Leveraging Lemma \ref{23.09.19.12.43} and the theoretical framework provided in \cite[Theorem G.5.3]{hytonen2018analysis}, we establish the validity of $(ii)$, thereby concluding the proof of the lemma.
\end{proof}

We conclude this subsection by establishing key estimates that play a pivotal role in the next subsection.
\begin{thm}
\label{23.09.18.17.09}
Let $p,q\in(1,\infty)$.

$(i)$ Suppose that $u_0\in U_{p,q}^{\phi,2}$.
Then there exist $\cT_0u_0\in \cH_{p,q}^{\phi,2}$ and $\tilde{\cT_0}u_0\in \bL_{p,q}$ such that
    \begin{equation}
    \label{23.06.05.16.16}
        (\cT_0 u_0(t,\cdot),\varphi)=(u_0,\varphi)+\int_0^t(\tilde{\cT}_0u_0(s,\cdot),\varphi)\mathrm{d}s,\quad \forall \varphi\in C_c^{\infty}(\bR^d),
    \end{equation}
    and
    $$
    \|\cT_0u_0\|_{\bH_{p,q}^{\phi,2}}+\|\tilde{\cT}_0u_0\|_{\bL_{p,q}}\leq N_0\|u_0\|_{U_{p,q}^{\phi,2}},
    $$
    where $N_0$ is independent of $u_0$.

$(ii)$ Suppose that $f\in L_q([0,T];L_p(\bR^d))$.
Then
$$
\|\phi(\Delta)\cT_1f\|_{L_q([0,T];L_p(\bR^d))}\leq N_1\|f\|_{L_q([0,T];L_p(\bR^d))},
$$
where $N_1$ is independent of $f$, and
$$
\cT_1f(t,x):=\int_0^tT_{t-s}f(s,x)\mathrm{d}s.
$$

$(iii)$ Suppose that $g\in L_q(\bR_+;L_p(\bR^d;\ell_2))$.
Suppose also that $p=q=2$ or $p\in[2,\infty)$, $q\in(2,\infty)$.
Then 
$$
\|(-\phi(\Delta))^{1/2}\cT_2g\|_{L_q(\bR_+;L_p(\bR^d))}\leq N_2\|g\|_{L_q(\bR_+;L_p(\bR^d;\ell_2))},
$$
where $N_2$ is independent of $g$,
$$
\cT_2g(t,x):=\sum_{k=1}^{\infty}\int_0^tT_{t-s}g^k(s,x)\mathrm{d}w_s^k,
$$
and $\{w^k\}_{k=1}^{\infty}$ is a sequence of $1$-dimensional independent Brownian motions.
\end{thm}
\begin{proof}
    $(i)$ 
    Given $u_0 \in U_{p,q}^{\phi,2}$, Theorem \ref{23.06.25.12.37} guarantees the existence of a sequence ${u_{0,n}} \subseteq \boldsymbol{\mathrm{H}}_{c,0}^{\infty}$ converging to $u_0$ in $U_{p,q}^{\phi,2}$. Each $u_{0,n}$ can be expressed as
    $$
    u_{0,n}=\sum_{i=1}^{k(n)}1_{A_i}u_{0,n}^{i},
    $$
    with $u_{0,n}^i \in \mathcal{S}(\mathbb{R}^d)$.
    Applying the extension theorem (\textit{cf.} \cite[Theorem 1.5]{CLSW2023trace}), we find corresponding $\cT_0u_{0,n}^{i} \in L_q(\mathbb{R}_+; H_p^{\phi,2}(\mathbb{R}^d))$ and $\tilde{\cT}_0u_{0,n}^{i} \in L_q(\mathbb{R}_+; L_p(\mathbb{R}^d))$, satisfying
    $$
    (\cT_0u_{0,n}^{i},\varphi)=(u_{0,n}^i,\varphi)+\int_0^t(\tilde{\cT}_0u_{0,n}^{i},\varphi)\mathrm{d}s,\quad \forall \varphi\in C_c^{\infty}(\bR^d)
    $$
    and
    $$
    \|\cT_0u_{0,n}^{i}\|_{L_q(\bR_+;H_p^{\phi,2}(\bR^d))}+\|\tilde{\cT}_0u_{0,n}^{i}\|_{L_q(\bR_+;L_p(\bR^d))}\leq N\|u_{0,n}^i\|_{(H_p^{\phi,2}(\bR^d),L_p(\bR^d))_{1/q,q}}.
    $$
    Constructing $\cT_0u_{0,n}$ and $\tilde{\cT}_0u_{0,n}$ as
    $$
    \cT_0u_{0,n}:=\sum_{i=1}^{k(n)}1_{A_i}\cT_0u_{0,n}^{i},\quad \tilde{\cT}_0u_{0,n}:=\sum_{i=1}^{k(n)}1_{A_i}\tilde{\cT}_0u_{0,n}^{i},
    $$
    we find that they satisfy
    $$
        (\cT_0u_{0,n}(t,\cdot),\varphi)=(u_{0,n},\varphi)+\int_0^t(\tilde{\cT}_0u_{0,n}(s,\cdot),\varphi)\mathrm{d}s
    $$
    for all $\varphi \in C_c^{\infty}(\mathbb{R}^d)$ and 
    $$
    \|\cT_0u_{0,n}\|_{\bH_{p,q}^{\phi,2}}+\|\tilde{\cT}_0u_{0,n}\|_{\bL_{p,q}}\leq N\|u_{0,n}\|_{U_{p,q}^{\phi,2}},
    $$
    indicating $\cT_0u_{0,n} \in \cH_{p,q}^{\phi,2}$ and $\tilde{\cT}_0u_{0,n} \in \bL_{p,q}$. The convergence of $\cT_0u_{0,n} \to \cT_0u_{0}$ in $\cH_{p,q}^{\phi,2}$ and $\tilde{\cT}_0u_{0,n} \to \tilde{\cT}_0u_{0}$ in $\bL_{p,q}$, alongside the norm bound, establishes $(i)$.

    $(ii)$ By \cite[Theorem 11.9]{D2004}, $-\phi(\Delta)$ has bounded imaginary powers.
    For $f \in C_c^{\infty}([0,T] \times \mathbb{R}^d)$, Lemma \ref{23.09.19.12.43} provides that $u(t,x) := \int_0^t T_{t-s}f(s,x)  \mathrm{d}s$ fulfills 
    $$
    u(t,x)=\int_0^t(\phi(\Delta)u(s,x)+f(s,x))\mathrm{d}s,\quad t\in[0,T).
    $$
    This aligns with the conditions of \cite[Theorem 8.7]{P2013}, implying the boundedness of $\cT_2$ by applying \cite[Theorem 8.7]{P2013} and a density argument.

    $(iii)$ The preconditions set forth in Lemmas \ref{23.09.19.12.43} and \ref{23.09.19.12.41} fulfill the criteria required by \cite[Theorem 1.1]{NVW2012}.
    By invoking \cite[Theorem 1.1]{NVW2012}, we establish the validity of claim $(iii)$, thereby completing the proof of the theorem.
\end{proof}

\begin{rem}
    We remark that Theorem \ref{23.09.18.17.09}-($iii$) does not hold when $q=2$ and $p\in(2,\infty)$.
For more detail, see \cite[Section 6]{NVW2012}.
\end{rem}

\subsection{Solvability of semilinear stochastic integro-differential equations}
In this subsection, we address the solvability of a semilinear stochastic integro-differential equation in the space $\cH_{p,q}^{\phi,\gamma+2}(\tau)$, as defined by
\begin{equation}
\label{23.04.25.17.15}
    \mathrm{d}u=(\phi(\Delta)u+f(u))\,\mathrm{d}t+\sum_{k=1}^{\infty}g^k(u)\,\mathrm{d}w_t^k;\quad u(0)=u_0.
\end{equation}
\begin{thm}
\label{23.07.04.14.29}
Let $\tau\leq T$ be a bounded stopping time, $\gamma\in\bR$, and $p,q\in(2,\infty)$.
Suppose that $f(u)$ and $g(u)$ satisfy the followings:
\begin{enumerate}[(a)]
        \item for any $u\in H_p^{\phi,\gamma+2}(\bR^d)$, the functions $f(u)$ and $g(u)$ are  $H_p^{\phi,\gamma}(\bR^d)$ and $H_{p}^{\phi,\gamma+1}(\bR^d;\ell_2)$-valued predictable functions, respectively.
        \item for any $\varepsilon>0$, there exists a constant $N_{\varepsilon}>0$ such that
        \begin{equation}
\label{ass_lin}
\begin{aligned}
    &\|f(\omega,t,\cdot,u(\cdot))-f(\omega,t,\cdot,v(\cdot))\|_{H_p^{\phi,\gamma}(\bR^d)}\\
    &\quad+\|g(\omega,t,\cdot,u(\cdot))-g(\omega,t,\cdot,v(\cdot))\|_{H_p^{\phi,\gamma+1}(\bR^d;\ell_2)} \\
    &\leq \varepsilon\|u-v\|_{H_p^{\phi,\gamma+2}(\bR^d)}+N_{\varepsilon}\|u-v\|_{H_p^{\phi,\gamma}(\bR^d)}
    \end{aligned}
\end{equation}
for any $u,v\in H_p^{\phi,\gamma}(\bR^d)$ and $(\omega,t)\in\opar0,\tau\cbrk$.
    \item $(f(0),g(0))\in\bH_{p,q}^{\phi,\gamma}(\tau)\times\bH_{p,q}^{\gamma+1}(\tau,\ell_2)$.
    \end{enumerate}
Then, for any $u_0\in U_{p,q}^{\phi,\gamma+2}$, the equation \eqref{23.04.25.17.15} has a unique solution $u\in \cH_{p,q}^{\phi,\gamma+2}(\tau)$ with the estimate
\begin{equation}
\label{25.02.18.12.32}
\|u\|_{\cH_{p,q}^{\phi,\gamma+2}(\tau)}\leq N\left(\|u_0\|_{U_{p,q}^{\phi,\gamma+2}}+\|f(0)\|_{\bH_{p,q}^{\phi,\gamma}(\tau)}+\|g(0)\|_{\bH_{p,q}^{\phi,\gamma+1}(\tau,\ell_2)}\right),
\end{equation}
where $N=N(N_0,N_1,N_2,T)$.
Here $N_0,N_1,N_2$ are constants in Theorem \ref{23.09.18.17.09}.
\end{thm}
\begin{proof}
Since $(1-\phi(\Delta))^{\gamma/2}:U_{p,q}^{\phi,\gamma+2}\times\bH_{p,q}^{\phi,\gamma}(\tau)\times\bH_{p,q}^{\phi,\gamma+1}(\tau,\ell_2)\to U_{p,q}^{\phi,2}\times\bL_{p,q}(\tau)\times\bH_{p,q}^{\phi,1}(\tau,\ell_2)$ is an isometry, we assume that $\gamma=0$.

\textbf{Case 1.}  $f(u)=f$ and $g(u)=g$.

\textit{Uniqueness.}
     Suppose that $u\in \cH_{p,q}^{\phi,2}(\tau)$ satisfies
     \begin{align}
     \label{23.06.02.13.01}
         (u(t,\cdot),\varphi)=\int_0^t(\phi(\Delta)u(s,\cdot),\varphi)\mathrm{d}s,\quad \forall \varphi\in C_c^{\infty}(\bR^d).
     \end{align}
    If we put $\eta_{\varepsilon}(x-y):=\varepsilon^{-d}\eta(x/\varepsilon)$ instead of $\varphi(x)$ in \eqref{23.06.02.13.01} where $\varphi$ is nonnegative with unit integral, then $u_{\varepsilon}(t,x):=(u(t,\cdot),\eta_{\varepsilon}(x-.\cdot))$ satisfies
    $$
    u_{\varepsilon}(t,x)=\int_0^t\phi(\Delta)u_{\varepsilon}(s,x)\mathrm{d}s.
    $$
    Let $v(t,x):=\mathrm{e}^{-\lambda t}u_{\varepsilon}(t,x)$.
    Then $v$ satisfies the assumption in \cite[Lemma 6.5]{KKK2013}, thus, $v\equiv0$ and thus, $u_{\varepsilon}\equiv0$.
    Since $\varepsilon$ is an arbitrary positive constant, $u\equiv0$.

    \textit{Existence and estimates.}
    Due to Theorem \ref{23.06.25.12.37}, we consider only $(u_0,f,g)\in \boldsymbol{\mathrm{H}}_{c,0}^{\infty}\times \boldsymbol{\mathrm{H}}_{c}^{\infty}\times \boldsymbol{\mathrm{H}}_{c}^{\infty}(\ell_2)$.
    Let $u_1(t,x):=\cT_0u_0(t,x)$,
    \begin{equation*}
         u_2(t,x):=\cT_1(f-\tilde{\cT}_0u_0+\phi(\Delta)\cT_0u_0)(t,x),\quad u_3(t,x):=\cT_2g(t,x),
    \end{equation*}
where $\cT_0$, $\tilde{\cT}_0$, $\cT_1$ and $\cT_2$ are operators in Theorem \ref{23.09.18.17.09}.
Since $g\in \boldsymbol{\mathrm{H}}_{c}^{\infty}(\ell_2)$, by \eqref{23.04.27.15.01}, $u_3(t,x)=v(t,x)+\phi(\Delta)\cT_1v(t,x)$, where
    $$
    v(t,x):=\sum_{k=1}^{\infty}\int_0^tg^k(s,x)\mathrm{d}w_s^k.
    $$
    For $\varphi\in C_c^{\infty}(\bR^d)
    $,
    \begin{align*}
        (u_3(t,\cdot)-v(t,\cdot),\varphi)=(\phi(\Delta)\cT_1v(t,\cdot),\varphi)&=\int_0^t(\phi(\Delta)(\phi(\Delta)\cT_1v)
        (s,\cdot)+\phi(\Delta)v(s,\cdot),\varphi)\mathrm{d}s\\
        &=\int_0^t(\phi(\Delta)(u_3(s,\cdot)-v(s,\cdot))+\phi(\Delta)v(s,\cdot),\varphi)\mathrm{d}s\\
        &=\int_0^t(\phi(\Delta)u_3(s,\cdot),\varphi)\mathrm{d}s.
    \end{align*}
    Therefore, $u_3$ satisfies
    $$
    (u_3(t,\cdot),\varphi)=\int_0^t(\phi(\Delta)u_3(s,\cdot),\varphi)\mathrm{d}s+\sum_{k=1}^{\infty}\int_0^t(g^k(s,\cdot),\varphi)\mathrm{d}w_s^k,\quad \forall \varphi\in C_c^{\infty}(\bR^d).
    $$
    Since $u_1+u_2$ satisfies
    $$
    (u_1(t,\cdot)+u_2(t,\cdot),\varphi)=(u_0,\varphi)+\int_0^t(\phi(\Delta)(u_1(s,\cdot)+u_2(s,\cdot))+f(s,\cdot),\varphi)\mathrm{d}s,\quad \forall \varphi\in C_c^{\infty}(\bR^d),
    $$
    $u:=u_1+u_2+u_3$ satisfies
    $$
    (u(t,\cdot),\varphi)=(u_0,\varphi)+\int_0^t(\phi(\Delta)u(s,\cdot)+f(s,\cdot),\varphi)\mathrm{d}s+\sum_{k=1}^{\infty}\int_0^t(g^k(s,\cdot),\varphi)\mathrm{d}w_s^k,\quad \forall \varphi\in C_c^{\infty}(\bR^d).
    $$
    Now, we prove $\phi(\Delta)u\in \bL_{p,q}(\tau)$.
    One can check that
    \begin{align*}
        \phi(\Delta)u&=\phi(\Delta)u_1+\phi(\Delta)u_2+\phi(\Delta)u_3\\
        &=\phi(\Delta)\cT_0u_0+\phi(\Delta)\cT_1(f-\tilde{\cT}_0u_0+\phi(\Delta)\cT_0u_0)+\phi(\Delta)\cT_2g.
    \end{align*}
    By Theorem \ref{23.09.18.17.09}, we have
    \begin{align*}
        \|\phi(\Delta)u\|_{\bL_{p,q}(\tau)}&\leq N\bigg(\|\phi(\Delta)\cT_0u_0\|_{\bL_{p,q}(\tau)}+\|\phi(\Delta)\cT_1(f-\tilde{\cT}_0u_0+\phi(\Delta)\cT_0u_0)\|_{\bL_{p,q}(\tau)}\\
        &\qquad+\|\phi(\Delta)\cT_2g\|_{\bL_{p,q}(\tau,\ell_2)}\bigg)\\
        &\leq N\left(\|\cT_0u_0\|_{\bH_{p,q}^{\phi,2}(\tau)}+\|\tilde{\cT}_0u_0\|_{\bL_{p,q}(\tau)}+\|f\|_{\bL_{p,q}(\tau)}+\|(-\phi(\Delta))^{1/2}g\|_{\bL_{p,q}(\tau,\ell_2)}\right)\\
        &\leq N\left(\|u_0\|_{U_{p,q}^{\phi,2}}+\|f\|_{\bL_{p,q}(\tau)}+\|g\|_{\bH_{p,q}^{\phi,1}(\tau,\ell_2)}\right).
    \end{align*}
    This certainly implies that
    $$
    \|u\|_{\cH_{p,q}^{\phi,2}(\tau)}\leq N\left(\|u_0\|_{U_{p,q}^{\phi,2}}+\|f\|_{\bL_{p,q}(\tau)}+\|g\|_{\bH_{p,q}^{\phi,1}(\tau,\ell_2)}\right).
    $$

\textbf{Case 2.} General case.

By putting
$$
\tilde{f}:=f1_{\opar 0,\tau\cbrk},\quad\tilde{g}:=g1_{\opar0,\tau\cbrk},
$$
it suffices to prove the result for $\tau\equiv T$.
Indeed, $\tilde{f}$ and $\tilde{g}$ satisfy \eqref{ass_lin} for all $t\leq T$.
Therefore, if the result holds for $\tau\equiv T$, then there exists a unique solution $\tilde{u}\in\mathcal{H}_{p,q}^{\phi,\gamma+2}(T)$ to \eqref{23.04.25.17.15} and estimate \eqref{25.02.18.12.32} with $(\tilde{u},\tilde{f},\tilde{g})$, in place of $(u,f,g)$. 
It then follows that $\tilde{u}\in\mathcal{H}_{p,q}^{\phi,\gamma+2}(\tau)$ and satisfies \eqref{23.04.25.17.15} for $t\leq \tau$.

    Let $\tilde{u}\in\cH_{p,q}^{\phi,\gamma+2}(T)$.
    By Case 1, there exists a unique solution $v\in\cH_{p,q}^{\phi,\gamma+2}(T)$ such that
    \begin{equation*}
    \mathrm{d}v=(\phi(\Delta)v+f(\tilde{u}))\mathrm{d}t+\sum_{k=1}^{\infty}g^k(\tilde{u})\mathrm{d}w_t^k;\quad v(0)=u_0
    \end{equation*}
    and
    $$
    \|v\|_{\cH_{p,q}^{\phi,\gamma+2}(t)}\leq N\left(\|u_0\|_{U_{p,q}^{\phi,\gamma+2}}+\|f(\tilde{u})\|_{\bH_{p,q}^{\phi,\gamma}(t)}+\|g(\tilde{u})\|_{\bH_{p,q}^{\phi,\gamma+1}(t,\ell_2)}\right).
    $$
    If we denote $v:=\mathcal{R}\tilde{u}$, then $\mathcal{R}$ satisfies
    $$
    \|\mathcal{R}u-\mathcal{R}v\|_{\cH_{p,q}^{\phi,\gamma+2}(t)}\leq N\left(\|f(u)-f(v)\|_{\bH_{p,q}^{\phi,\gamma}(t)}+\|g(u)-g(v)\|_{\bH_{p,q}^{\phi,\gamma+1}(t,\ell_2)}\right).
    $$
    By the assumption on $f$ and $g$, and Theorem \ref{23.06.25.12.37}-$(ii)$,
    \begin{equation}
    \label{24.03.04.14.54}
        \|\mathcal{R}u-\mathcal{R}v\|_{\cH_{p,q}^{\phi,\gamma+2}(t)}^q\leq N\varepsilon^q\|u-v\|_{\cH_{p,q}^{\phi,\gamma+2}(t)}^q+N\int_0^t\|u-v\|_{\cH_{p,q}^{\phi,\gamma+2}(s)}^q\mathrm{d}s.
    \end{equation}
    Let $\theta:=N\varepsilon^q$. Using the mathematical induction, we have
    \begin{align*}
        &\|\mathcal{R}^mu-\mathcal{R}^mv\|_{\cH_{p,q}^{\phi,\gamma+2}(t)}^q\\
        &\leq \theta^m\|u-v\|_{\cH_{p,q}^{\phi,\gamma+2}(t)}^q+\sum_{k=1}^m\binom{m}{k}\theta^{m-k}N_{\theta}^k\int_0^t\frac{(t-s)^{k-1}}{(k-1)!}\|u-v\|_{\cH_{p,q}^{\phi,\gamma+2}(s)}^q\mathrm{d}s.
    \end{align*}
    By elementary calculation,
    \begin{align*}
    &\sum_{k=1}^{m}\binom{m}{k}\theta^{m-k}N_{\theta}^k\int_0^T\frac{(T-t)^{k-1}}{(k-1)!}\|u-v\|_{\mathcal{H}_{p,q}^{\phi,\gamma+2}(t)}^q\mathrm{d}t\\
    &\leq \|u-v\|_{\mathcal{H}_{p,q}^{\phi,\gamma+2}(T)}^q\sum_{k=1}^{m}\binom{m}{k}\theta^{m-k}N_{\theta}^k\int_0^T\frac{(T-t)^{k-1}}{(k-1)!}\mathrm{d}t\\
    &=\|u-v\|_{\mathcal{H}_{p,q}^{\phi,\gamma+2}(T)}^q\sum_{k=1}^{m}\frac{\binom{m}{k}\theta^{m-k}(TN_{\theta})^k}{k!}\\
    &\leq \|u-v\|_{\mathcal{H}_{p,q}^{\phi,\gamma+2}(T)}^q2^m\theta^m\sum_{k=1}^{\infty}\frac{(TN_{\theta}/\theta)^k}{k!}=\|u-v\|_{\mathcal{H}_{p,q}^{\phi,\gamma+2}(T)}^q2^m\theta^m\mathrm{e}^{TN_{\theta}/\theta}.
\end{align*}
    If we take $\theta=1/8$ and $m\in\mathbb{N}$ such that $\mathrm{e}^{TN_{\theta}/\theta}<2^m$, then we have
    \begin{equation}
    \label{24.03.04.13.41}
    \|\mathcal{R}^mu-\mathcal{R}^mv\|_{\cH_{p,q}^{\phi,\gamma+2}(T)}^q\leq \frac{1}{2^m}\|u-v\|_{\cH_{p,q}^{\phi,\gamma+2}(T)}^q.
    \end{equation}
    Therefore, $\mathcal{R}^m$ is contraction in $\cH_{p,q}^{\phi,\gamma+2}(T)$.
    This implies that there exists $u\in\cH_{p,q}^{\phi,\gamma+2}(T)$ such that $\mathcal{R}^mu=u$.
    Moreover,
    $$
    \|\mathcal{R}u-u\|_{\cH_{p,q}^{\phi,\gamma+2}(T)}^q=\|\mathcal{R}^{m+1}u-\mathcal{R}^mu\|_{\cH_{p,q}^{\phi,\gamma+2}(T)}^q\leq\frac{1}{2}\|\mathcal{R}u-u\|_{\cH_{p,q}^{\phi,\gamma+2}(T)}^q,
    $$
    thus, $\mathcal{R}u=u$.
    If there exist two fixed points $u,w\in\cH_{p,q}^{\phi,\gamma+2}(T)$ of $\mathcal{R}$, then due to \eqref{24.03.04.13.41},
    \begin{align*}
    \|u-w\|_{\cH_{p,q}^{\phi,\gamma+2}(T)}&=\|\mathcal{R}u-\mathcal{R}w\|_{\cH_{p,q}^{\phi,\gamma+2}(T)}\\
    &=\|\mathcal{R}^mu-\mathcal{R}^mw\|_{\cH_{p,q}^{\phi,\gamma+2}(T)}\leq \frac{1}{2}\|u-w\|_{\cH_{p,q}^{\phi,\gamma+2}(T)}.
    \end{align*}
    Therefore, the mapping $\mathcal{R}:\cH_{p,q}^{\phi,\gamma+2}(T)\to \cH_{p,q}^{\phi,\gamma+2}(T)$ has a unique fixed point $u$.
    For estimation, we use \eqref{24.03.04.14.54} and Case 1:
    \begin{align*}
        \|u\|_{\cH_{p,q}^{\phi,\gamma+2}(t)}&\leq\|\mathcal{R}u-\mathcal{R}0\|_{\cH_{p,q}^{\phi,\gamma+2}(t)}+\|\mathcal{R}0\|_{\cH_{p,q}^{\phi,\gamma+2}(t)} \\
        &\leq N\varepsilon^q\|u\|_{\cH_{p,q}^{\phi,\gamma+2}(t)}^q+N\int_0^t\|u\|_{\cH_{p,q}^{\phi,\gamma+2}(s)}^q\mathrm{d}s\\
        &\quad + N\left(\|u_0\|_{U_{p,q}^{\phi,\gamma+2}}+\|f(0)\|_{\bH_{p,q}^{\phi,\gamma}(T)}+\|g(0)\|_{\bH_{p,q}^{\phi,\gamma+1}(T,\ell_2)}\right).
    \end{align*}
    By taking sufficiently small $\varepsilon>0$ and the Gr\"onwall inequality, we have the desired estimates. The theorem is proved.
\end{proof}

\mysection{Proof of Theorem \ref{main}}
\label{24.03.21.14.39}

This section dedicates our efforts to the proof of Theorem \ref{main}.
Unless explicitly mentioned otherwise, we operate under the assumption that Assumptions \ref{main_ass_corr} ($\delta_0,\gamma,r$), \ref{24.03.14.11.58} ($\delta_0$), \ref{ass_nl_term}, and \ref{ass_coeff}($\delta_0$,$\gamma$) are satisfied throughout.
Our strategy for demonstrating Theorem \ref{main} involves initially establishing a local (in time) solution.
We then extend this to a global (in time) solution by leveraging the established existence, uniqueness, and regularity properties of local solutions to infer analogous properties for global solutions.
Here are the results for local solutions and the proofs will be given in Sections \ref{24.03.05.14.26}, \ref{24.03.24.11.02}, \ref{24.03.22.14.23}, and \ref{24.03.22.14.29}.
\begin{thm}
\label{local}
    Suppose that $p,q\in(2,\infty)$, $\gamma\in(0,1)$, and $u_0\in U_{p,q}^{\phi,\gamma}\cap L_1(\Omega\times\mathbb{R}^d)$ is nonnegative.
    \begin{enumerate}[(i)]
        \item (Existence, Uniqueness, and Regularity of local solutions) For any bounded stopping time $\tau$, the equation
    \begin{equation}
    \label{23.07.08.13.42}
    \mathrm{d}u_m=(\phi(\Delta)u_m+\zeta F_m(u_m)+\vec{b}\cdot\nabla_{x} (B_m(u_m)))\mathrm{d}t+\sum_{k=1}^{\infty}\xi\varphi_m(u_m)(\pi\ast e_k)\mathrm{d}w_t^k,
    \end{equation}
    on $\opar0,\tau\cbrk\times\bR^d$
    with initial data $u_0$ has a unique solution $u_m\in\cH_{p,q}^{\phi,\gamma}(\tau)$, where
    \begin{equation}
    \label{24.03.24.12.42}
    F_m(u):=F(u\vee0)h_m(u),\quad B_m(u):=B(u\vee0)h_m(u),\quad \varphi_m(u):=\varphi(u\vee0)h_m(u).
    \end{equation}
    Here, $h\in C^1(\bR)$ is a nonnegative function such that $h(z)=1$ on $|z|\leq 1$ and $h(z)=0$ on $|z|\geq2$, and
    $$
    h_m(z):=h(z/m),\quad \forall m\in\bN.
    $$
    
    \item (Nonnegativity of local solutions) The solution $u_m$ in $(i)$ is nonnegative.
    
    \item (Space-time H\"older type regularity of local solutions) If
    $$
    \frac{1}{q}\leq \alpha<\beta<\frac{\gamma}{2}-\frac{d}{2\delta_0p},
    $$
    then there exists a positive constant $N=N(d,N_0,N_1,N_2,p, q, T)$ such that
    $$
    \mathbb{E}\left[|u_m|_{C_{t,x}^{\alpha-\frac{1}{q},\phi_{\gamma-2\beta}-\frac{d}{p}}([0,\tau]\times\bR^d)}^q\right]\leq N\|u_m\|_{\cH_{p,q}^{\phi,\gamma}(\tau)}^q.
    $$
    If $\alpha=1/q$, then
    $$
    C_{t,x}^{\alpha-\frac{1}{q},\phi_{\gamma-2\beta}-\frac{d}{p}}([0,\tau]\times\bR^d):=C([0,\tau];C^{\phi_{\gamma-2\beta}-\frac{d}{p}}(\bR^d))
    $$
    
    \item (Non-explosivity of local solutions) Suppose also that $u_0\in L_1(\Omega\times\bR^d)$,
\begin{align*}
        0<&\lambda_{\mathrm{b.}}+\frac{d+1-\delta_0(2-\gamma)}{d}-\frac{1}{p}-\frac{\lambda_{\mathrm{s.d.}}}{p\vee q}<1, \\
        0<&\lambda_{\mathrm{s.m.}}+\frac{1}{2r}+\frac{1}{2}-\frac{1}{p}-\frac{\lambda_{\mathrm{s.d.}}}{p\vee q}<1,
\end{align*}
and
$$
\frac{2}{q}+\frac{d}{\delta_0 p}<\gamma.
$$
    Then for any $T$,
    $$
    \lim_{R\to\infty}\sup_{m\in\bN}\bP\left(\left\{\omega\in\Omega:\sup_{t\leq T}\sup_{x\in\bR^d}|u_m(\omega,t,x)|\geq R\right\}\right)=0.
    $$
    \end{enumerate}
\end{thm}

The proof of Theorem \ref{local} will be provided in Sections \ref{24.03.05.14.26}, \ref{24.03.24.11.02}, \ref{24.03.22.14.23} and \ref{24.03.22.14.29}.
Assuming Theorem \ref{local} holds, we now proceed to prove Theorem \ref{main}.
\begin{proof}[Proof of Theorem \ref{main}]
$(i)$ The proof is structured into two main parts: uniqueness and existence.

\textbf{Step 1. Uniqueness.}
Let us consider two nonnegative solutions $u, v \in \cH_{p,q,loc}^{\phi,\gamma}$ to equation \eqref{23.06.28.17.21}.
We can identify a sequence of stopping times $\{\tau_m\}_{m=1}^{\infty}$, where $\tau_m \nearrow \infty$ and both $u, v \in \cH_{p,q}^{\phi,\gamma}(\tau_m)$, satisfying
$$
\sup_{\omega\in\Omega}\tau_m(\omega)<\infty,\quad \forall m\in\bN.
$$
Referencing Theorem \ref{local}-$(iii)$, we find that
    \begin{align}
    \label{23.09.25.12.06}
        \bE\left[|u|_{C([0,\tau_m]\times\bR^d)}+|v|_{C([0,\tau_m]\times\bR^d)}\right]<\infty.
    \end{align}
    For each $n \in \mathbb{N}$, define
    $$
    \tau_m^n(u):=\inf\{t\leq\tau_m:|u(t,\cdot)|_{C(\bR^d)}>n\},\quad \tau_m^n(v):=\inf\{t\leq\tau_m:|v(t,\cdot)|_{C(\bR^d)}>n\}
    $$
    and let $\tau_m^n := \tau_m^n(u) \wedge \tau_m^n(v)$.
    From \eqref{23.09.25.12.06}, $\tau_m^n(u)$ and $\tau_m^n(v)$ are stopping times, hence so is $\tau_m^n$.
    Considering the functions $F_n$, $B_n$, and $\varphi_n$ defined in \eqref{24.03.24.12.42}, we note that
    $$
    F_n(u)=F(u),\quad B_n(u)=B(u), \quad \varphi_n(u)=\varphi(u),\quad \forall (\omega,t)\in \cpar0,\tau_m^n\cbrk,
    $$
    making $u$ a solution to
    \begin{align}
    \label{23.09.25.13.41}
        \mathrm{d}u=(\phi(\Delta)u+\zeta F_n(u)+\vec{b}\cdot\nabla_{x} (B_n(u)))\mathrm{d}t+\xi\varphi_n(u)\dot{W},\quad 0<t\leq\tau_m^n
    \end{align}
    due to the definition of $\tau_m^n$.
    Similarly, $v$ satisfies equation \eqref{23.09.25.13.41} as well.
    According to Theorem \ref{local}, $u = v$ in $\cH_{p,q}^{\phi,\gamma}(\tau_m^n)$ for all $n \in \mathbb{N}$.
    As $\tau_m^n \nearrow \tau_m$ almost surely and by the application of the monotone convergence theorem, it follows that $u = v$ in $\cH_{p,q}^{\phi,\gamma}(\tau_m)$, thereby establishing uniqueness.

\textbf{Step 2. Existence.}
    For the proof of existence, let $u_m$ denote a solution to \eqref{23.07.08.13.42}, as described in Theorem \ref{local}. Given that $u_m \in \cH_{p,q}^{\phi,\gamma}(T)$ for any finite $T$, Theorem \ref{local}-$(iii)$ ensures
    $$
    \bE\left[\sup_{t\leq T}\sup_{x\in\bR^d}|u_m(t,x)|\right]<\infty.
    $$
    For $1 \leq R \leq m$, we define the stopping times
    $$
    \tau_m^R:=\inf\left\{t>0:\sup_{x\in\bR^d}|u_m(t,x)|\geq R\right\}.
    $$
    Our initial claim is that
    \begin{equation}
    \label{23.10.01.21.35}
        \tau_R^R=\tau_m^R\quad (a.s.).
    \end{equation}
    Given that
    $$
    \sup_{x\in\bR^d}|u_m(t,x)|\leq R\leq m,\quad \forall t\leq\tau_m^R,
    $$
    it follows that $h_m(u_m) = h_R(u_m) = 1$ for $t \leq \tau_m^R$, where $h$ is defined in Theorem \ref{local}.
    Similarly, given that
    $$
    \sup_{x\in\bR^d}|u_R(t,x)|\leq R\leq m,\quad \forall t\leq\tau_R^R,
    $$
    it holds that $h_R(u_R) = h_m(u_R) = 1$ for $t \leq \tau_R^R$.
    By the uniqueness established in Theorem \ref{local}, we have $u_m = u_R \in \cH_{p,q}^{\phi,1-\gamma}((\tau_R^R \vee \tau_m^R) \wedge T)$ for any $T < \infty$.
    Consequently, for almost surely $\omega \in \Omega$,
    $$
    \sup_{s\leq t}\sup_{x\in\bR^d}|u_R(\omega,s,x)|=\sup_{s\leq t}\sup_{x\in\bR^d}|u_m(\omega,t,x)|\leq R\quad \forall t\leq \tau_m^R,
    $$
    directly implying that $\tau_R^R \leq \tau_m^R$ (a.s.).
    A similar argument establishes that $\tau_m^R \leq \tau_R^R$ (a.s.), thereby proving \eqref{23.10.01.21.35}.
    
    Next, we demonstrate that $\tau_m^m \nearrow \infty$ a.s. as $m \to \infty$.
    Since $\tau_m^R \leq \tau_m^m$, the sequence $\{\tau_m^m\}_{m=1}^{\infty}$ is an a.s. increasing sequence of stopping times.
    By Theorem \ref{local},
    \begin{align*}
        \limsup_{m\to\infty}\bP(\tau_m^m\leq T)&=\limsup_{m\to\infty}\bP\left(\left\{\omega\in\Omega:\sup_{t\leq T}\sup_{x\in\bR^d}|u_m(\omega,t,x)|\geq m\right\}\right)\\
        &\leq \limsup_{S\to\infty}\sup_{m\in\bN}\bP\left(\left\{\omega\in\Omega:\sup_{t\leq T}\sup_{x\in\bR^d}|u_m(\omega,t,x)|\geq S\right\}\right)=0,
    \end{align*}
    implying $\tau_m^m \to \infty$ in probability. 
    As $\{\tau_m^m\}_{m=1}^{\infty}$ is almost surely increasing, we conclude $\tau_m^m \nearrow \infty$ (a.s.) as $m \to \infty$.

    Let us define the stopping time $\tau_m$ as the minimum of $\tau_m^m$ and $m$, formally presented as
    \begin{equation}
    \label{24.03.19.12.31}
        \tau_m:=\tau_m^m\wedge m.
    \end{equation}
    Consequently, we specify the solution $u(t,x)$ in the interval $[0, \tau_m)$ as
    $$
    u(t,x):=u_m(t,x),\quad t\in[0,\tau_m).
    $$
    Given that $\tau_m \leq m$ and $\tau_m \nearrow \infty$ almost surely, alongside the fact that $u_m \in \cH_{p,q}^{\phi,\gamma}(\tau_m)$, it follows that $u \in \cH_{p,q,loc}^{\phi,\gamma}$.
    This is supported by the observation that
    $$
    \sup_{x\in\bR^d}|u(t,x)|=\sup_{x\in\bR^d}|u_m(t,x)|\leq m,\quad \forall t\leq\tau_m,
    $$
    ensuring that $u$ satisfies equation \eqref{23.06.28.17.21}.

    $(ii)$ Consider $\tau_m$ as defined in \eqref{24.03.19.12.31} to be a stopping time.
    For any bounded stopping time $\tau$, we have $u = u_m \in \cH_{p,q}^{\phi,\gamma}(\tau_m \wedge \tau)$.
    Invoking Theorem \ref{local}-($iii$), we obtain
    $$
    \mathbb{E}\left[|u|_{C_{t,x}^{\alpha-\frac{1}{q},\phi_{\gamma-2\beta}-\frac{d}{p}}([0,\tau_m\wedge\tau]\times\bR^d)}^q\right]\leq N\|u\|_{\cH_{p,q}^{\phi,\gamma}(\tau_m\wedge\tau)}^q.
    $$
    Subsequently, we define a subset $\Omega_m \subseteq \Omega$ where $\mathbb{P}(\Omega_m) = 1$ and for every $\omega \in \Omega_m$,
    $$
    |u(\omega,\cdot)|_{C_{t,x}^{\alpha-\frac{1}{q},\phi_{\gamma-2\beta}-\frac{d}{p}}([0,\tau_m(\omega)\wedge\tau(\omega)]\times\bR^d)}<\infty.
    $$
    Given $\tau_m \nearrow \infty$ almost surely, there exists $\Omega' \subset \Omega$ with $\mathbb{P}(\Omega') = 1$, ensuring
    $$
    \lim_{m\to\infty}\tau_m(\omega)=\infty,\quad \forall \omega\in\Omega'.
    $$
    Set 
    $$\widetilde{\Omega} := \bigcap_{m=1}^{\infty} (\Omega_m \cap \Omega').
    $$
    It is straightforward to verify that $\mathbb{P}(\widetilde{\Omega}) = 1$.
    For $\omega \in \widetilde{\Omega}$, there exists $\bar{m} = \bar{m}(\omega)$ such that for all $m \geq \bar{m}(\omega)$, $\tau_m(\omega) \geq T \geq \tau(\omega)$.
    Consequently, for all $\omega \in \widetilde{\Omega}$,
    $$
    |u(\omega,\cdot)|_{C_{t,x}^{\alpha-\frac{1}{q},\phi_{\gamma-2\beta}-\frac{d}{p}}([0,\tau(\omega)]\times\bR^d)}<\infty.
    $$
    The theorem is proved.
\end{proof}

\subsection{Existence, Uniqueness, and Regularity of local solutions: Proof of Theorem \ref{local}-(i)}
\label{24.03.05.14.26}
Theorem \ref{local}-(i) is established through a direct application of the following corollary derived from Theorem \ref{23.07.04.14.29}.
\begin{corollary}
\label{23.07.04.17.12}
    Let $\tau\leq T$ be a bounded stopping time, $\delta_0\in(0,1]$, $\gamma\in(0,1)$ and $r\in[1,\infty)$.
    Suppose that Assumptions  \ref{main_ass_corr} $(\delta_0,\gamma,r)$, \ref{24.03.14.11.58} ($\delta_0$) and \ref{ass_coeff}($\delta_0,\gamma$) hold.
    Suppose also that
    \begin{enumerate}[(a)]
        \item $(F,B,\varphi)=(F(\omega,t,x,u),B(\omega,t,x,u),\varphi(\omega,t,x,u))$ are $\cP\times\cB(\bR^d)\times\cB(\bR)$-measurable functions.
        \item $F(\omega,t,x,0),\nabla_xB(\omega,t,x,0),\varphi(\omega,t,x,0)\in\bL_{p,q}(\tau)$,
        \item there exists a constant $K'>0$ such that for any $(\omega,t)\in \opar0,\tau \cbrk$, $x\in\bR^d$ and $u,v\in\bR$,
    \begin{align*}
        &|F(\omega,t,x,u)-F(\omega,t,x,v)|+|B(\omega,t,x,u)-B(\omega,t,x,v)|+|\varphi(\omega,t,x,u)-\varphi(\omega,t,x,v)|\\
        &\leq K'|u-v|.
    \end{align*}
    \end{enumerate}
    Then the equation
    \begin{equation*}
    \mathrm{d}u=(\phi(\Delta)u+\zeta F(u)+\vec{b}\cdot\nabla_{x}(B(u)))\,\mathrm{d}t+\sum_{k=1}^{\infty}\xi\varphi(u)(\pi\ast e_k)\mathrm{d}w_t^k, \quad (t,x)\in(0,\tau]\times\bR^d,
    \end{equation*}
    with initial data $u_0$ has a unique solution $u\in \cH_{p,q}^{\phi,\gamma}(\tau)$ with estimate
    $$
    \|u\|_{\cH_{p,q}^{\phi,\gamma}(\tau)}\leq N(\|u_0\|_{U_{p,q}^{\phi,\gamma}}+\||F(0)|+|\nabla_{x} B(0)|+|\varphi(0)|\|_{\bL_{p,q}(\tau)}),
    $$
where $N=N(K,K',N_0,N_1,N_2,\nu_{\delta_0,\gamma,1},p,q)$.
\end{corollary}
\begin{proof}
To prove the assertion, we apply Theorem \ref{23.07.04.14.29}.
To this end, it suffices to establish the following estimate, which corresponds to the assumption of Theorem \ref{23.07.04.14.29}:
    \begin{align*}
    &\|\zeta(F(u)-F(v))+\vec{b}\cdot\nabla_x(B(u)-B(v))\|_{H_p^{\phi,\gamma-2}(\bR^d)}+\|\xi\boldsymbol{\pi}(\varphi(u)-\varphi(v))\|_{H_p^{\phi,\gamma-1}(\bR^d;\ell_2)}\\
    &\leq N\|u-v\|_{L_p(\bR^d)},
    \end{align*}
    where $\boldsymbol{\pi} = (\pi * e_1, \pi * e_2, \ldots)$.
    
    Assuming the condition \eqref{23.06.28.17.44} holds and for $\gamma \leq 0$, we have
\begin{equation}
    \label{23.06.29.12.40}
    \|u\|_{H_{p}^{\phi,\gamma}(\bR^d)}\leq N\|u\|_{H_{p}^{\delta_0\gamma}(\bR^d)}.
\end{equation}
Hence,
\begin{align*}
    &\|\zeta(F(u)-F(v))+\vec{b}\cdot\nabla_x(B(u)-B(v))\|_{H_p^{\phi,\gamma-2}(\bR^d)}+\|\xi\boldsymbol{\pi}(\varphi(u)-\varphi(v))\|_{H_p^{\phi,\gamma-1}(\bR^d;\ell_2)}\\
    &\leq \|\zeta(F(u)-F(v))+\vec{b}\cdot\nabla_x(B(u)-B(v))\|_{H_p^{\delta_0(\gamma-2)}(\bR^d)}+\|\xi\boldsymbol{\pi}(\varphi(u)-\varphi(v))\|_{H_p^{\delta_0(\gamma-1)}(\bR^d;\ell_2)}.
\end{align*}
Since $\|R_{\beta}\|_{L_1(\mathbb{R}^d)}=1$ for all $\beta>0$ (see \eqref{24.03.06.12.01}),
by Young’s convolution inequality and  $\gamma \in (0,1)$, we obtain
\begin{equation}
\label{24.03.24.12.44}
    \begin{aligned}
        &\|\zeta(F(u)-F(v))\|_{H_p^{\delta_0(\gamma-2)}(\bR^d)}^p\\
        &=\int_{\bR^d}|(R_{\delta_0(2-\gamma)}\ast(\zeta(F(u)-F(v))))(x)|^p\mathrm{d}x\\
        &\leq K' \int_{\bR^d}\left(\int_{\bR^d}R_{\delta_0(2-\gamma)}(x-y)|u(y)-v(y)||\zeta(t,y)|\mathrm{d}y\right)^p\mathrm{d}x\\
        &\leq N(K,K',p)\|R_{\delta_0(2-\gamma)}\|_{L_1(\bR^d)}^p\|u-v\|_{L_p(\bR^d)}^p\\
        &=N(K,K',p)\|u-v\|_{L_p(\bR^d)}^p,
    \end{aligned}
\end{equation}
where $R_{\delta_0(2-\gamma)}$ is defined in \eqref{24.03.04.15.14}. 
Similarly, we also have
\begin{equation}
\label{24.03.24.12.45}
\begin{aligned}
    \|\vec{b}\cdot\nabla_x(B(u)-B(v))\|_{{H_p^{\delta_0(\gamma-2)}}(\bR^d)}&\leq N\|B(u)-B(v)\|_{{H_p^{\delta_0(\gamma-2)+1}}(\bR^d)}\\
    &\leq N(K,K',p)\|R_{\delta_0(2-\gamma)-1}\|_{L_1(\bR^d)}\|u-v\|_{L_p(\bR^d)}\\
    &\leq N(K,K',p)\|u-v\|_{L_p(\bR^d)}.
\end{aligned}
\end{equation}
Here, the condition $\delta_0(2 - \gamma) > 1$ is ensured by Assumption \ref{ass_coeff}($\delta_0, \gamma$).
For the inequality
\begin{equation}
\label{24.03.24.12.46}
\|\xi\boldsymbol{\pi}(\varphi(u)-\varphi(v))\|_{H_p^{\delta_0(\gamma-1)}(\bR^d;\ell_2)}\leq N(K,K',\nu_{\delta_0,\gamma,1})\|u-v\|_{L_p(\bR^d)},
\end{equation}
it is essential to demonstrate that
\begin{align}
\label{24.03.04.14.29}
\|\xi u \boldsymbol{\pi}\|_{H_p^{\delta_0(\gamma-1)}(\bR^d;\ell_2)}\leq N(K,K',\nu_{\delta_0,\gamma,1})\|u\|_{L_p(\bR^d)}.
\end{align}
Here, the constant $\nu_{\delta_0, \gamma, 1}$ is given by
$$
\nu_{\delta_0, \gamma, 1} := \int_{\mathbb{R}^d} R_{2\delta_0(1 - \gamma)}(x) \, \pi(\mathrm{d}x),
$$
which is finite by Proposition \ref{24.04.21.21.16}.
The validity of estimate \eqref{24.03.04.14.29} is confirmed by references to \cite[Lemma 4.4]{CH2021} and \cite[Lemma 3]{FS2006}.
By amalgamating \eqref{24.03.24.12.44}, \eqref{24.03.24.12.45}, and \eqref{24.03.24.12.46}, the corollary's proof is concluded.
\end{proof}

\begin{proof}[Proof of Theorem \ref{local}-$(i)$]
    Drawing upon Assumption \ref{ass_nl_term}, we have the inequalities
    $$
    |F_m(u)-F_m(v)|+|B_m(u)-B_m(v)|+|\varphi_m(u)-\varphi_m(v)|\leq N(m,C_m)|u-v|.
    $$
    Hence, by applying Corollary \ref{23.07.04.17.12}, it is established that for any bounded stopping time $\tau$, there exists a unique solution $u_m^{\tau} \in \cH_{p,q}^{\phi,\gamma}(\tau)$.
    Thus, the theorem is conclusively proved.
\end{proof}

\subsection{Nonnegativity of local solutions: Proof of Theorem \ref{local}-(ii)}
\label{24.03.24.11.02}
To establish Theorem \ref{local}-(ii), which concerns the nonnegativity of local solutions, we employ the maximum principle applicable to semilinear stochastic integro-differential equations represented as
        $$
        \mathrm{d}u=(\phi(\Delta)u+f(u))\mathrm{d}t+\sum_{k=1}^{\infty}g^k(u)\mathrm{d}w^k_t.
        $$
Our approach is guided by the proof of Krylov's maximum principle as detailed in \cite[Theorem 2.5]{Kry2007} and \cite[Theorem 3.1]{Kry2013}.
To use Krylov's argument, we need the following lemma.

\begin{lem}\label{lem 03.30.10:44}
Let $u\in L_2(\bR^d)$.
\begin{enumerate}[(i)]
    \item Then, $u\in H_2^{\phi,1}(\bR^d)$ if and only if
\begin{align}\label{eqn 04.07.20:14}
\|\phi^{1/2}(\Delta)u\|^{2}_{L_{2}(\bR^{d})} = b_{\phi}\|\nabla u\|_{L^2(\mathbb{R}^d)}^2+\frac{1}{2}\int_{\bR^{d}} \int_{\bR^{d}} \left| u(x)-u(y) \right|^{2} j(|x-y|) \mathrm{d}x\mathrm{d}y<\infty,
\end{align}
where $b_{\phi}$ is the drift of $\phi$, and $j$ is defined in \eqref{eqn 03.30.10:15}.
    \item Suppose that $u\in H^{\phi,1}_{2}(\bR^{d})$.
    Then, for any Lipschitz continuous function $r$ with $r(0)=0$, $r(u)\in H^{\phi,1}_{2}(\bR^{d})$, and
\begin{equation}
\label{24.04.07.23.05}
    \|\phi^{1/2}(\Delta)(r(u))\|_{L_{2}(\bR^{d})} \leq  \|\phi^{1/2}(\Delta)u\|_{L_{2}(\bR^{d})}.
\end{equation}
\item Suppose that $u,v\in H_2^{\phi,1}(\bR^d)$.
Then for any defining sequence $\{u_n\}_{n=1}^{\infty}\subseteq \cS(\bR^d)$, the sesquilinear functional
$$
\mathcal{I}(u_n,v):= \int_{\bR^d}\overline{v(x)}\phi(\Delta)u_n(x)\mathrm{d}x
$$
converges to
$$
b_{\phi}\int_{\mathbb{R}^d}\overline{\nabla v(x)}\nabla u(x)\mathrm{d}x-\frac{1}{2} \int_{\bR^{d}} \int_{\bR^{d}} \left( \overline{v(x)} - \overline{v(y)} \right)  \left( u(x)-u(y) \right) j(|x-y|) \mathrm{d}x\mathrm{d}y.
$$
\end{enumerate}
\end{lem}
\begin{proof}
$(i)$ Since $\cS(\bR^{d})$ is dense in $H^{\phi,1}_{2}(\bR^{d})$, and 
$$
\|u\|^{2}_{H^{\phi,1}_{2}(\bR^{d})}\simeq \|u\|^{2}_{L_{2}(\bR^{d})} + \|\phi^{1/2}(\Delta)u\|^{2}_{L_{2}(\bR^{d})},
$$
it suffices only to prove for $u\in \cS(\bR^d)$, \eqref{eqn 04.07.20:14} holds.
Using Plancherel's theorem and definition of $\phi^{1/2}(\Delta)$, we have
\begin{align*}
\int_{\bR^{d}} \left| \phi^{1/2}(\Delta)u(x) \right|^{2}\mathrm{d}x &= \int_{\bR^{d}} \phi(|\xi|^{2}) \overline{\mathcal{F}[u](\xi)} \mathcal{F}[u](\xi) \mathrm{d}\xi
\\
&= \int_{\bR^{d}} \left(b_{\phi}|\xi|^2+\int_{\bR^{d}} \left( 1-\cos{(\xi \cdot y)} \right) j(|y|) \mathrm{d}y\right)  \overline{\mathcal{F}[u](\xi)}  \mathcal{F}[u](\xi)  \mathrm{d}\xi,
\end{align*}
where the last equality follows from \eqref{eqn 02.01.13:39}, \eqref{eqn 03.30.10:39} and \eqref{23.04.25.17.19}.
Employing the relation $(1-\cos{z}) = |1-\mathrm{e}^{\mathrm{i} z}|^{2}/2$, and applying Plancherel's theorem again, we achieve
\begin{align*}
&\int_{\bR^{d}} \left| \phi^{1/2}(\Delta)u(x) \right|^{2}\mathrm{d}x \\
&= \int_{\bR^{d}} \int_{\bR^{d}} \frac{1}{2} \left| 1- \mathrm{e}^{\mathrm{i} \xi \cdot y} \right|^{2}\overline{\mathcal{F}[u](\xi)}  \mathcal{F}[u](\xi) j(|y|)\mathrm{d}\xi \mathrm{d}y+b_{\phi}\int_{\mathbb{R}^d}|\xi|^2||\mathcal{F}[u](\xi)|^2\mathrm{d}\xi
\\
&= \frac{1}{2} \int_{\bR^{d}} \int_{\bR^{d}} \left( 1-\mathrm{e}^{-\mathrm{i} \xi \cdot y}  \right)\overline{\mathcal{F}[u](\xi)}   \left( 1-\mathrm{e}^{\mathrm{i} \xi \cdot y}  \right)\mathcal{F}[u](\xi)   j(|y|)\mathrm{d}\xi \mathrm{d}y+b_{\phi}\int_{\mathbb{R}^d}|\mathrm{i}\xi\mathcal{F}[u](\xi)|^2\mathrm{d}\xi
\\
&= \frac{1}{2} \int_{\bR^{d}} \int_{\bR^{d}} \left| u(x)-u(x-y) \right|^{2} j(|y|)\mathrm{d}x\mathrm{d}y+b_{\phi}\int_{\mathbb{R}^d}|\nabla u(x)|^2\mathrm{d}x
\\
&=  \frac{1}{2} \int_{\bR^{d}} \int_{\bR^{d}} \left| u(x)-u(y) \right|^{2} j(|x-y|) \mathrm{d}x\mathrm{d}y+b_{\phi}\|\nabla u\|_{L_2(\mathbb{R}^d)}^2.
\end{align*}

$(ii)$ Since $u\in H^{\phi,1}_{2}(\bR^{d}) \subset L_{2}(\bR^{d})$, it can be easily checked that $r(u)\in L_2(\bR^d)$.
Due to $(i)$, we only need to prove \eqref{24.04.07.23.05}.
Since $|r(u(x)) - r(u(y))| \leq N|u(x) -u(y)|$ for all $x,y\in\bR^{d}$, we have
\begin{equation*}
\begin{aligned}
&\int_{\bR^{d}} \int_{\bR^{d}} \left| r(u(x))-r(u(y)) \right|^{2} j(|x-y|)\mathrm{d}x\mathrm{d}y+b_{\phi}\|\nabla(r(u))\|_{L_2(\mathbb{R}^d)}^2 \\
&\leq N\int_{\bR^{d}} \int_{\bR^{d}} \left| u(x)-u(y) \right|^{2} j(|x-y|)\mathrm{d}x\mathrm{d}y+N\|\nabla u\|_{L_2(\mathbb{R}^d)}^2\\
&=N\|\phi^{1/2}(\Delta)u\|_{L_{2}(\bR^{d})}^{2}<\infty.
\end{aligned}
\end{equation*}

$(iii)$ Let $\{u_{n}\}_{n=1}^{\infty}\subseteq \mathcal{S}(\bR^{d})$ be a sequence of functions satisfying $u_n\to u$ in $H^{\phi,1}_{2}(\bR^{d})$.
Then, $\mathcal{I}(u_{n},v)$ is well-defined.
In addition, by following the argument in $(i)$, we have
\begin{align*}
\mathcal{I}(u_{n},v) =& - \int_{\bR^{d}} \phi(|\xi|^{2}) \overline{\mathcal{F}[v](\xi)} \mathcal{F}[u_{n}](\xi) \mathrm{d}\xi 
\\
 =& -\frac{1}{2} \int_{\bR^{d}}\int_{\bR^{d}} \left( \overline{v(x)} - \overline{v(y)}  \right) \left( u_{n}(x)-u_{n}(y) \right) j(|x-y|) \mathrm{d}x\mathrm{d}y\\
&-b_{\phi}\int_{\mathbb{R}^d}\overline{\nabla v(x)}\nabla u_n(x)\mathrm{d}x.
\end{align*}
By Lemma \ref{lem 03.30.10:44}-$(i)$ and the Cauchy-Bunyakovsky-Schwarz inequality,
\begin{align*}
&\left|\cI(u_n,v)+\frac{1}{2} \int_{\bR^{d}}\int_{\bR^{d}} \left( \overline{v(x)} - \overline{v(y)}  \right) \left( u(x)-u(y) \right) j(|x-y|) \mathrm{d}x\mathrm{d}y+b_{\phi}\int_{\mathbb{R}^d}\overline{\nabla v(x)}\cdot\nabla u(x)\mathrm{d}x\right|\\
&\leq N\|\phi^{1/2}(\Delta)v\|_{L_2(\bR^d)}\|\phi^{1/2}(\Delta)(u_n-u)\|_{L_2(\bR^d)}\to0
\end{align*}
as $n\to\infty$.
Therefore, by defining $\mathcal{I}(u,v)$ as the limit of $\mathcal{I}(u_{n},v)$, we prove the lemma.
\end{proof}

\begin{rem}
\label{24.04.09.11.39}
    Due to Lemma \ref{lem 03.30.10:44}-($iii$), for $u,v\in H_2^{\phi,1}(\bR^d)$, the term
    $$
    \int_{\bR^d}\overline{v(x)}\phi(\Delta)u(x)\mathrm{d}x
    $$
    can be interpreted as the limit of $\cI(u_n,v)$,
    $$
    -\frac{1}{2} \int_{\bR^{d}}\int_{\bR^{d}} \left( \overline{v(x)} - \overline{v(y)}  \right) \left( u(x)-u(y) \right) j(|x-y|) \mathrm{d}x\mathrm{d}y-b_{\phi}\int_{\mathbb{R}^d}\overline{\nabla v(x)}\cdot\nabla u(x)\mathrm{d}x.
    $$
\end{rem}

Now, we prove the maximum principle for semilinear stochastic integro-differential equations.
\begin{thm}[Maximum principle]
\label{max_prin}
    Let $\tau\leq T$ be a bounded stopping time.
    Suppose that
    \begin{enumerate}[(i)]
        \item $u=u(t,\cdot)$ and $f = f (t, \cdot)$ are $L_2(\bR^d)$-valued $\bF_t$-adapted, jointly measurable functions.
        \item $g = g(t,\cdot)=(g^1(t,\cdot),g^2(t,\cdot),\cdots)$ is a $L_2(\bR^d;\ell_2)$-valued $\bF_t$-adapted, jointly measurable function.
        \item $u,f,g$ satisfy
        $$
        \mathbb{E}\left[\|u(0,\cdot)\|_{L_2(\bR^d)}^2+\int_0^{\tau}\||f(s,\cdot)|+|g(s,\cdot)|_{\ell_2}\|_{L_2(\bR^d)}^2\mathrm{d}s\right]<\infty,
        $$
        and
        $$
        \mathrm{d}u=(\phi(\Delta)u+f)\mathrm{d}t+\sum_{k=1}^{\infty}g^k\mathrm{d}w^k_t.
        $$
    \end{enumerate}
    Suppose also that $u(0,\cdot)\leq0$ and for any $\omega\in\Omega$, $k=1,2,\cdots$
    \begin{equation}
   \label{24.03.31.13.07}
    1_{u(t,\cdot)>0}f(t,\cdot)\leq0,\quad 1_{u(t,\cdot)>0}g^k(t,\cdot)=0,
    \end{equation}
    for almost $t\in(0,\tau(\omega))$.
    Then almost surely $\omega\in\Omega$, $u(t,\cdot)\leq0$ for all $t\in[0,\tau(\omega)]$.
\end{thm}
\begin{proof}
    The aim of this proof is to establish that, for almost every $\omega \in \Omega$, we have
    \begin{equation}
    \label{24.03.31.13.12}
    \int_{\bR^d}|u_{+}(t\wedge \tau,x)|^2\mathrm{d}x\leq 0,
    \end{equation}
    where $u_{+}(t, x) := \max(u(t, x), 0)$.
    
    \textbf{Step 1.} In this step, we prove $u\in \cH_{2,2}^{\phi,1}(\tau)$ and 
    \begin{equation}
    \label{24.04.18.15.09}
    \mathbb{E}\left[\sup_{t\leq\tau}\|u(t,\cdot)\|_{L_2(\bR^d)}^2\right]\leq N\mathbb{E}\left[\|u(0,\cdot)\|_{L_2(\bR^d)}^2+\int_0^{\tau}\||f(s,\cdot)|+|g(s,\cdot)|_{\ell_2}\|_{L_2(\bR^d)}^2\mathrm{d}s\right].
    \end{equation}
    One can check that $f\in \bL_{2,2}(\tau)\subseteq \bH_{2,2}^{\phi,-1}(\tau)$ and $(H_2^{\phi,1}(\bR^d),H_2^{\phi,-1}(\bR^d))_{1/2,2}=L_2(\bR^d)$, thus, by Theorem \ref{23.07.04.14.29}, $u\in \cH_{2,2}^{\phi,1}(\tau)$.
    Consider a nonnegative function $\zeta \in C_c^{\infty}(\mathbb{R}^d)$ with $\int_{\mathbb{R}^d} \zeta(z)  \mathrm{d}z = 1$, and define
    $$
    \zeta_{\varepsilon}(x):=\varepsilon^{-d}\zeta(x/\varepsilon).
    $$
    For $\chi = u, f, g^k$, let
    \begin{align*}
    \chi_{\varepsilon}(t,x):=\int_{\bR^d}\chi(t,y)\zeta_{\varepsilon}(x-y)\mathrm{d}y,
    \end{align*}
    then for any $x\in\bR^d$ the equality
    $$
    u_{\varepsilon}(t,x)=u_{\varepsilon}(0,x)+\int_0^t(\phi(\Delta)u_{\varepsilon}(s,x)+f_{\varepsilon}(s,x))\mathrm{d}s+\sum_{k=1}^{\infty}\int_0^tg^k_{\varepsilon}(s,x)\mathrm{d}w_s^k
    $$
    holds almost surely for all $t\leq\tau$.
    By It\^o's formula,
    \begin{align*}
        &\mathrm{e}^{-\lambda t}\int_{\bR^d}|u_{\varepsilon}(t,x)|^2\mathrm{d}x\\
        &=\int_{\bR^d}|u_{\varepsilon}(0,x)|^2\mathrm{d}x+2\int_0^t\int_{\bR^d}u_{\varepsilon}(s,x)(\phi(\Delta)u_{\varepsilon}(s,x)+f_{\varepsilon}(s,x))\mathrm{e}^{-\lambda s}\mathrm{d}x\mathrm{d}s\\
        &\quad +\int_0^{t}\int_{\bR^d}|g_{\varepsilon}(s,x)|^2_{\ell_2}\mathrm{e}^{-\lambda s}\mathrm{d}x\mathrm{d}s+2\sum_{k=1}^{\infty}\int_0^t\int_{\bR^d}u_{\varepsilon}(s,x)g_{\varepsilon}^k(s,x)\mathrm{e}^{-\lambda s}\mathrm{d}x\mathrm{d}w_{s}^k\\
        &\quad-\lambda\int_{0}^t\|u_{\varepsilon}(s,\cdot)\|_{L_2(\bR^d)}^2\mathrm{e}^{-\lambda s}\mathrm{d}s.
    \end{align*}
    By Young's inequality,
    $$
    2\int_{\bR^d}u_{\varepsilon}(t,x)f_{\varepsilon}(t,x)\mathrm{d}x\leq \|u(t,\cdot)\|_{L_2(\bR^d)}^2+\|f(t,\cdot)\|_{L_2(\bR^d)}^2.
    $$
    By Lemma \ref{lem 03.30.10:44}-$(i)$ and Remark \ref{24.04.09.11.39},
    \begin{equation*}
    \begin{gathered}
        \int_{\bR^d}u_{\varepsilon}(s,x)\phi(\Delta)u_{\varepsilon}(s,x)\mathrm{d}x=-\|\phi^{1/2}(\Delta)u_{\varepsilon}(s,\cdot)\|_{L_2(\bR^d)}^2\leq0.
    \end{gathered}
    \end{equation*}
By the Burkholder-Davis-Gundy inequality, H\"older's inequality and Young's inequality,
\begin{align*}
    &2\mathbb{E}\left[\sup_{t\leq\tau}\left|\sum_{k=1}^{\infty}\int_0^t\int_{\bR^d}u_{\varepsilon}(s,x)g_{\varepsilon}^k(s,x)\mathrm{e}^{-\lambda s}\mathrm{d}x\mathrm{d}w_{s}^k\right|\right]\\
    &\leq N\mathbb{E}\left[\left(\int_0^{\tau}\|u_{\varepsilon}(s,\cdot)\|_{L_2(\bR^d)}^2\|g_{\varepsilon}(s,\cdot)\|_{L_2(\bR^d;\ell_2)}^2\mathrm{e}^{-2\lambda s}\mathrm{d}s\right)^{1/2}\right]\\
    &\leq N\left(\mathbb{E}\left[\sup_{t\leq\tau}\mathrm{e}^{-\lambda t}\|u_{\varepsilon}(t,\cdot)\|_{L_2(\bR^d)}^2\right]\right)^{1/2}\|g_{\varepsilon}\|_{\bL_{2,2}(\tau)}\\
    &\leq \delta \mathbb{E}\left[\sup_{t\leq\tau}\mathrm{e}^{-\lambda t}\|u_{\varepsilon}(t,\cdot)\|_{L_2(\bR^d)}^2\right]+N_{\delta}\|g_{\varepsilon}\|_{\bL_{2,2}(\tau)}^2.
\end{align*}
By taking $\lambda\geq1$ and sufficiently small $\delta>0$, we have
\begin{equation}
\label{24.04.23.14.15}
    \begin{aligned}
        &\mathbb{E}\left[\sup_{t\leq\tau}\|u_{\varepsilon}(t,\cdot)\|_{L_2(\bR^d)}^2\right]\\
        &\leq N\mathbb{E}\left[\|u_{\varepsilon}(0,\cdot)\|_{L_2(\bR^d)}^2+\int_0^{\tau}\||f_{\varepsilon}(s,\cdot)|+|g_{\varepsilon}(s,\cdot)|_{\ell_2}\|_{L_2(\bR^d)}^2\mathrm{d}s\right]\\
        &\leq N\mathbb{E}\left[\|u(0,\cdot)\|_{L_2(\bR^d)}^2+\int_0^{\tau}\||f(s,\cdot)|+|g(s,\cdot)|_{\ell_2}\|_{L_2(\bR^d)}^2\mathrm{d}s\right],
    \end{aligned}
\end{equation}
where $N=N(d,\delta,\lambda,T)$.
For achieving \eqref{24.04.18.15.09}, the remaining part of step 1 is only to prove
\begin{equation}
\label{24.04.23.15.38}
\lim_{m\to\infty}\mathbb{E}\left[\sup_{t\leq\tau}\|u_{1/m}(t,\cdot)-u(t,\cdot)\|_{L_2(\bR^d)}^2\right]=0.
\end{equation}
By the dominated convergence theorem,
$$
\mathbb{E}\left[\|u_{\varepsilon_1}(0,\cdot)-u_{\varepsilon_2}(0,\cdot)\|_{L_2(\bR^d)}^2+\int_0^{\tau}\||f_{\varepsilon_1}(s,\cdot)-f_{\varepsilon_2}(s,\cdot)|+|g_{\varepsilon_1}(s,\cdot)-g_{\varepsilon_2}(s,\cdot)|_{\ell_2}\|_{L_2(\bR^d)}^2\mathrm{d}s\right]
$$
converges to $0$ as $\varepsilon_1,\varepsilon_2\downarrow0$.
Due to \eqref{24.04.23.14.15}, it can be easily checked that $\{u_{1/m}\}_{m=1}^{\infty}$ is a Cauchy sequence in $L_2(\Omega;L_{\infty}([0,\tau];L_2(\bR^d)))$.
Therefore, there exists a limit $\bar{u}$ such that $u_{1/m}\to \bar{u}$ in $L_2(\Omega;L_{\infty}([0,\tau];L_2(\bR^d)))$.
One can also easily check that $u_{1/m}(t,\cdot)\to u(t,\cdot)$ in the sense of distribution for all $t\leq \tau$ almost surely.
Therefore, with probability $1$, $u(t,\cdot)\in L_2(\bR^d)$ for all $t\leq\tau$.
This certainly implies that $\bar{u}(t,\cdot)=u(t,\cdot)$ in $L_2(\bR^d)$ for all $t\leq\tau$, thus \eqref{24.04.23.15.38} holds.

    \textbf{Step 2.} Now, we prove \eqref{24.03.31.13.12}.
    Let $r(z) := (\max(z, 0))^2$ for $z \in \mathbb{R}$.
    Drawing upon \cite[Remark 3.1]{Kry2007}, we introduce a sequence of real-valued convex functions $\{r_n\}_{n \in \mathbb{N}}$ that satisfy:
    \begin{itemize}
        \item $r_n$ is infinitely differentiable.
        \item $|r_n(z)|\leq N|z|^2$, $|r_n'(z)|\leq N|z|$, and $|r_n''(z)|\leq N$, where $N$ is independent of $z$ and $n$.
        \item $r_n,r_n',r_n''\to r,r',r''$ as $n\to\infty$.
        Here $r''(z)=2\times1_{z>0}$.
    \end{itemize}
    
    Applying It\^o's formula, we have
    \begin{align*}
        \int_{\bR^d}r_n(u_{1/m}(\tau\wedge t,x))\mathrm{d}x&=\int_{\bR^d}r_n(u_{1/m}(0,x))\mathrm{d}x\\
        &\quad+\int_0^{\tau\wedge t}\int_{\bR^d}r_n'(u_{1/m}(s,x))(\phi(\Delta)u_{1/m}(s,x)+f_{1/m}(s,x))\mathrm{d}x\mathrm{d}s\\
        &\quad+\frac{1}{2}\int_0^{\tau\wedge t}\int_{\bR^d}r_n''(u_{1/m}(s,x))|g_{1/m}(s,x)|_{\ell_2}^2\mathrm{d}x\mathrm{d}s\\
        &\quad+\sum_{k=1}^{\infty}\int_0^{\tau\wedge t}\int_{\bR^d}r_{n}'(u_{1/m}(s,x))g^k_{1/m}(s,x)\mathrm{d}x\mathrm{d}w_{s}^k.
    \end{align*}
    Due to \eqref{24.04.23.15.38},
    \begin{equation*}
        \begin{aligned}
            &\mathbb{E}\left[\sup_{s\leq\tau}\int_{\bR^d}|r_n(u_{1/m}(s,x))-r_n(u(s,x))|\mathrm{d}x\right]\\
            &\leq N\mathbb{E}\left[\sup_{s\leq\tau}\|u(s,\cdot)\|_{L_2(\bR^d)}\|u_{1/m}(s,\cdot)-u(s,\cdot)\|_{L_2(\bR^d)}\right]\to0
        \end{aligned}
    \end{equation*}
    and
    \begin{align*}
        &\mathbb{E}\left[\sup_{s\leq\tau}\|r_n'(u_{1/m}(s,\cdot))-r_n'(u(s,\cdot))\|_{L_2(\bR^d)}\right]\\
        &\leq N\mathbb{E}\left[\sup_{s\leq\tau}\|u_{1/m}(s,\cdot)-u(s,\cdot)\|_{L_2(\bR^d)}\right]\to 0,
    \end{align*}
    as $m\to\infty$.
    Then, with probability 1, the equality
    \begin{align*}
        \int_{\bR^d}r_n(u(\tau\wedge t,x))\mathrm{d}x&=\int_{\bR^d}r_n(u(0,x))\mathrm{d}x\\
        &\quad+\int_0^{\tau\wedge t}\int_{\bR^d}r_n'(u(s,x))(\phi(\Delta)u(s,x)+f(s,x))\mathrm{d}x\mathrm{d}s\\
        &\quad+\frac{1}{2}\int_0^{\tau\wedge t}\int_{\bR^d}r_n''(u(s,x))|g(s,x)|_{\ell_2}^2\mathrm{d}x\mathrm{d}s\\
        &\quad+\sum_{k=1}^{\infty}\int_0^{\tau\wedge t}\int_{\bR^d}r_{n}'(u(s,x))g^k(s,x)\mathrm{d}x\mathrm{d}w_{s}^k
    \end{align*} 
    holds for all $t\leq\tau$.
    Taking the expectation both-sides, we have
    \begin{align*}
        \mathbb{E}\left[\int_{\bR^d}r_n(u(\tau\wedge t,x))\mathrm{d}x\right]&=\mathbb{E}\left[\int_{\bR^d}r_n(u(0,x))\mathrm{d}x\right]\\
        &\quad+\mathbb{E}\left[\int_0^{\tau\wedge t}\int_{\bR^d}r_n'(u(s,x))(\phi(\Delta)u(s,x)+f(s,x))\mathrm{d}x\mathrm{d}s\right]\\
        &\quad+\mathbb{E}\left[\frac{1}{2}\int_0^{\tau\wedge t}\int_{\bR^d}r_n''(u(s,x))|g(s,x)|_{\ell_2}^2\mathrm{d}x\mathrm{d}s\right].
    \end{align*} 
    By the dominated convergence theorem and \eqref{24.03.31.13.07}, we infer that
    \begin{equation}
    \label{24.03.31.13.00}
    \begin{aligned}
        \mathbb{E}\left[\int_{\bR^d}|u_+(\tau\wedge t,x)|^2\mathrm{d}x\right]&\leq\mathbb{E}\left[\int_{\bR^d}|u_+(0,x)|^2\mathrm{d}x\right]\\
        &\quad+2\mathbb{E}\left[\int_0^{\tau\wedge t}\int_{\bR^d}u_+(s,x)(\phi(\Delta)u(s,x)+f(s,x))\mathrm{d}x\mathrm{d}s\right]\\
        &\quad+\mathbb{E}\left[\int_0^{\tau\wedge t}\int_{\bR^d}1_{u(s,x)>0}|g(s,x)|_{\ell_2}^2\mathrm{d}x\mathrm{d}s\right]\\
        &\leq 2\mathbb{E}\left[\int_0^{\tau\wedge t}\int_{\bR^d}u_+(s,x)\phi(\Delta)u(s,x)\mathrm{d}x\mathrm{d}s\right].
    \end{aligned}
    \end{equation}
By Lemma \ref{lem 03.30.10:44}-($ii$) and Remark \ref{24.04.09.11.39},
\begin{equation*}
\begin{aligned}
    &2\int_{\bR^d}u_+(s,x)\phi(\Delta)u(s,x)\mathrm{d}x\\
    &=-\int_{\bR^d}\int_{\bR^d}(u_+(s,x)-u_+(s,y))(u(s,x)-u(s,y))j(|x-y|)\mathrm{d}x\mathrm{d}y\\
    &\quad-2b_{\phi}\int_{\mathbb{R}^d}\nabla u_+(x)\nabla u(x)\mathrm{d}x.
\end{aligned}
\end{equation*}
Observe that
$$
\big(u_+(s,x) - u_+(s,y)\big)\big(u(s,x) - u(s,y)\big) \geq 0,
$$
and for almost every $x \in \mathbb{R}^d$,
$$
\nabla u_+(x) = 1_{\{u > 0\}}(x) \nabla u(x).
$$
It follows that
\begin{equation}
\label{24.03.31.13.11}
    2\int_{\mathbb{R}^d} u_+(s,x)\, \phi(\Delta) u(s,x)\, \mathrm{d}x \leq 0.
\end{equation}
Combining \eqref{24.03.31.13.00} and \eqref{24.03.31.13.11}, we obtain the desired inequality \eqref{24.03.31.13.12}.
This completes the proof of the theorem.
\end{proof}
\begin{rem}
Theorem \ref{max_prin} is stated as a \emph{nonpositivity} principle.
The \emph{nonnegativity} version is equivalent by a sign change.
Indeed, let $u$ solve
$$
\mathrm{d}u=(\phi(\Delta)u+f)\,\mathrm{d}t+\sum_{k=1}^{\infty} g^k\,\mathrm{d}w_t^k,
$$
and assume
$$
u(0,\cdot)\ge 0,\qquad 1_{\{u<0\}}\,f\ge 0,\qquad 1_{\{u<0\}}\,g^k=0.
$$
Set $v:=-u$. 
Then $v$ solves
$$
\mathrm{d}v=(\phi(\Delta)v-f)\,\mathrm{d}t-\sum_{k=1}^{\infty} g^k\,\mathrm{d}w_t^k,
$$
and satisfies
$$
v(0,\cdot)\le 0,\qquad 1_{\{v>0\}}(-f)= -1_{\{u<0\}}f\le 0,\qquad
1_{\{v>0\}}(-g^k)= -1_{\{u<0\}}g^k=0.
$$
Applying Theorem~\ref{max_prin} to $v$ yields $v\le 0$, hence $u\ge 0$. 
Thus, the nonpositivity and nonnegativity formulations are equivalent.
\end{rem}

\begin{rem}
    Theorem \ref{max_prin} says the \emph{nonpositivity} of the solution $u$.
    However, the \emph{nonpositivity principle} and the \emph{nonnegativity principle} are equivalent.
    Suppose that $u$ is a solution to
    $$
        \mathrm{d}u=(\phi(\Delta)u+f)\mathrm{d}t+\sum_{k=1}^{\infty}g^k\mathrm{d}w^k_t.
        $$
    Suppose also that
    $$
    u(0,\cdot)\geq 0,\quad 1_{u(t,\cdot)<0}f(t,\cdot)\geq0,\quad 1_{u(t,\cdot)<0}g^k(t,\cdot)=0.
    $$
    Then $v:=-u$ is a solution to
    $$
        \mathrm{d}v=(\phi(\Delta)v-f)\mathrm{d}t-\sum_{k=1}^{\infty}g^k\mathrm{d}w^k_t.
        $$
        Then
$$
v(0,\cdot)=-u(0,\cdot)\leq0,\quad 1_{v(t,\cdot)>0}(-f(t,\cdot))=-1_{u(t,\cdot)<0}f(t,\cdot)\leq0,
$$
and
$$
1_{v(t,\cdot)>0}(-g^k(t,\cdot))=-1_{u(t,\cdot)<0}g^k(t,\cdot)=0.
$$
By Theorem \ref{max_prin}, we have $-u=v\leq 0$, thus $u\geq0$.
\end{rem}

\begin{proof}[Proof of Theorem \ref{local}-$(ii)$]
    Consider $\eta, \rho \in C_c^{\infty}(\mathbb{R})$ to be nonnegative functions satisfying 
    $$
    \int_{\bR}\eta(x)\mathrm{d}x=1,\quad \rho(0)=1.
    $$
    Define $u_{0,n} := (u_0 * \eta_n) \rho_n$, where $\eta_n(x) := n^d \eta(n|x|)$ and $\rho_n(x) := \rho(|x|/n)$.
    Thus, $u_{0,n} \in L_2(\Omega \times \mathbb{R}^d) \cap U_{p,q}^{\phi,\gamma}$, and it follows that
    \begin{equation}
    \label{24.03.24.14.54}
        \lim_{n\to\infty}\|u_{0,n}-u_0\|_{U_{p,q}^{\phi,\gamma}}=0.
    \end{equation}
    According to Theorems \ref{23.07.04.14.29} and \ref{local}-($i$), there exists a unique solution $u_m^n \in \cH_{2,2}^{\phi,1}(\tau) \cap \cH_{p,q}^{\phi,\gamma}(\tau)$ for the equation
    $$
    \mathrm{d}u_m^n=(\phi(\Delta)u_m^n+\zeta F_m(u_m^n)+\vec{b}\cdot\nabla_x(B_m(u_m^n)))\mathrm{d}t+\sum_{k=1}^n\xi\varphi_m(u_m^n)(\pi\ast e_k)\mathrm{d}w_t^k,\quad t\in(0,\tau],
    $$
    with initial condition $u_m^n(0, \cdot) = u_{0,n}$.
    It is verifiable that
    $$
    f_m^n:=(\zeta F_m(u_m^n)+\vec{b}\cdot\nabla_x(B_m(u_m^n)))1_{\cpar0,\tau\cbrk}\in L_2(\Omega\times(0,\infty)\times\bR^d),
    $$
    and $f_m^n\leq 0$ on $\{(\omega,t,x):u_m^n(\omega,t,x)\leq0\}$.
    Similarly, 
    $$
    g_m^n:=\xi\varphi_m(u_m^n)\boldsymbol{\pi}_n1_{\cpar0,\tau\cbrk}\in L_2(\Omega\times(0,\infty)\times\bR^d;\ell_2),\quad \boldsymbol{\pi}_n:=(\pi\ast e_1,\cdots,\pi\ast e_n,0,\cdots),
    $$
    and $g_m^n=0$ on $\{(\omega,t,x):u_m^n(\omega,t,x)\leq0\}$.
    Applying Theorem \ref{max_prin}, we conclude $u_m^n \geq 0$ for all $t \in [0, \tau]$ almost surely.
    The final step of the proof is to establish that
    \begin{equation}
    \label{24.03.05.16.04}
        \lim_{n\to\infty}\|u_m^n-u_m\|_{\cH_{p,q}^{\phi,\gamma}(T)}=0.
    \end{equation}
    Here, $u_m$ is the solution in Theorem \ref{local}-$(i)$.
    According to Corollary \ref{23.07.04.17.12},
    \begin{align*}
        \|u_m^n-u_m\|_{\cH_{p,q}^{\phi,\gamma}(\tau\wedge t)}&\leq N\|u_{0,n}-u_0\|_{U_{p,q}^{\phi,\gamma}}+N\|u_m^n-u_m\|_{\bL_{p,q}(\tau\wedge t)}\\
        &\quad +N\|(\boldsymbol{\pi}-\boldsymbol{\pi}_n)u_m\|_{\bH_{p,q}^{\phi,\gamma-1}(\tau\wedge t,\ell_2)},
    \end{align*}
    where $N$ is a constant independent of $n$ and $t$.
    Employing the inequality \eqref{stochastic_gronwall} and Gr\"onwall's inequality, we derive
    \begin{align*}
        \|u_m^n-u_m\|_{\cH_{p,q}^{\phi,\gamma}(\tau)}&\leq N\|u_{0,n}-u_0\|_{U_{p,q}^{\phi,\gamma}}+N\|(\boldsymbol{\pi}-\boldsymbol{\pi}_n)u_m\|_{\bH_{p,q}^{\phi,\gamma-1}(\tau,\ell_2)}.
    \end{align*}
    With \eqref{24.03.04.14.29} in consideration,
    \begin{align*}
    \|(\boldsymbol{\pi}-\boldsymbol{\pi}_n)u_m\|_{\bH_{p,q}^{\phi,\gamma-1}(\tau,\ell_2)}^q\leq N\|u_m\|_{\bL_{p,q}(\tau)},
    \end{align*}
    leading to the conclusion via the dominated convergence theorem and \eqref{24.03.24.14.54} that \eqref{24.03.05.16.04} holds true.
    Thus, the theorem is proven.
\end{proof}

\subsection{Space-time H\"older type regularity of local solutions: Proof of Theorem \ref{local}-(iii)}
\label{24.03.22.14.23}
The proof of Theorem \ref{local}-(iii), which concerns the space–time H\"older-type regularity of local solutions, follows as an immediate corollary of the theorem stated below. We begin by recalling the definition of parabolic H\"older spaces adapted to our setting.

\begin{defn}[Parabolic H\"older space]
\label{25.08.11.21.09}
Let $\delta_0\in(0,1]$, $\gamma\in(0,\infty)$, $p\in[2,\infty)$, $q\in(2,\infty)$, and let $\phi$ be a Bernstein function satisfying Assumption \ref{24.03.14.11.58}~($\delta_0$).
Assume that
$$
\frac{1}{q}\leq\alpha<\beta\leq\frac{1}{2},\qquad \delta_0(\gamma-2\beta)-\frac{d}{p}>0.
$$
For an integer $L>\gamma-\frac{d}{p}$, define the $L$-th order difference in the parabolic direction $(\tau,h)\in\mathbb{R}\times\mathbb{R}^d$ by
$$
\mathcal{D}^{0}_{(\tau,h)}u:=u,\qquad
\mathcal{D}^{L}_{(\tau,h)}u(t,x):=\mathcal{D}^{1}_{(\tau,h)}\mathcal{D}^{L-1}_{(\tau,h)}u(t,x)=\sum_{k=0}^{L}(-1)^{L-k}\binom{L}{k}\,u(t+k\tau,x+kh).
$$
For $T>0$, set
$$
\mathbb{R}^d_{T,L,\tau}:=\Big\{(t,x)\in [0,T]\times\mathbb{R}^d:\ (t+L\tau,x)\in [0,T]\times\mathbb{R}^d\Big\}.
$$
We denote by $C_{t,x}^{\alpha-\frac{1}{q},\,\phi_{\gamma-2\beta}-\frac{d}{p}}([0,T]\times\mathbb{R}^d)$ the set of continuous functions $u$ on $[0,T]\times\mathbb{R}^d$ such that the norm
$$
|u|_{C_{t,x}^{\alpha-\frac{1}{q},\,\phi_{\gamma-2\beta}-\frac{d}{p}}([0,T]\times\mathbb{R}^d)}
:=|u|_{C([0,T]\times\mathbb{R}^d)}+\sup_{(\tau,h)\neq(0,0)}
\frac{\|\mathcal{D}^{L}_{(\tau,h)}u\|_{L_\infty(\mathbb{R}^d_{T,L,\tau})}}
{|\tau|^{\alpha-\frac{1}{q}}+\phi(|h|^{-2})^{-\frac{\gamma}{2}+\beta}\,|h|^{-\frac{d}{p}}}
$$
is finite.
\end{defn}

\begin{prop}
\label{25.08.11.21.08}
The space $C_{t,x}^{\alpha - \frac{1}{q},\, \phi_{\gamma - 2\beta} - \frac{d}{p}}([0,T] \times \mathbb{R}^d)$ can be characterized as follows:
$$
    C_{t,x}^{\alpha - \frac{1}{q},\, \phi_{\gamma - 2\beta} - \frac{d}{p}}([0,T] \times \mathbb{R}^d)
    = C(\mathbb{R}^d; C^{\alpha - \frac{1}{q}}([0,T])) 
      \cap C([0,T]; C^{\phi_{\gamma - 2\beta} - \frac{d}{p}}(\mathbb{R}^d)),
$$
where
$$
\begin{aligned}
    &|u|_{C(\mathbb{R}^d; C^{\alpha - \frac{1}{q}}([0,T]))} 
    := |u|_{C([0,T] \times \mathbb{R}^d)} 
    + \sup_{x \in \mathbb{R}^d} \sup_{t,s \in [0,T]} 
      \frac{|u(t,x) - u(s,x)|}{|t-s|^{\alpha - \frac{1}{q}}}, \\
    &|u|_{C([0,T]; C^{\phi_{\gamma - 2\beta} - \frac{d}{p}}(\mathbb{R}^d))} 
    := |u|_{C([0,T] \times \mathbb{R}^d)} 
    + \sup_{t \in [0,T]} \sup_{\substack{x \in \mathbb{R}^d \\ |h| \leq1}} 
      \frac{|\mathcal{D}^L_{(0,h)} u(t,x)|}{\phi(|h|^{-2})^{-\frac{\gamma}{2} + \beta} \, |h|^{-\frac{d}{p}}}.
\end{aligned}
$$
\end{prop}

\begin{proof}
First, note that
$$
\begin{aligned}
    |\mathcal{D}^L_{(\tau,h)} u(t,x)|
    &\leq|\mathcal{D}^L_{(\tau,h)} u(t,x) - \mathcal{D}^L_{(0,h)} u(t,x)|
        + |\mathcal{D}^L_{(0,h)} u(t,x)| \\
    &\leq N(L) \sum_{j=1}^L |u(t + j\tau, x + jh) - u(t, x + jh)|
        + |\mathcal{D}^L_{(0,h)} u(t,x)|.
\end{aligned}
$$
Hence
$$
\begin{aligned}
    &\frac{\|\mathcal{D}^L_{(\tau,h)} u\|_{L_\infty(\mathbb{R}^d_{T,L,\tau})}}
    {|\tau|^{\alpha - \frac{1}{q}} + \phi(|h|^{-2})^{-\frac{\gamma}{2} + \beta} |h|^{-\frac{d}{p}}} \\
    &\quad \leq N(\alpha, L, q) 
      \sup_{x \in \mathbb{R}^d} \sup_{t,s \in [0,T]} 
      \frac{|u(t,x) - u(s,x)|}{|t-s|^{\alpha - \frac{1}{q}}}
      + \sup_{t \in [0,T]} \sup_{\substack{x \in \mathbb{R}^d \\ |h| \leq 1}}
      \frac{|\mathcal{D}^L_{(0,h)} u(t,x)|}{\phi(|h|^{-2})^{-\frac{\gamma}{2} + \beta} |h|^{-\frac{d}{p}}}.
\end{aligned}
$$
This shows that
$$
    C(\mathbb{R}^d; C^{\alpha - \frac{1}{q}}([0,T])) 
    \cap C([0,T]; C^{\phi_{\gamma - 2\beta} - \frac{d}{p}}(\mathbb{R}^d))
    \subseteq C_{t,x}^{\alpha - \frac{1}{q},\, \phi_{\gamma - 2\beta} - \frac{d}{p}}([0,T] \times \mathbb{R}^d).
$$

Conversely, note that
$$
\frac{|\mathcal{D}^L_{(\tau,0)} u(t,x)|}{|\tau|^{\alpha - \frac{1}{q}}}
+ \frac{|\mathcal{D}^L_{(0,h)} u(t,x)|}{\phi(|h|^{-2})^{-\frac{\gamma}{2} + \beta} |h|^{-\frac{d}{p}}}
\leq \sup_{(\tau,h) \neq (0,0)}
\frac{\|\mathcal{D}^L_{(\tau,h)} u\|_{L_\infty(\mathbb{R}^d_{T,L,\tau})}}
{|\tau|^{\alpha - \frac{1}{q}} + \phi(|h|^{-2})^{-\frac{\gamma}{2} + \beta} |h|^{-\frac{d}{p}}}.
$$
It remains to prove
\begin{equation}\label{25.08.11.17.03}
\sup_{t,s\in[0,T]} \frac{|u(t,x)-u(s,x)|}{|t-s|^{\alpha-\frac1q}}
\leq N_1
\sup_{t\in[0,T]}|u(t,x)|+N_2\sup_{t>0}\frac{\omega_L(u(\cdot,x),t)}{t^{\alpha-\frac{1}{q}}},
\end{equation}
where $N_1=N_1(\alpha,L,q,T)$, $N_2=N_2(\alpha,L,q)$
$$
\omega_L(u(\cdot,x),t):=\sup_{0<|\tau|\leq t}\big\|\mathcal{D}^{L}_{(\tau,0)}u(\cdot,x)\big\|_{L_\infty(\mathrm{T}_{L,\tau})},
\qquad
\mathrm{T}_{L,\tau}:=\big\{s\in[0,T]:\, s+L\tau\in[0,T]\big\}.
$$
Since \eqref{25.08.11.17.03} is pointwise in $x$, it suffices to prove it for scalar functions on $[0,T]$. 
Thus, let $f:[0,T]\to\mathbb{R}$ and set
$$
\mathcal{D}^L_{\tau}f(t):=\sum_{k=0}^{L}(-1)^{L-k}\binom{L}{k}f(t+k\tau),\qquad
\omega_L(f,t):=\sup_{0<|\tau|\leq t}\big\|\mathcal{D}^{L}_{\tau}f\big\|_{L_\infty(\mathrm{T}_{L,\tau})},
$$
and
$$
[f]_{C^{\beta}_L([0,T])}:=\sup_{t>0}\frac{\omega_L(f,t)}{t^{\beta}}.
$$

\medskip\noindent
\textbf{Claim 1.} Let $a_j:=\omega_L(f,2^{-j}t)$ and $b_{j}:=\omega_{L+1}\!\left(f,2^{-j}t\right)$. Then,
\begin{equation}
\label{25.08.26.09.29}
a_0\leq 2^L(1+2^{-LJ})\,a_J+L2^{L-1}\sum_{j=0}^{J}(1+2^{-Lj})\,b_{j}, \quad \forall\, J\in \mathbb{N}.
\end{equation}
\begin{proof}[Proof of Claim 1]
Let $I$ be the identity and $T_\tau f(t):=f(t+\tau)$. Then
$$
\mathcal{D}_{\tau}=(I+T_{\tau/2})\mathcal{D}_{\tau/2},
$$
hence
\begin{align*}
\mathcal{D}^{L}_{\tau}-\mathcal{D}^{L}_{\tau/2}
&=(I+T_{\tau/2})^{L}\mathcal{D}^{L}_{\tau/2}-\mathcal{D}^{L}_{\tau/2}
=\sum_{k=1}^{L}\binom{L}{k}(T_{\tau/2}^{k}-2^{-L}I)\,\mathcal{D}^{L}_{\tau/2}\\
&=\sum_{k=1}^{L}\binom{L}{k}\sum_{r=0}^{k-1}2^{-\frac{L(k-1-r)}{k}}T_{\tau/2}^{\,r}\,\mathcal{D}^{L+1}_{\tau/2}.
\end{align*}
Therefore,
$$
\big\|\mathcal{D}^{L}_{\tau}f-\mathcal{D}^{L}_{\tau/2}f\big\|_{L_\infty(\mathrm{T}_{L,\tau})}
\leq L2^{L-1}\,\big\|\mathcal{D}^{L+1}_{\tau/2}f\big\|_{L_\infty(\mathrm{T}_{L+1,\tau/2})}.
$$
Then for integers $0\leq k<J$, define
$$
S_{k,J}:=\sum_{j=k}^{J-1}2^{-Lj}\Big(\mathcal{D}^{L}_{\tau/2^{j}}-\mathcal{D}^{L}_{\tau/2^{j+1}}\Big),
$$
which telescopes to
$$
S_{k,J}
=2^{-Lk}\mathcal{D}^{L}_{\tau/2^{k}}
-(2^{L}-1)\sum_{j=k+1}^{J-1}2^{-Lj}\mathcal{D}^{L}_{\tau/2^{j}}
-2^{-L(J-1)}\mathcal{D}^{L}_{\tau/2^{J}}.
$$
Taking sup–norms on the appropriate domains yields
\begin{equation*}
\begin{aligned}
2^{-Lk}\big\|\mathcal{D}^{L}_{\tau/2^{k}}f\big\|_{L_\infty(\mathrm{T}_{L,\tau/2^{k}})}
&\leq (2^{L}-1)\sum_{j=k+1}^{J-1}2^{-Lj}\big\|\mathcal{D}^{L}_{\tau/2^{j}}f\big\|_{L_\infty(\mathrm{T}_{L,\tau/2^{j}})}\\
&\quad +2^{-L(J-1)}\big\|\mathcal{D}^{L}_{\tau/2^{J}}f\big\|_{L_\infty(\mathrm{T}_{L,\tau/2^{J}})}\\
&\quad +L2^{L-1}\sum_{j=k}^{J-1}2^{-Lj}\big\|\mathcal{D}^{L+1}_{\tau/2^{j+1}}f\big\|_{L_\infty(\mathrm{T}_{L+1,\tau/2^{j+1}})}.
\end{aligned}
\end{equation*}
Hence, we obtain the discrete Volterra inequality: for $k=0,1,2,\cdots,J-2$,
\begin{equation}\label{25.08.22.16.37}
a_k\leq 2^{-L(J-1-k)}a_J+(2^{L}-1)\sum_{j=k+1}^{J-1}2^{-L(j-k)}a_j+L2^{L-1}\sum_{j=k}^{J-1}2^{-L(j-k)}b_{j+1}.
\end{equation}
By summing over $k=0,1,\cdots,J-2$,
\begin{align*}
    \sum_{k=0}^{J-2}a_k&\leq a_J\sum_{k=0}^{J-2}2^{-L(J-1-k)}+(2^{L}-1)\sum_{k=0}^{J-2}\sum_{j=k+1}^{J-1}2^{-L(j-k)}a_j+L2^{L-1}\sum_{k=0}^{J-2}\sum_{j=k}^{J-1}2^{-L(j-k)}b_{j+1}\\
    &\leq \frac{2^{L}a_J}{2^{L}-1}+\sum_{j=1}^{J-1}(1-2^{-Lj})a_j+\frac{L2^{L-1}}{2^{L}-1}\sum_{j=0}^{J-1}b_{j+1}.
\end{align*}
Since $a_{k+1}\leq a_k$, we have
$$
\sum_{k=1}^{J-1}a_k\leq\sum_{k=0}^{J-2}a_k\leq \frac{2^{L}a_J}{2^{L}-1}+\sum_{j=1}^{J-1}(1-2^{-Lj})a_j+\frac{L2^{L-1}}{2^{L}-1}\sum_{j=0}^{J-1}b_{j+1},
$$
thus
\begin{equation}
\label{25.08.26.09.27}
\sum_{j=1}^{J-1}2^{-Lj}a_j\leq \frac{2^{L}a_J}{2^{L}-1}+\frac{L2^{L-1}}{2^{L}-1}\sum_{j=0}^{J-1}b_{j+1}
\end{equation}
Putting $k=0$ in \eqref{25.08.22.16.37} and using \eqref{25.08.26.09.27} yields \eqref{25.08.26.09.29}.
\end{proof}

\medskip\noindent
\textbf{Claim 2.} For all $t\in(0,T]$,
\begin{equation}
\label{25.08.26.10.40}
    \omega_L(f,t)\leq N(L)t^{L}\|f\|_{L_{\infty}([0,T])}+N(\alpha,L,q)t^{\alpha-\frac{1}{q}}[f]_{C_{L+1}^{\alpha-\frac{1}{q}}([0,T])}.
\end{equation}

\begin{proof}[Proof of Claim 2]
Note that for any $J\in \mathbb{N}$,
$$
1= \frac{1}{\log 2}\int_{2^{-J}t}^{2^{-(J-1)}t} \frac{1}{\lambda} \,\mathrm{d}\lambda,
\qquad
2^{-LJ}=\frac{t^L}{L}\int_{2^Jt}^{\infty}\lambda^{-L-1}\,\mathrm{d}\lambda.
$$
Since $\lambda\mapsto\omega_L(f,\lambda)$ is nondecreasing, we have
$$
(1+2^{-LJ})a_J\leq\frac{1}{\log2}\int_{0}^{2^{-(J-1)}t}\frac{\omega_L(f,\lambda)}{\lambda}\mathrm{d}\lambda+\frac{t^L}{L}\int_{2^Jt}^{\infty}\frac{\omega_L(f,\lambda)}{\lambda^{L+1}}\mathrm{d}\lambda.
$$
Using 
$$
\omega_L(f,\lambda)\leq \min\left(2^L\|f\|_{L_{\infty}([0,T])},\lambda^{\alpha-\frac{1}{q}}[f]_{C_L^{\alpha-\frac{1}{q}}([0,T])}\right),
$$
we obtain, for all $J\in\mathbb{N}$,
\begin{align*}
&(1+2^{-LJ})a_J
\\
&\leq\frac{[f]_{C_L^{\alpha-\frac{1}{q}}([0,T])}}{\log2}\int_{0}^{2^{-(J-1)}t}\lambda^{\alpha-\frac{1}{q}-1}\mathrm{d}\lambda
+\frac{2^Lt^L}{L}\|f\|_{L_{\infty}([0,T])}\int_{2^{J}t}^{\infty}\lambda^{-L-1}\mathrm{d}\lambda\\
&\leq N(\alpha,q)t^{\alpha-\frac{1}{q}}2^{-J(\alpha-\frac{1}{q})}[f]_{C_L^{\alpha-\frac{1}{q}}([0,T])}+N(L)t^L\|f\|_{L_{\infty}([0,T])}\int_{2^{J}t}^{\infty}\lambda^{-L-1}\mathrm{d}\lambda.
\end{align*}
Taking the limit superior as $J\to\infty$ gives
$$
\limsup_{J\to\infty}(1+2^{-LJ})a_J\leq N(L)t^L\|f\|_{L_{\infty}([0,T])}\limsup_{J\to\infty}\int_{2^{J}t}^{\infty}\lambda^{-L-1}\mathrm{d}\lambda.
$$
Since
$$
\limsup_{J\to\infty}\int_{2^{J}t}^{\infty}\lambda^{-L-1}\mathrm{d}\lambda\leq \int_{1}^{\infty}\lambda^{-L-1}\mathrm{d}\lambda=N(L),
$$
we have
\begin{equation}
\label{25.08.26.10.20}
\limsup_{J\to\infty}(1+2^{-LJ})a_J\leq N(L)t^L\|f\|_{L_{\infty}([0,T])}.
\end{equation}
Next, since $\lambda \mapsto\omega_{L+1}(f, \lambda)$ is nondecreasing,
$$
b_{j+1}
=\frac1{\log 2}\int_{2^{-(j+1)}t}^{2^{-j}t}\frac{\omega_{L+1}\left(f,2^{-(j+1)}t\right)}{\lambda}\,\mathrm{d}\lambda
\leq \frac1{\log 2}\int_{2^{-(j+1)}t}^{2^{-j}t}\frac{\omega_{L+1}(f,\lambda)}{\lambda}\,\mathrm{d}\lambda,
$$
and thus, for any $J\in\mathbb{N}$,
\begin{align*}
\sum_{j=0}^{J-1}(1+2^{-Lj})b_{j+1}&\leq\frac{1}{\log2}\sum_{j=0}^{J-1}(1+2^{-Lj})\int_{2^{-(j+1)}t}^{2^{-j}t}\frac{\omega_{L+1}(f,\lambda)}{\lambda}\mathrm{d}\lambda \\
&\leq \frac{1}{\log2}\sum_{j=0}^{J-1}\int_{2^{-(j+1)}t}^{2^{-j}t}\left(1+\left(\frac{2\lambda}{t}\right)^{L}\right)\frac{\omega_{L+1}(f,\lambda)}{\lambda}\mathrm{d}\lambda\\
&=\frac{1}{\log2}\int_{2^{-J}t}^{t}\left(1+\left(\frac{2\lambda}{t}\right)^{L}\right)\frac{\omega_{L+1}(f,\lambda)}{\lambda}\mathrm{d}\lambda\\
&\leq \frac{(1+2^L)}{\log2}\int_{2^{-J}t}^{t}\frac{\omega_{L+1}(f,\lambda)}{\lambda}\mathrm{d}\lambda.
\end{align*}
Hence,
\begin{equation}
\label{25.08.26.10.22}
\begin{aligned}
\limsup_{J\to\infty}\sum_{j=0}^{J-1}(1+2^{-Lj})b_{j+1}&\leq \frac{(1+2^L)}{\log2}\int_{0}^{t}\frac{\omega_{L+1}(f,\lambda)}{\lambda}\mathrm{d}\lambda\\
&\leq [f]_{C_{L+1}^{\alpha-\frac{1}{q}}([0,T])}\frac{L2^{L-1}}{\log 2}\int_{0}^t\lambda^{\alpha-\frac{1}{q}-1}\mathrm{d}\lambda\\
&= N(\alpha,L,q)t^{\alpha-\frac{1}{q}}[f]_{C_{L+1}^{\alpha-\frac{1}{q}}([0,T])}.
\end{aligned}
\end{equation}
Starting from \eqref{25.08.26.09.29} and using \eqref{25.08.26.10.20} and \eqref{25.08.26.10.22}, we conclude
\begin{align*}
\omega_{L}(f,t) =a_{0} &\leq \limsup_{J\to\infty} \left( 2^{L}(1+2^{-LJ})a_{J} + L2^{L-1}\sum_{j=0}^{J}(1+2^{-Lj})b_{j}  \right)
\\
& \leq N(L)t^{L}\|f\|_{L_{\infty}([0,T])} + N(\alpha,L,q)t^{\alpha-\frac{1}{q}} [f]_{C^{\alpha-\frac{1}{q}}_{L+1}([0,T])},
\end{align*}
which is \eqref{25.08.26.10.40}.
\end{proof}

Now, we prove \eqref{25.08.11.17.03}.
Dividing both sides of \eqref{25.08.26.10.40} by $t^{\alpha-\frac{1}{q}}$ and taking the supremum over $t\in(0,T]$, we obtain
\begin{align}\label{eqn 08.26.15:06}
[f]_{C^{\alpha-\frac1q}_L([0,T])}
\leq N(L)\,T^{L-\alpha+\frac1q}\|f\|_{L_\infty([0,T])}+
N(\alpha,L,q)\,[f]_{C^{\alpha-\frac1q}_{L+1}([0,T])}.
\end{align}
Since
\begin{align*}
\sup_{t,s \in [0,T]} \frac{ |f(t)-f(s)| }{|t-s|^{\alpha-\frac{1}{q}}} \leq \sup_{t>0} \frac{\omega_{1}(f,t)}{t^{\alpha-\frac{1}{q}}} =: [f]_{C^{\alpha-\frac{1}{q}}_{1}([0,T])},
\end{align*}
\eqref{eqn 08.26.15:06} yields
\begin{align*}
\sup_{t,s \in [0,T]} \frac{ |f(t)-f(s)| }{|t-s|^{\alpha-\frac{1}{q}}} &\leq [f]_{C^{\alpha-\frac{1}{q}}_{1}([0,T])} \\
&\leq N(L) [f]_{C^{\alpha-\frac{1}{q}}_{2}([0,T])} + N(\alpha,L,q)\|f\|_{L_{\infty}([0,T])} 
\\
& \leq \cdots \leq N_1\|f\|_{L_{\infty}([0,T])}+N_2 [f]_{C^{\alpha-\frac{1}{q}}_{L}([0,T])} .
\end{align*}
Finally, applying this with $f(\cdot)=u(\cdot,x)$ gives \eqref{25.08.11.17.03}.
This completes the proof.
\end{proof}
\begin{rem}
The choice of $L$ in Definition \ref{25.08.11.21.09} is immaterial for the definition of 
$$
C_{t,x}^{\alpha - \frac{1}{q},\, \phi_{\gamma - 2\beta} - \frac{d}{p}}([0,T] \times \mathbb{R}^d).
$$  
Indeed, by Proposition \ref{25.08.11.21.08}, the parameter $L$ appears only in the definition of 
$C^{\phi_{\gamma - 2\beta} - \frac{d}{p}}(\mathbb{R}^d)$.  
Moreover, by \cite[Theorem 1.5]{CLSW2023},
$$
    |u|_{C^{\phi_{\gamma - 2\beta} - \frac{d}{p}}(\mathbb{R}^d)}
    \simeq \|u\|_{B_{\infty,\infty}^{\phi_{\gamma,\beta,p,d}}(\mathbb{R}^d)},
$$
where $\phi_{\gamma,\beta,p,d}(\lambda) := \phi(\lambda^2)^{\frac{\gamma}{2} - \beta} \lambda^{-\frac{d}{p}}$.  
It is straightforward to verify that the definition of  
$$
B_{\infty,\infty}^{\phi_{\gamma,\beta,p,d}}(\mathbb{R}^d)
$$
does not depend on the choice of $L$.
\end{rem}

\begin{thm}
\label{23.09.05.17.43}
Let $\delta_0\in(0,1]$, $\gamma\in(0,\infty)$, $p\in[2,\infty)$, $q\in(2,\infty)$, and suppose Assumption \ref{24.03.14.11.58}~($\delta_0$) holds.
If
$$
\frac{1}{q}\leq\alpha<\beta\leq\frac{1}{2},\qquad \delta_0(\gamma-2\beta)-\frac{d}{p}>0,
$$
then the space $\cH_{p,q}^{\phi,\gamma}(T)$ is continuously embedded into
$L_q\left(\Omega;\, C_{t,x}^{\alpha-\frac{1}{q},\,\phi_{\gamma-2\beta}-\frac{d}{p}}([0,T]\times\mathbb{R}^d)\right)$.
Equivalently, there exists a constant $N=N(d,N_0,N_1,N_2,p,q,T)$ such that for all $u\in\cH_{p,q}^{\phi,\gamma}(T)$,
$$
\bE\left[|u|_{C_{t,x}^{\alpha-\frac{1}{q},\,\phi_{\gamma-2\beta}-\frac{d}{p}}([0,T]\times\mathbb{R}^d)}^{\,q}\right]
\leq N\,\|u\|_{\cH_{p,q}^{\phi,\gamma}(T)}^{\,q}.
$$
\end{thm}
\begin{proof}
Utilizing \cite[Theorem 5.3 and Remark 5.4]{CLSW2023}, for the condition $\delta_0\gamma - d/p > 0$, we obtain
\begin{equation} \label{holder embedding}
\begin{aligned}
\sum_{i=0}^{n}|D^iu|_{C(\bR^d)}+|D^nu|_{C^{\nu}(\bR^d)}&\leq N\left(| u |_{C(\bR^d)} + \sup_{|h|\leq1}\left(\frac{|\cD_h^Lu|_{C(\bR^d)}}{\phi(|h|^{-2})^{-\gamma/2}|h|^{-d/p}}\right)\right) \\
&\leq N \| u \|_{H_{p}^{\phi,\gamma}(\bR^d)},
\end{aligned}
\end{equation}
where $\mathcal{D}^L_hu(x) := \mathcal{D}_h(\mathcal{D}^{L-1}_hu)(x)$, $\mathcal{D}_hu(x) := u(x+h) - u(x)$, $L$ is an integer greater than $\gamma - d/p$, and $\delta_0\gamma - d/p = n + \nu$ for $n\in\mathbb{N}\cup\{0\}$ and $\nu \in (0,1]$.
By applying \eqref{holder embedding} and considering the continuous embedding given by Theorem \ref{prop}, \textit{i.e.},
$$
\cH_{p,q}^{\phi,\gamma}(T)\hookrightarrow L_q(\Omega; C^{\alpha-1/q}([0,T];H_{p}^{\phi,\gamma-2\beta}(\bR^d))),
$$
it follows that
    \begin{align*}           \|u\|_{\cH_{p,q}^{\phi,\gamma}(T)}^q&\geq N\bE\left[\left(\sup_{0\leq t\leq T}\|u(t,\cdot)\|_{H_p^{\phi,\gamma-2\beta}(\bR^d)}+\sup_{0\leq s<t\leq T}\frac{\|u(t,\cdot)-u(s,\cdot)\|_{H_p^{\phi,\gamma-2\beta}(\bR^d)}}{(t-s)^{\alpha-1/q}}\right)^q\right]\\
    &\geq N\bE\left[\left(\sup_{|h|\leq 1}\frac{|\cD_h^Lu|_{C([0,T]\times\bR^d)}}{\phi(|h|^{-2})^{-\frac{\gamma}{2}+\beta}|h|^{-d/p}}+\sup_{0\leq s<t\leq T}\frac{|u(t,\cdot)-u(s,\cdot)|_{C(\bR^d)}}{(t-s)^{\alpha-1/q}}\right)^q\right]\\
    &\quad +N\mathbb{E}\left[\left(\sup_{(t,x)\in[0,T]\times\bR^d}|u(t,x)|\right)^q\right].
    \end{align*}
    Proposition \ref{25.08.11.21.08} leads us to conclude that
    $$
    \bE\left[|u|_{C_{t,x}^{\alpha-\frac{1}{q},\phi_{\gamma-2\beta}-\frac{d}{p}}([0,T]\times\bR^d)}^q\right]\leq N\|u\|_{\cH_{p,q}^{\phi,\gamma}(T)}^q.
    $$
    The theorem is proved.
\end{proof}

The remaining task in this subsection is to establish the continuous embedding
$$
\cH_{p,q}^{\phi,\gamma}(T)\hookrightarrow L_q(\Omega; C^{\alpha-1/q}([0,T];H_{p}^{\phi,\gamma-2\beta}(\bR^d))),
$$
which plays a crucial role in the proof of Theorem \ref{23.09.05.17.43}.
\begin{thm} \label{prop}
Given $\tau\leq T$, $p \in [2, \infty)$, $q \in (2, \infty)$, and $\gamma \in \mathbb{R}$, and for
$$
\frac{1}{q}\leq\alpha<\beta\leq\frac{1}{2},
$$
the space $\cH_{p,q}^{\phi,\gamma}(\tau)$ is continuously embedded into $L_q(\Omega; C^{\alpha-1/q}([0,\tau]; H_{p}^{\phi,\gamma-2\beta}(\mathbb{R}^d)))$.
That is, there exists a constant $N = N(d,N_0,N_1,N_2,p, q, T)$ such that
\begin{equation*}
\bE\left[\left(\sup_{t\leq \tau}\|u(t,\cdot)\|_{H_p^{\phi,\gamma-2\beta}(\bR^d)}+\sup_{0\leq s<t\leq \tau}\frac{\|u(t)-u(s)\|_{H_p^{\phi,\gamma-2\beta}(\bR^d)}}{(t-s)^{\alpha-1/q}}\right)^q\right]\leq N\|u\|_{\cH_{p,q}^{\phi,\gamma}(\tau)}^q.
\end{equation*}
\end{thm}
The proof of Theorem \ref{prop} is contingent upon the following two lemmas, and thus, will be presented at the conclusion of this subsection.

\begin{lem}
\label{23.09.03.15.14}
\begin{enumerate}[(i)]
    \item Let $\phi$ be a Bernstein function,
    $$
    T_th(x):=\mathbb{E}[h(x+X_t)],
    $$
    and $X$ be a SBM with characteristic exponent $\phi$.
    For any $h\in L_p(\bR^d)$, $\theta\in[0,1]$, and $t>0$, we have
    \begin{equation}
    \label{23.08.23.13.11}
        \|\mathrm{e}^{-t}T_th\|_{L_p(\bR^d)}\leq \frac{N}{t^{\theta}}\|h\|_{H_p^{\phi,-2\theta}(\bR^d)},\quad \|(T_t-1)h\|_{L_p(\bR^d)}\leq Nt^{\theta}\|h\|_{H_p^{\phi,2\theta}(\bR^d)},
    \end{equation}
    where $N=N(d,p,\theta)$.
    \item (Garsia-Rodemich-Rumsey inequality) Let $V$ be a Banach space and $h$ be a $V$-valued continuous function.
    Then for $s\leq t$, $\alpha \geq1/q$ and $q>0$,
    $$
    \|h(t)-h(s)\|_{V}\leq N(\alpha,q)|t-s|^{\alpha -\frac{1}{q}}\left(\int_s^t\int_s^t\frac{\|h(r_1)-h(r_2)\|_{V}^q}{|r_1-r_2|^{\alpha q+1}}\mathrm{d}r_2\mathrm{d}r_1\right)^{1/q}.
    $$
\end{enumerate}
\end{lem}
\begin{proof}
    $(i)$ For $\theta\in(0,1)$,
    \begin{align*}
        &\|(1-\phi(\Delta))^{\theta}(\mathrm{e}^{-t}T_th)\|_{L_p(\bR^d)}\\
        &=c(\theta)\left\|\int_0^{\infty}\frac{\mathrm{e}^{-(s+t)}T_{s+t}h-\mathrm{e}^{-t}T_th}{s^{\theta}}\frac{\mathrm{d}s}{s}\right\|_{L_p(\bR^d)}\\
        &\leq N\int_0^t\|\mathrm{e}^{-(s+t)}T_{s+t}h-\mathrm{e}^{-t}T_th\|_{L_p(\bR^d)}\frac{\mathrm{d}s}{s^{1+\theta}}+N\|h\|_{L_p(\bR^d)}\int_t^{\infty}\frac{\mathrm{d}s}{s^{1+\theta}}.
    \end{align*}
    By inequality \eqref{23.08.23.12.33} and Lemma \ref{23.09.19.12.43}-$(i)$, we have
    \begin{equation}
    \label{23.08.23.13.02}
        \|\partial_t(T_th)\|_{L_p(\bR^d)}+\|\partial_t(\mathrm{e}^{-t}T_th)\|_{L_p(\bR^d)}\leq\frac{N}{t}\|h\|_{L_p(\bR^d)}
    \end{equation}
    leading to the conclusion that
    $$
    \int_0^t\|\mathrm{e}^{-(s+t)}T_{s+t}h-\mathrm{e}^{-t}T_th\|_{L_p(\bR^d)}\frac{\mathrm{d}s}{s^{1+\theta}}\leq N\|h\|_{L_p(\bR^d)}t^{-1}\int_0^t\frac{\mathrm{d}s}{s^{\theta}}=\frac{N}{t^{\theta}}\|h\|_{L_p(\bR^d)}.
    $$
    For $\theta=1$, we use Lemma \ref{23.09.19.12.43}-$(i)$;
    $$
    \partial_t(\mathrm{e}^{-t}T_th)=(\phi(\Delta)-1)(\mathrm{e}^{-t}T_th)
    $$
    and \eqref{23.08.23.13.02}.

    The second inequality in \eqref{23.08.23.13.11} is derived from equation \eqref{23.04.27.15.01} and the initial inequality in \eqref{23.08.23.13.11}, as follows:
    \begin{align*}
        \|(T_t-1)h\|_{L_p(\bR^d)}&\leq \int_0^t\|[\phi(\Delta)(1-\phi(\Delta))^{-1}](1-\phi(\Delta))^{1-\theta}T_s[(1-\phi(\Delta))^{\theta}h]\|_{L_p(\bR^d)}\mathrm{d}s\\
        &\leq N\int_0^t\|(1-\phi(\Delta))^{1-\theta}T_s[(1-\phi(\Delta))^{\theta}h]\|_{L_p(\bR^d)}\mathrm{d}s\\
        &\leq N\|h\|_{H_p^{\phi,2\theta}(\bR^d)}\int_0^ts^{\theta-1}\mathrm{d}s=Nt^{\theta}\|h\|_{H_p^{\phi,2\theta}(\bR^d)}.
    \end{align*}
    $(ii)$ We follow the argument of \cite[Lemma 1.1]{GRR1970}.
    First, we establish the inequality
    \begin{equation}
    \label{23.09.03.14.00}
        \|h(1)-h(0)\|_{V}^q\leq N(\alpha,q)\int_0^1\int_0^1\frac{\|h(t)-h(s)\|_{V}^q}{|t-s|^{\alpha q+1}}\mathrm{d}s\mathrm{d}t,
    \end{equation}
    for a $V$-valued continuous function $h$ on $[0, 1]$.
    Define
    $$
    I(t):=\int_0^1\frac{\|h(t)-h(s)\|_{V}^q}{|t-s|^{\alpha q+1}}\mathrm{d}s.
    $$
    Utilizing the mean value theorem, there exists $t_0 \in (0, 1)$ such that
    $$
    I(t_0)\leq \int_0^1\int_0^1\frac{\|h(t)-h(s)\|_{V}^q}{|t-s|^{\alpha q+1}}\mathrm{d}s\mathrm{d}t.
    $$
    If $t_{n-1}$ is specified, then let
    $$
    d_{n-1}:=2^{-\frac{q}{\alpha q+1}}t_{n-1}.
    $$
    Let us also denote
    $$
    A:=\int_0^{1}\int_0^1\frac{\|h(t)-h(s)\|_{V}^q}{|t-s|^{\alpha q+1}}\mathrm{d}s\mathrm{d}t,\quad B:=I(t_{n-1}).
    $$
    Applying the mean value theorem once again, we find $t_n \in (0, d_{n-1})$ satisfying \begin{align*}
        2&\geq \frac{1}{A}\int_0^{d_{n-1}}\int_0^1\frac{\|h(t)-h(s)\|_{V}^q}{|t-s|^{\alpha q+1}}\mathrm{d}t\mathrm{d}s+\frac{1}{B}\int_0^{d_{n-1}}\frac{\|h(t_{n-1})-h(s)\|_{V}^q}{|t_{n-1}-s|^{\alpha q+1}}\mathrm{d}s\\
        &=\int_0^{d_{n-1}}\left(\frac{1}{A}\int_0^1\frac{\|h(t)-h(s)\|_{V}^q}{|t-s|^{\alpha q+1}}\mathrm{d}t+\frac{1}{B}\frac{\|h(t_{n-1})-h(s)\|_{V}^q}{|t_{n-1}-s|^{\alpha q+1}}\right)\mathrm{d}s\\
        &=d_{n-1}\left(\frac{1}{A}\int_0^1\frac{\|h(t)-h(t_{n})\|_{V}^q}{|t-t_n|^{\alpha q+1}}\mathrm{d}t+\frac{1}{B}\frac{\|h(t_{n-1})-h(t_n)\|_{V}^q}{|t_{n-1}-t_n|^{\alpha q+1}}\right).
    \end{align*}
    Therefore,
    \begin{equation}
    \label{24.04.10.00.09}
        d_{n-1}I(t_n)\leq 2\int_0^{1}\int_0^1\frac{\|h(t)-h(s)\|_{V}^q}{|t-s|^{\alpha q+1}}\mathrm{d}s\mathrm{d}t
    \end{equation}
    and
    \begin{equation}
    \label{23.08.30.16.21}
        \frac{d_{n-1}\|h(t_n)-h(t_{n-1})\|_{V}^q}{|t_n-t_{n-1}|^{\alpha q+1}}\leq 2I(t_{n-1}).
    \end{equation}
    It follows that $t_n \downarrow 0$ as $n \to \infty$.
    Using $d_{n} \leq d_{n-1}$ and \eqref{24.04.10.00.09}, 
    \begin{equation}
        \label{23.08.30.16.22}
        d_nI(t_n)\leq 2\int_0^1\int_0^1\frac{\|h(t)-h(s)\|_{V}^q}{|t-s|^{\alpha q+1}}\mathrm{d}s\mathrm{d}t.
    \end{equation}
    From \eqref{23.08.30.16.21} and \eqref{23.08.30.16.22}, we derive
    \begin{align*}
        \|h(t_n)-h(t_{n-1})\|_{V}&\leq |t_n-t_{n-1}|^{\alpha +\frac{1}{q}} \left(\frac{2I(t_{n-1})}{d_{n-1}}\right)^{1/q}\\
        &\leq 2^{\frac{2}{q}}|t_n-t_{n-1}|^{\alpha +\frac{1}{q}} d_{n-1}^{-\frac{2}{q}}\left(\int_0^1\int_0^1\frac{\|h(t)-h(s)\|_{V}^q}{|t-s|^{\alpha q+1}}\mathrm{d}s\mathrm{d}t\right)^{1/q}
    \end{align*}
    Since
    \begin{equation*}
    \begin{gathered}
    |t_{n-1}-t_n|^{\alpha+\frac{1}{q}}\leq t_{n-1}^{\alpha+\frac{1}{q}}=2d_{n-1}^{\alpha+\frac{1}{q}},\quad d_{n}^{\alpha+\frac{1}{q}}=\frac{1}{2}t_n^{\alpha+\frac{1}{q}}\leq \frac{1}{2}d_{n-1}^{\alpha+\frac{1}{q}},
    \end{gathered}
    \end{equation*}
    we have
    \begin{align*}
    |t_n-t_{n-1}|^{\alpha+\frac{1}{q}}\leq 2d_{n-1}^{\alpha+\frac{1}{q}}=4\left(d_{n-1}^{\alpha+\frac{1}{q}}-\frac{d_{n-1}^{\alpha+\frac{1}{q}}}{2}\right)\leq 4(d_{n-1}^{\alpha+\frac{1}{q}}-d_{n}^{\alpha+\frac{1}{q}}).
    \end{align*}
    This leads to
    \begin{equation}
    \label{23.09.03.13.58}
    \begin{aligned}
        \|h(t_0)-h(0)\|_{V}&\leq N\left(\int_0^1\int_0^1\frac{\|h(t)-h(s)\|_{V}^q}{|t-s|^{\alpha q+1}}\mathrm{d}s\mathrm{d}t\right)^{1/q}\sum_{n=1}^{\infty}(d_{n-1}^{\alpha+\frac{1}{q}}-d_{n}^{\alpha+\frac{1}{q}})d_{n-1}^{-\frac{2}{q}}\\
        &\leq N\left(\int_0^1\int_0^1\frac{\|h(t)-h(s)\|_{V}^q}{|t-s|^{\alpha q+1}}\mathrm{d}s\mathrm{d}t\right)^{1/q}\sum_{n=1}^{\infty}\int_{d_n^{\alpha+\frac{1}{q}}}^{d_{n-1}^{\alpha+\frac{1}{q}}}x^{-\frac{2}{\alpha q+1}}\mathrm{d}x\\
        &=N\left(\int_0^1\int_0^1\frac{\|h(t)-h(s)\|_{V}^q}{|t-s|^{\alpha q+1}}\mathrm{d}s\mathrm{d}t\right)^{1/q}\sum_{n=1}^{\infty}\int_{d_n}^{d_{n-1}}x^{\alpha-1-\frac{1}{q}}\mathrm{d}x\\
        &\leq N(\alpha,q)\left(\int_0^1\int_0^1\frac{\|h(t)-h(s)\|_{V}^q}{|t-s|^{\alpha q+1}}\mathrm{d}s\mathrm{d}t\right)^{1/q}.
    \end{aligned}
    \end{equation}
    For the function $\tilde{h}(t) := h(1 - t)$, we consider
    $$
    \tilde{I}(t):=\int_0^1\frac{\|\tilde{h}(t)-\tilde{h}(s)\|_{V}^q}{|t-s|^{\alpha q+1}}\mathrm{d}s.
    $$
    Given $\tilde{t}_0 := 1 - t_0$, it follows that
    $$
    \tilde{I}(\tilde{t}_0)=I(t_0)\leq \int_0^1\int_0^1\frac{\|h(t)-h(s)\|_{V}^q}{|t-s|^{\alpha q+1}}\mathrm{d}s\mathrm{d}t.
    $$
    Thus, using similar reasoning as before, we can show
    \begin{equation}
    \label{23.09.03.13.59}
        \|h(t_0)-h(1)\|_V\leq N(\alpha,q)\left(\int_0^1\int_0^1\frac{\|h(t)-h(s)\|_{V}^q}{|t-s|^{\alpha q+1}}\mathrm{d}s\mathrm{d}t\right)^{1/q}.
    \end{equation}
By combining \eqref{23.09.03.13.58} and \eqref{23.09.03.13.59}, we have \eqref{23.09.03.14.00}.

For a general function $h$, consider
$$
\bar{h}(\lambda):=h(s+\lambda(t-s)).
$$
which maps $h$, a $V$-valued continuous function over the interval $[s, t]$, to $\bar{h}$, a similarly $V$-valued continuous function over $[0, 1]$.
Applying the established relation (\ref{23.09.03.14.00}) to $\bar{h}$, we deduce
\begin{align*}
\|h(t)-h(s)\|_V&=\|\bar{h}(1)-\bar{h}(0)\|_V\\
&\leq N(\alpha,q)\left(\int_0^1\int_0^1\frac{\|\bar{h}(r_1)-\bar{h}(r_2)\|_{V}^q}{|r_1-r_2|^{\alpha q+1}}\mathrm{d}r_2\mathrm{d}r_1\right)^{1/q}\\
&\leq N(\alpha,q)|t-s|^{\alpha q-1}\left(\int_s^t\int_s^t\frac{\|h(r_1)-h(r_2)\|_{V}^q}{|r_1-r_2|^{\alpha q+1}}\mathrm{d}r_2\mathrm{d}r_1\right)^{1/q}.
\end{align*}
The lemma is proved.
\end{proof}

\begin{proof}[Proof of Theorem \ref{prop}]
The proof is structured in four distinct steps for clarity.

\textbf{Step 1.} Problem reduction.

It suffices to prove the desired result for the case $\gamma = 2\beta$, since the general case $\gamma \in \mathbb{R}$ can be reduced to this special case via the isometry induced by the operator $(1 - \phi(\Delta))^{\frac{\gamma}{2} - \beta}$. Indeed, we have
\begin{align*}
    &N\|u\|_{\mathcal{H}_{p,q}^{\phi,\gamma}(\tau)}^q\\
    &=N\|(1-\phi(\Delta))^{\frac{\gamma}{2}-\beta}u\|_{\mathcal{H}_{p,q}^{\phi,2\beta}(\tau)}^q\\
    &\geq \bE\left[\sup_{t\leq \tau}\|(1-\phi(\Delta))^{\frac{\gamma}{2}-\beta}u(t,\cdot)\|_{L_p(\bR^d)}^q\right]\\
    &\quad+\bE\left[\left(\sup_{0\leq s<t\leq \tau}\frac{\|((1-\phi(\Delta))^{\frac{\gamma}{2}-\beta}(u(t,\cdot)-u(s,\cdot))\|_{L_p(\bR^d)}}{(t-s)^{\alpha-1/q}}\right)^q\right]\\
    &= \bE\left[\sup_{t\leq \tau}\|u(t,\cdot)\|_{H_p^{\gamma-2\beta}(\bR^d)}^q\right]+\bE\left[\left(\sup_{0\leq s<t\leq \tau}\frac{\|u(t,\cdot)-u(s,\cdot)\|_{H_p^{\phi,\gamma-2\beta}(\bR^d)}}{(t-s)^{\alpha-1/q}}\right)^q\right].
\end{align*}
Therefore, it is enough to establish the estimate for $\gamma = 2\beta$.

Given $\tau \leq T$ and $u \in \cH_{p,q}^{\phi,\gamma}(\tau)$, we define $f := (\mathbb{D}u - \phi(\Delta)u)1_{t \leq \tau}$ and $g = \mathbb{S}u1_{t \leq \tau}$.
Consequently, $f$ resides within $\bH_{p,q}^{\phi,\gamma-2}(T)$, and $g$ within $\bH_{p,q}^{\phi,\gamma-1}(T,\ell_2)$.
Leveraging Theorem \ref{23.07.04.14.29}, we establish that:
    \begin{equation}
    \label{24.03.25.14.33}
    v(t,x):=\mathcal{T}_0u_0(t,x)+\mathcal{T}_1(f-\tilde{\cT}_0u_0+\phi(\Delta)\cT_0u_0)(t,x)+\mathcal{T}_2g(t,x)
    \end{equation}
    is a unique solution to
    \begin{equation}
    \label{23.08.23.13.31}
        \mathrm{d}v=(\phi(\Delta)v+f)\mathrm{d}t+\sum_{k=1}^{\infty}g^k\mathrm{d}w_t^k;\quad v(0)=u(0),
    \end{equation}
    within $\cH_{p,q}^{\phi,2\beta}(T)$ fulfilling the estimate:
    \begin{equation}
    \label{24.03.25.14.13}
    \|v\|_{\cH_{p,q}^{\phi,2\beta}(T)}^q\leq N\left(\|u(0)\|_{U_{p,q}^{\phi,2\beta}}^q+\|f\|_{\bH_{p,q}^{\phi,2\beta-2}(T)}^q+\|g\|_{\bH_{p,q}^{\phi,2\beta-1}(T,\ell_2)}^q\right)\leq N\|u\|_{\cH_{p,q}^{\phi,2\beta}(\tau)}^q.
    \end{equation}
    The notation $\cT_0$, $\tilde{\cT}_0$, $\cT_1$, and $\cT_2$ refers to operators introduced in Theorem \ref{23.09.18.17.09}.
    Given $u$'s equivalency to a solution of \eqref{23.08.23.13.31} and the uniqueness, we deduce $u = v$ within $\cH_{p,q}^{\phi,\gamma}(\tau)$.
    Subsequent application of Theorem \ref{23.07.04.14.29} validates that:
    $$
    \tilde{v}(t,x):=T_tu_0(x)+\cT_1f(t,x)+\cT_2g(t,x)\in \cH_{p,q}^{\phi,2\beta}(T)
    $$
    also constitutes a solution to equation (\ref{23.08.23.13.31}), thereby establishing $u = v = \tilde{v}$ in $\cH_{p,q}^{\phi,2\beta}(\tau)$ due to the uniqueness.
    To fulfill the proof, it suffices to demonstrate the inequality
    \begin{equation}
    \label{24.03.25.14.07}
        \begin{aligned}
        &\bE\left[\left(\sup_{0\leq t\leq T}\|\tilde{v}(t)\|_{L_p(\bR^d)}+\sup_{0\leq s<t\leq T}\frac{\|\tilde{v}(t)-\tilde{v}(s)\|_{L_p(\bR^d)}}{(t-s)^{\alpha-1/q}}\right)^q\right]\\
    &\leq N\left(\|u(0)\|_{U_{p,q}^{\phi,2\beta}}^q+\|f\|_{\bH_{p,q}^{\phi,2\beta-2}(T)}^q+\|g\|_{\bH_{p,q}^{\phi,2\beta-1}(T,\ell_2)}\right),
    \end{aligned}
    \end{equation}
    from which we derive that
\begin{align*}           
    &\bE\left[\left(\sup_{0\leq t\leq \tau}\|u(t)\|_{L_p(\bR^d)}+\sup_{0\leq s<t\leq \tau}\frac{\|u(t)-u(s)\|_{L_p(\bR^d)}}{(t-s)^{\alpha-1/q}}\right)^q\right]\\
    &=\bE\left[\left(\sup_{0\leq t\leq \tau}\|\tilde{v}(t)\|_{L_p(\bR^d)}+\sup_{0\leq s<t\leq \tau}\frac{\|\tilde{v}(t)-\tilde{v}(s)\|_{L_p(\bR^d)}}{(t-s)^{\alpha-1/q}}\right)^q\right]\\
    &\leq\bE\left[\left(\sup_{0\leq t\leq T}\|\tilde{v}(t)\|_{L_p(\bR^d)}+\sup_{0\leq s<t\leq T}\frac{\|\tilde{v}(t)-\tilde{v}(s)\|_{L_p(\bR^d)}}{(t-s)^{\alpha-1/q}}\right)^q\right]\\
    &\leq N\left(\|u(0)\|_{U_{p,q}^{\phi,2\beta}}^q+\|f\|_{\bH_{p,q}^{\phi,2\beta-2}(T)}^q+\|g\|_{\bH_{p,q}^{\phi,2\beta-1}(T,\ell_2)}\right)\leq N\|u\|_{\cH_{p,q}^{\phi,2\beta}(\tau)}^q.
    \end{align*}

    Now, we proceed to affirm that if
    \begin{equation*}
    \label{24.03.11.14.04}
    \left(\mathbb{E}\left[\int_0^T\|h(\lambda)\|_{L_p(\bR^d)}^q\mathrm{d}\lambda\right]\right)+\mathbb{E}\left[\left(\sup_{0\leq s\leq t\leq T}\frac{\|h(t)-h(s)\|_{L_p(\bR^d)}}{(t-s)^{\alpha-1/q}}\right)^{q}\right]\leq K,
    \end{equation*}
    then it follows that
    \begin{equation*}
    \label{24.03.11.14.05}
    \mathbb{E}\left[\left(\sup_{0\leq t\leq T}\|h(t)\|_{L_p(\bR^d)}\right)^q\right]\leq N(\alpha,T,q)K.
    \end{equation*}
    Applying Minkowski's inequality for a fixed $\lambda\in[0,T]$, we derive that
    \begin{equation*}
    \label{24.03.11.13.57}
    \begin{aligned}
        \mathbb{E}\left[\left(\sup_{0\leq t\leq T}\|h(t)\|_{L_p(\bR^d)}\right)^q\right]&\leq N(q)\mathbb{E}\left[\left(\sup_{0\leq s\leq t\leq T}\|h(t)-h(s)\|_{L_p(\bR^d)}\right)^q\right]\\
        &\quad + N(q)\mathbb{E}\left[\|h(\lambda)\|_{L_p(\bR^d)}^q\right].
    \end{aligned}
    \end{equation*}
    Integrating with respect to $\lambda$ over the interval $[0,T]$ yields
    \begin{align*}
        \mathbb{E}\left[\left(\sup_{0\leq t\leq T}\|h(t)\|_{L_p(\bR^d)}\right)^q\right]&\leq N(q)\mathbb{E}\left[\left(\sup_{0\leq s\leq t\leq T}\|h(t)-h(s)\|_{L_p(\bR^d)}\right)^q\right]\\
        &\quad + \frac{N(q)}{T}\mathbb{E}\left[\int_0^T\|h(\lambda)\|_{L_p(\bR^d)}^q\mathrm{d}\lambda\right]\\
        &\leq N(\alpha,T,q)K.
    \end{align*}
    Given \eqref{24.03.25.14.13} and the aforementioned assertion, to substantiate \eqref{24.03.25.14.07}, it suffices to demonstrate
    \begin{equation}
    \label{24.03.25.21.04}
    \begin{aligned}
        &\mathbb{E}\left[\left(\sup_{0\leq s\leq t\leq T}\frac{\|\tilde{v}(t)-\tilde{v}(s)\|_{L_p(\bR^d)}}{(t-s)^{\alpha-1/q}}\right)^{q}\right]\\
        &\leq N\left(\|u(0)\|_{U_{p,q}^{\phi,2\beta}}^q+\|f\|_{\bH_{p,q}^{\phi,2\beta-2}(T)}^q+\|g\|_{\bH_{p,q}^{\phi,2\beta-1}(T,\ell_2)}\right).
    \end{aligned}
    \end{equation}
    
By \eqref{24.03.25.14.33},
\begin{equation*}
    \begin{aligned}
        &\mathbb{E}\left[\left(\sup_{0\leq s\leq t\leq T}\frac{\|\tilde{v}(t)-\tilde{v}(s)\|_{L_p(\bR^d)}}{(t-s)^{\alpha-1/q}}\right)^{q}\right]\\
        &\leq N\mathbb{E}\left[\left(\sup_{0\leq s\leq t\leq T}\frac{\|T_tu_0-T_su_0\|_{L_p(\bR^d)}}{(t-s)^{\alpha-1/q}}\right)^{q}\right]+N\mathbb{E}\left[\left(\sup_{0\leq s\leq t\leq T}\frac{\|\cT_1f(t)-\cT_1f(s)\|_{L_p(\bR^d)}}{(t-s)^{\alpha-1/q}}\right)^{q}\right]\\
        &+N\mathbb{E}\left[\left(\sup_{0\leq s\leq t\leq T}\frac{\|\cT_2g(t)-\cT_2g(s)\|_{L_p(\bR^d)}}{(t-s)^{\alpha-1/q}}\right)^{q}\right],
    \end{aligned}
\end{equation*}
where $N=N(q)$.
To satisfy \eqref{24.03.25.21.04}, it suffices to show
\begin{align}
    &\bE\left[\left(\sup_{0\leq s<t\leq T}\frac{\|T_tu_0-T_su_0\|_{L_p(\bR^d)}}{(t-s)^{\alpha -1/q}}\right)^{q}\right]\leq N\|u_0\|_{U_{p,q}^{\phi,2\beta}}^q,\label{24.03.25.21.05}\\
    &\bE\left[\left(\sup_{0\leq s<t\leq T}\frac{\|\cT_1f(t)-\cT_1f(s)\|_{L_p(\bR^d)}}{(t-s)^{\alpha-1/q}}\right)^q\right]\leq N\|f\|_{\bH_{p,q}^{\phi,2\beta-2}(T)}^q,\label{24.03.25.21.10}\\
    &\bE\left[\left(\sup_{0\leq s<t\leq T}\frac{\|\cT_2g(t)-\cT_2g(s)\|_{L_p(\bR^d)}}{(t-s)^{\alpha-1/q}}\right)^q\right]\leq N\|g\|_{\bH_{p,q}^{\phi,2\beta-1}(T,\ell_2)}^q,\label{24.03.25.21.11}
\end{align}
then
\begin{equation*}
    \begin{aligned}
        \mathbb{E}\left[\left(\sup_{0\leq s\leq t\leq T}\frac{\|\tilde{v}(t)-\tilde{v}(s)\|_{L_p(\bR^d)}}{(t-s)^{\alpha-1/q}}\right)^{q}\right]\leq N\left(\|u_0\|_{U_{p,q}^{\phi,2\beta}}^q+\|f\|_{\bH_{p,q}^{\phi,2\beta-2}(T)}^q+\|g\|_{\bH_{p,q}^{\phi,2\beta-1}(T,\ell_2)}^q\right).
    \end{aligned}
\end{equation*}
Now, we proceed to verify \eqref{24.03.25.21.05}, \eqref{24.03.25.21.10}, and \eqref{24.03.25.21.11} in the subsequent portions of this proof.

    \textbf{Step 2.} Proof of \eqref{24.03.25.21.05} for $1 \geq \beta > \alpha \geq 1/q$.
    
    Applying Lemma \ref{23.09.03.15.14}-$(ii)$ in conjunction with H\"older's inequality, we derive that
    \begin{align*}
        &\bE\left[\left(\sup_{0\leq s<t\leq T}\frac{\|T_tu_0-T_su_0\|_{L_p(\bR^d)}}{(t-s)^{\alpha-1/q}}\right)^{q}\right]\\
        &\leq N\int_0^T\int_0^T\frac{\bE\left[\|T_tu_0(r_1)-T_su_0(r_2)\|_{L_p(\bR^d)}^q\right]}{|r_1-r_2|^{\alpha q+1}}\mathrm{d}r_2\mathrm{d}r_1\\
        &\leq N\int_0^{T}\int_{0}^{T-l_1}\bE\left[\|(T_{l_1}-1)T_{l_2}u_0\|_{L_p(\bR^d)}^q\right]\mathrm{d}l_2\frac{\mathrm{d}l_1}{l_1^{1+\alpha q}}=:I_1(T).
    \end{align*}
    Utilizing Lemma \ref{23.09.03.15.14}-(i) and Theorem \ref{23.09.18.17.09}, we obtain
    \begin{align*}
        I_1(T)&\leq N\int_0^{T}\int_{0}^{T-l_1}\frac{\bE\left[\|T_{l_2}u_0\|_{H_p^{\phi,2\beta}(\bR^d)}^q\right]}{l_1^{1+(\alpha-\beta) q}}\mathrm{d}l_2\mathrm{d}l_1\leq N\int_0^{T}\frac{\|u_0\|_{U_{p,q}^{\phi,2\beta}}^q}{l_1^{1+(\alpha-\beta) q}}\mathrm{d}l_1= N\|u_0\|_{U_{p,q}^{\phi,2\beta}}^q.
    \end{align*}
    This calculation confirms \eqref{24.03.25.21.05}.

    \textbf{Step 3.} Proof of \eqref{24.03.25.21.10} for $1 \geq \beta > \alpha \geq 1/q$.
    
    Utilizing Lemma \ref{23.09.03.15.14}-$(ii)$ and employing H\"older's inequality, Minkowski's inequality, we deduce
    \begin{align*}
        &\bE\left[\left(\sup_{0\leq s<t\leq T}\frac{\|\cT_1f(t)-\cT_1f(s)\|_{L_p(\bR^d)}}{(t-s)^{\alpha-1/q}}\right)^{q}\right]\\
        &\leq N\int_0^T\int_0^T\frac{\bE\left[\|\cT_1f(r_1)-\cT_1f(r_2)\|_{L_p(\bR^d)}^q\right]}{|r_1-r_2|^{\alpha q+1}}\mathrm{d}r_2\mathrm{d}r_1\\
        &= N\int_0^{T}\int_{0}^{T-l_1}\frac{1}{l_1^{1+\alpha q}}\bE\left[\left\|(T_{l_1}-1)\cT_1f(l_2)+\int_0^{l_1}T_{\lambda}f(l_1+l_2-\lambda)\mathrm{d}\lambda\right\|_{L_p(\bR^d)}^q\right]\mathrm{d}l_2\mathrm{d}l_1\\
        &\leq N\int_0^{T}\int_{0}^{T-l_1}\frac{1}{l_1^{1+\alpha q}}\bE\left[\|(T_{l_1}-1)\cT_1f(l_2)\|_{L_p(\bR^d)}^q\right]\mathrm{d}l_2\mathrm{d}l_1\\
        &\quad+N\int_0^{T}\int_{0}^{T-l_1}\frac{1}{l_1^{1+\alpha q}}\mathbb{E}\left[\left(\int_0^{l_1}\|T_{\lambda}f(l_1+l_2-\lambda)\|_{L_p(\bR^d)}\mathrm{d}\lambda\right)^q\right]\mathrm{d}l_2\mathrm{d}l_1\\
        &=:I_{2,1}(T)+I_{2,2}(T).
    \end{align*}
    By Lemma \ref{23.09.03.15.14}-$(i)$ and Theorem \ref{23.09.18.17.09}-$(ii)$,
    \begin{align*}
        I_{2,1}(T)\leq N\int_0^T\int_0^{T-l_1}\frac{1}{l_1^{1+(\alpha-\beta)p}}\mathbb{E}\left[\|\cT_1f(l_2)\|_{H_p^{\phi,2\beta}(\bR^d)}^q\right]\mathrm{d}l_2\mathrm{d}l_1\leq N\|f\|_{\mathbb{H}_{p,q}^{\phi,2\beta-2}(T)}^q.
    \end{align*}
    Through H\"older's inequality, Fubini's theorem, and Lemma \ref{23.09.03.15.14}-$(i)$, it follows
    \begin{align*}
        I_{2,2}(T)&\leq N\int_0^{T}\int_{0}^{T-l_1}\frac{1}{l_1^{2+(\alpha-\beta) q}}\bE\left[\int_0^{l_1}\lambda^{(1-\beta)q}\|T_{\lambda}f(l_1+l_2-\lambda)\|_{L_p(\bR^d)}^q\mathrm{d}\lambda\right]\mathrm{d}l_2\mathrm{d}l_1\\
        &\leq N \int_0^{T}\int_{0}^{T-l_1}\frac{1}{l_1^{2+(\alpha-\beta) q}}\bE\left[\int_{0}^{l_1}\|f(l_2+\lambda)\|_{H_p^{\phi,2\beta-2}(\bR^d)}^q\mathrm{d}\lambda\right]\mathrm{d}l_2\mathrm{d}l_1\\
        &\leq N\|f\|_{\bH_{p,q}^{\phi,2\beta-2}(T)}^q.
    \end{align*}
This calculation confirms \eqref{24.03.25.21.10}.

    \textbf{Step 4.} This step is dedicated to establishing that \eqref{24.03.25.21.11} is valid when $1/2 \geq \beta > \alpha \geq 1/q$.
    
    Employing Lemma \ref{23.09.03.15.14}-$(ii)$ alongside H\"older's inequality, Minkowski's inequality we obtain
    \begin{align*}
        &\bE\left[\left(\sup_{0\leq s<t\leq T}\frac{\|\cT_2g(t)-\cT_2g(s)\|_{L_p(\bR^d)}}{(t-s)^{\alpha-1/q}}\right)^{q}\right]\\
        &\leq N\int_0^T\int_0^T\frac{\bE\left[\|\cT_2g(r_1)-\cT_2g(r_2)\|_{L_p(\bR^d)}^q\right]}{|r_1-r_2|^{\alpha q+1}}\mathrm{d}r_2\mathrm{d}r_1\\
        &= N\int_0^{T}\int_{0}^{T-l_1}\frac{1}{l_1^{1+\alpha q}}\bE\left[\left\|(T_{l_1}-1)\cT_2g(l_2)+\sum_{k=1}^{\infty}\int_{l_2}^{l_2+l_1}T_{l_1+l_2-\lambda}g^k(\lambda)\mathrm{d}w^k_{\lambda}\right\|_{L_p(\bR^d)}^q\right]\mathrm{d}l_2\mathrm{d}l_1\\
        &\leq N\int_0^{T}\int_{0}^{T-l_1}\frac{1}{l_1^{1+\alpha q}}\bE\left[\left\|(T_{l_1}-1)\cT_2g(l_2)\right\|_{L_p(\bR^d)}^q\right]\mathrm{d}l_2\mathrm{d}l_1\\
        &\quad +N\int_0^{T}\int_{0}^{T-l_1}\frac{1}{l_1^{1+\alpha q}}\bE\left[\left\|\sum_{k=1}^{\infty}\int_{l_2}^{l_2+l_1}T_{l_1+l_2-\lambda}g^k(\lambda)\mathrm{d}w^k_{\lambda}\right\|_{L_p(\bR^d)}^q\right]\mathrm{d}l_2\mathrm{d}l_1\\
        &=:I_{3,1}(T)+I_{3,2}(T).
    \end{align*}
    By Lemma \ref{23.09.03.15.14}-$(i)$ and Theorem \ref{23.09.18.17.09}-$(iii)$,
    \begin{align*}
        I_{3,1}(T)\leq N\int_0^T\int_0^{T-l_1}\frac{1}{l_1^{1+(\alpha-\beta)q}}\mathbb{E}\left[\|\cT_2g\|_{H_p^{\phi,2\beta}(\bR^d)}^q\right]\mathrm{d}l_2\mathrm{d}l_1\leq N\|g\|_{\mathbb{H}_{p,q}^{\phi,2\beta-1}(T,\ell_2)}^q.
    \end{align*}
Taking $\{\varphi_n^{\omega}\}_{n = 1}^{\infty} \subseteq C_c^{\infty}(\mathbb{R}^d)$ with $\|\varphi_n^{\omega}\|_{L_{p'}(\mathbb{R}^d)} \leq 1$ and
$$
\left\|\sum_{k=1}^{\infty}\int_{l_2}^{l_2+l_1}T_{l_1+l_2-\lambda}g^k(\lambda)\mathrm{d}w^k_{\lambda}\right\|_{L_p(\bR^d)}=\lim_{n\to\infty}\left|\int_{\bR^d}\sum_{k=1}^{\infty}\int_{l_2}^{l_2+l_1}T_{l_1+l_2-\lambda}g^k(\lambda,x)\mathrm{d}w^k_{\lambda}\varphi_n^{\omega}(x)\mathrm{d}x\right|.
$$
Applying Fatou's lemma, the Burkholder-Davis-Gundy inequality, and H\"older's inequality, it follows
\begin{align*}
    &\bE\left[\left\|\sum_{k=1}^{\infty}\int_{l_2}^{l_2+l_1}T_{l_1+l_2-\lambda}g^k(\lambda)\mathrm{d}w^k_{\lambda}\right\|_{L_p(\bR^d)}^q\right]\\
    &\leq \liminf_{n\to\infty}\bE\left[\left|\sum_{k=1}^{\infty}\int_{l_2}^{l_2+l_1}\left(\int_{\bR^d}T_{l_1+l_2-\lambda}g^k(\lambda,x)\varphi_n^{\omega}(x)\mathrm{d}x\right)\mathrm{d}w^k_{\lambda}\right|^q\right]\\
    &\leq \liminf_{n\to\infty}\bE\left[\left(\int_{l_2}^{l_2+l_1}\left|\int_{\bR^d}T_{l_1+l_2-\lambda}g(\lambda,x)\varphi_n^{\omega}(x)\mathrm{d}x\right|_{l_2}^2\mathrm{d}\lambda\right)^{q/2}\right]\\
    &\leq \bE\left[\left(\int_{0}^{l_1}\|T_{\lambda}g(l_1+l_2-\lambda)\|_{L_p(\bR^d;\ell_2)}^2\mathrm{d}\lambda\right)^{q/2}\right].
\end{align*}
Incorporating Lemma \ref{23.09.03.15.14}-(i) and \cite[Theorem 5.5.1]{grafakos2014classical}, we derive
\begin{align*}
    \left(\int_{0}^{l_1}\|T_{\lambda}g(l_1+l_2-\lambda)\|_{L_p(\bR^d;\ell_2)}^2\mathrm{d}\lambda\right)^{q/2}&\leq l_1^{\beta q-1}\int_0^{l_1}\lambda^{(1-2\beta)q/2}\|T_{\lambda}g(l_1+l_2-\lambda)\|_{L_p(\bR^d;\ell_2)}^q\mathrm{d}\lambda\\
    &\leq Nl_1^{\beta q-1}\int_0^{l_1}\|g(l_1+l_2-\lambda)\|_{H_p^{\phi,2\beta-1}(\bR^d;\ell_2)}^q\mathrm{d}\lambda.
\end{align*}
Therefore,
\begin{align*}
    I_{3,2}(T)&\leq N\int_0^{T}\int_{0}^{T-l_1}\frac{1}{l_1^{2+(\alpha-\beta) q}}\int_0^{l_1}\|g(l_1+l_2-\lambda)\|_{H_p^{\phi,2\beta-1}(\bR^d;\ell_2)}^q\mathrm{d}\lambda\mathrm{d}l_2\mathrm{d}l_1\\
    &\leq N\|g\|_{\bH_{p,q}^{\phi,2\beta-1}(T,\ell_2)}^q.
\end{align*}
This proves \eqref{24.03.25.21.11}.
The theorem is proved.
\end{proof}

\subsection{Non-explosivity of local solutions: Proof of Theorem \ref{local}-(iv)}
\label{24.03.22.14.29}

To prove Theorem \ref{local}-(iv), we require the auxiliary result presented in the following lemma.
\begin{lem}
\label{24.02.04.14.46}
    $(i)$ Define the function
    $$
    \chi(x):=\frac{1}{\cosh(|x|)}
    $$
    and let $X=(X_t)_{t\geq0}$ be a SBM with characteristic exponent $\phi$.
    Consider the function
    $$
    \psi_k(x):=\int_{0}^{\infty}\mathbb{E}[\chi((x+X_t)/k)]\mathrm{e}^{-t}\mathrm{d}t.
    $$
    Then, $\psi_k$ satisfies the integro-differential equation
    $$
    \phi(\Delta)\psi_k(x)-\psi_k(x)=-\chi(x/k),\quad \forall x\in\bR^d.
    $$
    Moreover,
    $$
    |\nabla\psi_k(x)|\leq \frac{\psi_k(x)}{k},\quad \psi_k(x)\leq \psi_{k+1}(x).
    $$

    $(ii)$ Assume the conditions of Theorem \ref{local}-(iv) are met and let $u_m$ be a solution to equation \eqref{23.07.08.13.42} as described in Theorem \ref{local}-(i).
    For any $\alpha \in (0, 1)$,
$$
    \bE\left[\sup_{t\leq\tau}\|u_m(t,\cdot)\|_{L_1(\bR^d)}^{\alpha}+\left(\int_0^{\tau}\int_{\bR^d}|u_m(t,x)|^{1+\lambda_{\mathrm{s.d.}}}h_m(u_m(t,x))\mathrm{d}x\mathrm{d}t\right)^{\alpha}\right]\leq N\|u_0\|_{L_1(\Omega\times\bR^d)}^{\alpha},
    $$
where $N = N(\alpha, T, \zeta)$.
\end{lem}
\begin{proof}
We address both parts of the lemma separately.

    $(i)$ Define the function $f_k(x) := \chi(x/k)$.
    Then,
    $
    \psi_k(x)=\int_0^{\infty}\bE[f_k(x+X_t)]\mathrm{e}^{-t}\mathrm{d}t$.
   Applying the Fourier transform to $\phi(\Delta)\psi_k - \psi_k$ yields
    \begin{align*}
        \cF[\phi(\Delta)\psi_k-\psi_k](\xi)&=(-\phi(|\xi|^2)-1)\cF[\psi_k](\xi)\\
        &=-(\phi(|\xi|^2)+1)\mathcal{F}[f_k](\xi)\int_0^{\infty}\mathbb{E}[\mathrm{e}^{i\xi\cdot X_t}]\mathrm{e}^{-t}\mathrm{d}t\\
        &=-(\phi(|\xi|^2)+1)\mathcal{F}[f_k](\xi)\int_0^{\infty}\mathrm{e}^{-t\phi(|\xi|^2)}\mathrm{e}^{-t}\mathrm{d}t=-\cF[f_k](\xi).
    \end{align*}
    This leads to $\phi(\Delta)\psi_k - \psi_k = f_k$.
    Noting that 
    $$
    |\nabla f_k(x)|\leq \frac{|\nabla\chi(x/k)|}{k}\leq \frac{\chi(x/k)}{k},
    $$
    we deduce
    \begin{align*}
        |\nabla\psi_k(x)|&\leq \int_{0}^{\infty}\mathbb{E}[|\nabla f_k(x+X_t)|]\mathrm{e}^{-t}\mathrm{d}t\\
        &\leq \frac{1}{k}\int_0^{\infty}\mathbb{E}[\chi((x+X_t)/k)]\mathrm{e}^{-t}\mathrm{d}t=\frac{\psi_k(x)}{k}.
    \end{align*}
    The second inequality is obtained by utilizing the monotonicity of $\chi$.

    $(ii)$ Employing the function $\psi_k$ defined in $(i)$ and considering $u_m$ as a non-negative solution to \eqref{23.07.08.13.42}, It\^o's formula reveals that for any stopping time $\tilde{\tau} \leq \tau \leq T$,
\begin{align*}
       (u_m(\tilde{\tau},\cdot),\psi_k)\mathrm{e}^{-\tilde{\tau}}&=(u_0,\psi_k)+\int_0^{\tilde{\tau}}\int_{\bR^d}u_m(s,x)(\phi(\Delta)\psi_k(x)-\psi_k(x))\mathrm{e}^{-s}\mathrm{d}x\mathrm{d}s\\
       &\quad+\int_0^{\tilde{\tau}}\int_{\bR^d}\zeta(s,x) F_m(s,u_m(s,x))\psi_k(x)\mathrm{e}^{-s}\mathrm{d}x\mathrm{d}s\\
       &\quad+
       \int_0^{\tilde{\tau}}\int_{\bR^d}B_m(s,u_m(s,x))\vec{b}(s,x)\cdot\nabla_x\psi_k(x)\mathrm{e}^{-s}\mathrm{d}x\mathrm{d}s\\
       &\quad+\sum_{i=1}^{\infty}\int_0^{\tilde{\tau}}\int_{\bR^d}\xi(s,x)\varphi_m(s,u_m(s,x))\psi_k(x)(f\ast e_i)(x)\mathrm{e}^{-s}\mathrm{d}x\mathrm{d}w^k_s.\\
       &=:(u_0,\psi_k)+I_1+I_2+I_3+I_4
   \end{align*}
Given $\mathbb{E}[I_4] = 0$, we focus on $I_1$, $I_2$, and $I_3$. Lemma \ref{24.02.04.14.46}-$(i)$ ensures $\mathbb{E}[I_1] \leq 0$, and
$$
\bE[I_3]\leq N(m)k^{-1}\bE\left[\int_0^{\tilde{\tau}}\int_{\bR^d}u_m(s,x)\psi_k(x)\mathrm{d}x\mathrm{d}s\right].
$$
For $I_2$, it holds that
$$
\bE[-I_2]\geq \mathrm{e}^{-T}c_{\mathrm{s.d.}}\inf_{\omega,t,x}\zeta(\omega,t,x)\bE\left[\int_0^{\tilde{\tau}}\int_{\bR^d}|u_m(s,x)|^{1+\lambda_{s.d.}}h_m(u_m(s,x))\psi_k(x)\mathrm{d}x\mathrm{d}s\right].
$$
By \cite[Theorem III.6.8]{Kry1995}, for any $\alpha \in (0,1)$,
\begin{align*}
    &\bE\left[\sup_{t\leq\tau}\left(\int_{\bR^d}u_m(t,x)\psi_k(x)\mathrm{d}x\right)^{\alpha}+\left(\int_0^{\tau}\int_{\bR^d}|u_m(s,x)|^{1+\lambda_{\mathrm{s.d.}}}h_m(u_m(s,x))\psi_k(x)\mathrm{d}x\mathrm{d}s\right)^{\alpha}\right]\\
    &\leq N_0\left(\int_{\bR^d}\bE[u_0(x)]\psi_k(x)\mathrm{d}x\right)^{\alpha}+\frac{N_1}{k^{\alpha}}\bE\left[\left(\int_0^{\tau}\int_{\bR^d}u_m(t,x)\psi_k(x)\mathrm{d}x\mathrm{d}t\right)^{\alpha}\right],
\end{align*}
where $N_0 = N_0(\alpha, c_{\mathrm{s.d.}}, T, \zeta)$ and $N_1 = N_1(\alpha, c_{\mathrm{s.d.}}, m, T, \zeta)$. For $k^{\alpha} \geq 2N_1T^{\alpha}$,
\begin{align*}
    &\bE\left[\sup_{t\leq\tau}\left(\int_{\bR^d}u_m(t,x)\psi_k(x)\mathrm{d}x\right)^{\alpha}+\left(\int_0^{\tau}\int_{\bR^d}|u_m(s,x)|^{1+\lambda_{\mathrm{s.d.}}}h_m(u_m(s,x))\psi_k(x)\mathrm{d}x\mathrm{d}s\right)^{\alpha}\right]\\
    &\leq N\left(\int_{\bR^d}\bE[u_0(x)]\psi_k(x)\mathrm{d}x\right)^{\alpha}.
\end{align*}
The proof concludes by applying the monotone convergence theorem, establishing
$$
\mathbb{E}\left[\sup_{t\leq\tau}\|u_m(t,\cdot)\|_{L_1(\bR^d)}^{\alpha}+\left(\int_0^{\tau}\int_{\bR^d}|u_m(s,x)|^{1+\lambda_{\mathrm{s.d.}}}h_m(u_m(s,x))\mathrm{d}x\mathrm{d}s\right)^{\alpha}\right]\leq N\|u_0\|_{L_1(\Omega\times\bR^d)}^{\alpha}.
$$
\end{proof}

\begin{proof}[Proof of Theorem \ref{local}-$(iv)$]
For fixed $m$ and $S$, we denote
$$
\tau_{m}(S):=\inf\left\{t\leq\tau:\sup_{s\leq t}\|u_m(s,\cdot)\|_{L_1(\bR^d)}+\||u_m|^{1+\lambda_{\mathrm{s.d.}}}h_m(u_m)\|_{L_1((0,t)\times\bR^d)}\geq S\right\}.
$$
Due to Lemma \ref{24.02.04.14.46}-$(ii)$, $\tau_m(S)$ is a stopping time.
Now, we consider the equation
\begin{equation}
    \label{23.07.25.16.05}
    \mathrm{d}v_m=(\phi(\Delta)v_m+\vec{b}\cdot\nabla_{x} (B_m(u_m)))\mathrm{d}t+\sum_{k=1}^{\infty}\xi\varphi_m(u_m)(\pi\ast e_k)\mathrm{d}w_t^k
\end{equation}
on $\opar0,\tau_m(S)\cbrk\times\bR^d$ with initial data $u_0$.

\textbf{Step 1.} Estimation of $\vec{b}\cdot\nabla_{x} (B_m(u_m))$.

One can check that
    \begin{align*}
    &\|\vec{b}(t,\cdot)\cdot\nabla_x(B_m(u_m(t,\cdot)))\|_{H_{p}^{\phi,\gamma-2}(\bR^d)}\\
    &\leq N\|B_m(u_m(t,\cdot))\|_{H_{p}^{\delta_0(\gamma-2)+1}(\bR^d)}\\
    &=N\left\|\int_{\bR^d}R_{\delta_0(2-\gamma)-1}(\cdot-y)B_m(y,u_m(t,y))\mathrm{d}y\right\|_{L_p(\bR^d)}\\
    &\leq N\left\|\int_{\bR^d}R_{\delta_0(2-\gamma)-1}(\cdot-y)|u_m(t,y)|^{1+\lambda_{\mathrm{b.}}}h_m(u_m(t,y))\mathrm{d}y\right\|_{L_p(\bR^d)}.
    \end{align*}
    Here, $R_{\delta_0(2-\gamma)-1}$ is the function in \eqref{24.03.04.15.14}.
    By H\"older's inequality and Minkowski's inequality, for $t\leq \tau_{m}(S)$,
\begin{align*}
&\left\|\int_{\bR^d}R_{\delta_0(2-\gamma)-1}(\cdot-y)|u_m(t,y)|^{1+\lambda_{\mathrm{b.}}}h_m(u_m(t,y))\mathrm{d}y\right\|_{L_p(\bR^d)}^p \\
&=\int_{\bR^{d}}\left|\int_{\bR^d}R_{\delta_0(2-\gamma)-1}(x-y)|u_m(t,y)|^{1+\lambda_{\mathrm{b.}}}h_m(u_m(t,y))\mathrm{d}y\right|^{p}\mathrm{d}x\\
&\leq\|u_m(t,\cdot)\|_{L_1(\bR^d)}^{\left(1-\frac{1}{r}\right)p}\int_{\bR^{d}}\left|\int_{\bR^d}\left|R_{\delta_0(2-\gamma)-1}(x-y)\right|^{r}|u_m(t,y)|^{1+r\lambda_{\mathrm{b.}}}h_m(u_m(t,y))\mathrm{d}y\right|^{\frac{p}{r}}\mathrm{d}x\\
&\leq S^{\left(1-\frac{1}{r}\right)p} \int_{\bR^{d}}\left|\int_{\bR^d}\left|R_{\delta_0(2-\gamma)-1}(y)\right|^{r}|u_m(t,x-y)|^{1+r\lambda_{\mathrm{b.}}}h_m(u_m(t,x-y))\mathrm{d}y\right|^{\frac{p}{r}}\mathrm{d}x \\
&\leq S^{\left(1-\frac{1}{r}\right)p} \|R_{\delta_0(2-\gamma)-1}\|_{L_r(\bR^d)}^p  \int_{\bR^{d}}|u_m(t,x)|^{\frac{p}{r}+p\lambda_{\mathrm{b.}}}h_m(u_m(t,x))\mathrm{d}x,
\end{align*}
where $1\leq r\leq p$.
Due to \eqref{24.04.06.19.26}, $q\in(2,\infty)$ and $\delta_0\in(0,1]$, we have
$$
p>\frac{d}{d+1-\delta_0(2-\gamma)}.
$$
With the help of \eqref{24.04.24.19.45}, we can choose $r\geq1$ such that
\begin{equation}
\label{25.08.06.17.59}
\frac{1}{p}<\frac{d+1-\delta_0(2-\gamma)}{d}< \frac{1}{r}<1+\frac{1}{p}+\frac{\lambda_{\mathrm{s.d.}}}{p\vee q}-\lambda_{\mathrm{b.}}.
\end{equation}
Moreover, $r$ also satisfies
\begin{equation}
\label{24.03.31.15.59}
    0<\frac{1}{r}+\lambda_{\mathrm{b.}}-\frac{1}{p}-\frac{\lambda_{\mathrm{s.d.}}}{p\vee q}<1.
\end{equation}
Since $\delta_0(2 - \gamma) - 1 \in (0,1)$, it follows from \eqref{24.04.09.15.21} and \eqref{25.08.06.17.55} that
$$
    \|R_{\delta_0(2 - \gamma) - 1}\|_{L_r(\mathbb{R}^d)} < \infty 
    \quad \Longleftrightarrow \quad r < \frac{d}{d + 1 - \delta_0(2 - \gamma)}.
$$
This condition is clearly satisfied by the choice of $r$ in \eqref{25.08.06.17.59}.
If $q\leq p$, then by H\"older's inequality,
\begin{align*}
    &\int_0^{\tau_m(S)}\left(\int_{\bR^{d}}|u_m(t,x)|^{\frac{p}{r}+p\lambda_{\mathrm{b.}}}h_m(u_m(t,x))\mathrm{d}x\right)^{\frac{q}{p}}\mathrm{d}t\\
    &\leq \tau_m(S)^{1-\frac{q}{p}}\left(\sup_{t\leq\tau_m(S)}\sup_{x\in\bR^d}|u_m(t,x)|\right)^{\frac{q}{r}+q\lambda_{\mathrm{b.}}-\frac{q}{p}(1+\lambda_{\mathrm{s.d.}})}\\
    &\qquad \times \left(\int_{0}^{\tau_m(S)}\int_{\bR^d}|u_m(t,x)|^{1+\lambda_{\mathrm{s.d.}}}h_m(u_m(t,x))\mathrm{d}x\mathrm{d}t\right)^{\frac{q}{p}}\\
    &\leq S^{\frac{q}{p}}T^{1-\frac{q}{p}}\left(\sup_{t\leq\tau_m(S)}\sup_{x\in\bR^d}|u_m(t,x)|\right)^{\frac{q}{r}+q\lambda_{\mathrm{b.}}-\frac{q}{p}(1+\lambda_{\mathrm{s.d.}})}.
\end{align*}
If $p<q$, then
\begin{align*}
    &\int_0^{\tau_m(S)}\left(\int_{\bR^{d}}|u_m(t,x)|^{\frac{p}{r}+p\lambda_{\mathrm{b.}}}h_m(u_m(t,x))\mathrm{d}x\right)^{\frac{q}{p}}\mathrm{d}t\\
    &=\int_0^{\tau_m(S)}\left(\int_{\bR^{d}}|u_m(t,x)|^{\frac{p}{r}+p\lambda_{\mathrm{b.}}}h_m(u_m(t,x))\mathrm{d}x\right)^{\frac{q}{p}-1}\\
    &\qquad \times\left(\int_{\bR^{d}}|u_m(t,x)|^{\frac{p}{r}+p\lambda_{\mathrm{b.}}}h_m(u_m(t,x))\mathrm{d}x\right)\mathrm{d}t\\
    &\leq \left(\sup_{t\leq\tau_m(S)}\|u_m(t,\cdot)\|_{L_1(\bR^d)}\right)^{\frac{q}{p}-1}\times\left(\sup_{t\leq\tau_m(S)}\sup_{x\in\bR^d}|u_m(t,x)|\right)^{\frac{q}{r}+q\lambda_{\mathrm{b.}}-\frac{q}{p}-\lambda_{\mathrm{s.d.}}}\\
    &\qquad\times\int_{0}^{\tau_m(S)}\int_{\bR^d}|u_m(t,x)|^{1+\lambda_{\mathrm{s.d.}}}h_m(u_m(t,x))\mathrm{d}x\mathrm{d}t\\
    &\leq S^{\frac{q}{p}}\left(\sup_{t\leq\tau_m(S)}\sup_{x\in\bR^d}|u_m(t,x)|\right)^{\frac{q}{r}+q\lambda_{\mathrm{b.}}-\frac{q}{p}-\lambda_{\mathrm{s.d.}}}.
\end{align*}
Therefore,
\begin{align*}
    &\|\vec{b}\cdot\nabla_x(B_m(u_m))\|_{\bH_{p,q}^{\phi,\gamma-2}(\tau_m(S))}\\
    &\leq NS^{1+\frac{1}{p}-\frac{1}{r}}\bE\left[\left(\sup_{t\leq\tau_m(S)}\sup_{x\in\bR^d}|u_m(t,x)|\right)^{\frac{q}{r}+q\lambda_{\mathrm{b.}}-\frac{q}{p}-\frac{q\lambda_{\mathrm{s.d.}}}{p\vee q}}\right]^{1/q}.
\end{align*}
Due to \eqref{24.03.31.15.59}, we have
$$
\|\vec{b}\cdot\nabla_x(B_m(u_m))\|_{\bH_{p,q}^{\phi,\gamma-2}(\tau_m(S))}\leq NS^{1+\frac{1}{p}-\frac{1}{r}}\bE[|u_m|_{C_{t,x}([0,\tau_m(S)]\times\bR^d)}^q]^{\frac{1}{q}\left(\frac{1}{r}+\lambda_{\mathrm{b.}}-\frac{1}{p}-\frac{\lambda_{\mathrm{s.d.}}}{p\vee q}\right)}.
$$

\textbf{Step 2.} Estimation of $\xi\varphi_m(u_m)\boldsymbol{\pi}$, where $\boldsymbol{\pi}:=(\pi\ast e_1,\pi\ast e_2,\cdots)$.

We modify the proof of \cite[Lemma 4.4]{CH2021}.
One can check that
$$
\|\xi(t,\cdot)\varphi_m(u_m(t,\cdot))\boldsymbol{\pi}\|_{H_{p}^{\phi,\gamma-1}(\ell_2)}^p\leq N\int_{\bR^d}\|R_{\delta_0(1-\gamma)}(x-\cdot)|u_m(t,\cdot)|^{1+\lambda_{\mathrm{s.m.}}}h_m(u_m(t,\cdot))\|_{\mathfrak{H}}^{p}\mathrm{d}x,
$$
where
\begin{align*}
    &\|R_{\delta_0(1-\gamma)}(x-\cdot)f\|_{\mathfrak{H}}^2\\
    &=\int_{\bR^d}\int_{\bR^d}R_{\delta_0(1-\gamma)}(x-(y-z))R_{\delta_0(1-\gamma)}(x+z)f(y-z)\overline{f(-z)}\mathrm{d}z\pi(\mathrm{d}y).
\end{align*}
By H\"older's inequality, for $t\leq\tau_m(S)$,
\begin{align*}
    &\|R_{\delta_0(1-\gamma)}(x-\cdot)|u_m(t,\cdot)|^{1+\lambda_{\mathrm{s.m.}}}h_m(u_m(t,\cdot))\|_{\mathfrak{H}}^2\\
    &\leq\left(\int_{\bR^d}\int_{\bR^d}|u_m(t,y-z)|^{\frac{1}{2}}|u_m(t,-z)|^{\frac{1}{2}}\mathrm{d}z\frac{\pi(\mathrm{d}y)}{(1+|y|^2)^{k/2}}\right)^{1-\frac{1}{r}}\\
    &\quad\times\bigg(\int_{\bR^d}\int_{\bR^d}|R_{\delta_0(1-\gamma)}(x-(y-z))R_{\delta_0(1-\gamma)}(x-z)|^{r}\\
    &\quad\quad\quad\times|u_m(t,y-z)|^{\frac{r+1}{2}+r\lambda_{\mathrm{s.m.}}}|u_m(t,-z)|^{\frac{r+1}{2}+r\lambda_{\mathrm{s.m.}}}\\
    &\quad\quad\quad\times h_m(u_m(t,y-z))h_m(u_m(t,-z))(1+|y|^2)^{k(r-1)/2}\mathrm{d}z\pi(\mathrm{d}y)\bigg)^{\frac{1}{r}}\\
    &\leq A_{\pi}^{1-\frac{1}{r}}\|u_m(t,\cdot)\|_{L_1(\bR^d)}^{1-\frac{1}{r}}\Gamma(x)^{\frac{1}{r}}\leq A_{\pi}^{1-\frac{1}{r}}S^{1-\frac{1}{r}}\Gamma(x)^{\frac{1}{r}},
\end{align*}
where $A_{\pi}$ is the constant defined in \eqref{24.04.21.15.46} and
\begin{align*}
\Gamma(x)&:=\int_{\bR^d}\int_{\bR^d}|R_{\delta_0(1-\gamma)}(x-(y-z))R_{\delta_0(1-\gamma)}(x-z)|^{r}\\
&\quad\quad\quad\times|u_m(t,y-z)|^{\frac{r+1}{2}+r\lambda_{\mathrm{s.m.}}}|u_m(t,-z)|^{\frac{r+1}{2}+r\lambda_{\mathrm{s.m.}}}\\
&\quad\quad\quad\times h_m(u_m(t,y-z))h_m(u_m(t,-z))(1+|y|^2)^{k(r-1)/2}\mathrm{d}z\pi(\mathrm{d}y).
\end{align*}
Since $r<p/2$, Minkowski's inequality yields that
\begin{align*}
\int_{\bR^d}\Gamma(x)^{\frac{p}{2r}}\mathrm{d}x&\leq \left(\int_{\bR^d}(R_{\delta_0(1-\gamma)}^r\ast R_{\delta_0(1-\gamma)}^r)(y)(1+|y|^{2})^{k(r-1)/2}\pi(\mathrm{d}y)\right)^{\frac{p}{2r}}\\
&\quad\times\left(\int_{\bR^d}|u_m(t,x)|^{\frac{p(r+1)}{2r}+p\lambda_{\mathrm{s.m.}}}h_m(u_m(t,x))\mathrm{d}x\right)\\
&=:I^{\frac{p}{2r}}\int_{\bR^d}|u_m(t,x)|^{\frac{p(r+1)}{2r}+p\lambda_{\mathrm{s.m.}}}h_m(u_m(t,x))\mathrm{d}x.
\end{align*}
Since $(R_{\delta_0(1-\gamma)}^r\ast R_{\delta_0(1-\gamma)}^r)$ decays exponentially as $|x|\to\infty$,
$$
I<\infty \Longleftrightarrow \int_{|y|< 1}(R_{\delta_0(1-\gamma)}^r\ast R_{\delta_0(1-\gamma)}^r)(y)\pi(\mathrm{d}y)<\infty.
$$
Due to Proposition \ref{24.04.21.21.05}, $I<\infty$.
Therefore,
\begin{align*}
&\int_0^{\tau_m(S)}\left(\int_{\bR^d}|u_m(t,x)|^{\frac{p(r+1)}{2r}+p\lambda_{\mathrm{s.m.}}}h_m(u_m(t,x))\mathrm{d}x\right)^{\frac{p}{q}}\mathrm{d}t\\
&\leq NS^{1+\frac{1}{p}-\frac{1}{r}}\mathbb{E}\left[\left(\sup_{t\leq\tau_m(S)}\sup_{x\in\bR^d}|u_m(t,x)|\right)^{\frac{q(r+1)}{2r}+q\lambda_{\mathrm{s.m.}}-\frac{q}{p}-\frac{q\lambda_{\mathrm{s.d.}}}{p\vee q}}\right]^{1/q}\\
&\leq NS^{1+\frac{1}{p}-\frac{1}{r}}\mathbb{E}\left[|u_m|_{C_{t,x}([0,\tau_m(S)]\times\bR^d)}^q\right]^{\frac{1}{q}\left(\frac{r+1}{2r}+\lambda_{\mathrm{s.m.}}-\frac{1}{p}-\frac{\lambda_{\mathrm{s.d.}}}{p\vee q}\right)}.
\end{align*}

\textbf{Step 3.} Conclusion.

Therefore, by Theorem \ref{23.07.04.17.12}, equation \eqref{23.07.25.16.05} has a unique solution $v_m\in\cH_{p,q}^{\phi,\gamma}(\tau_m(S))$ with estimate
\begin{align*}
    \|v_m\|_{\cH_{p,q}^{\phi,\gamma}(\tau_m(S))}&\leq N\bigg(\|u_0\|_{U_{p,q}^{\phi,\gamma}}+\bE[|u_m|_{C_{t,x}([0,\tau_m(S)]\times\bR^d)}^q]^{\frac{1}{q}\left((1+\lambda_{\mathrm{s.m.}})-\frac{1}{p}-\frac{\lambda_{\mathrm{s.d.}}}{p\vee q}\right)}\\
    &\qquad\quad+\bE[|u_m|_{C_{t,x}([0,\tau_m(S)]\times\bR^d)}^q]^{\frac{1}{q}\left((1+\lambda_{\mathrm{b.}})-\frac{1}{p}-\frac{\lambda_{\mathrm{s.d.}}}{p\vee q}\right)}\bigg),
\end{align*}
where $N=N(S)$.
If we put $z_m:=u_m-v_m$, then $z_m$ satisfies
$$
\mathrm{d}z_m=(\phi(\Delta)z_m+\zeta F_m(u_m))\mathrm{d}t,\quad z_m(0)=0.
$$
It can be easily checked that
$$
z_m(t)=\int_0^tT_{t-s}(\zeta F_m(u_m))\mathrm{d}s\leq 0.
$$
This implies that $u_m\leq v_m$, thus, by Young's inequality
$$
\|v_m\|_{\cH_{p,q}^{\phi,\gamma}(\tau_m(S))}\leq N(S)\|u_0\|_{U_{p,q}^{\phi,\gamma}}+N(\varepsilon,S)+\varepsilon\bE[|v_m|_{C_{t,x}([0,\tau_m(S)]\times\bR^d)}^q]^{1/q}.
$$
By Theorem \ref{23.09.05.17.43},
$$
\|v_m\|_{\cH_{p,q}^{\phi,\gamma}(\tau_m(S))}\leq N(S)\|u_0\|_{U_{p,q}^{\phi,\gamma}}+N(\varepsilon,S)+N\varepsilon\|v_m\|_{\cH_{p,q}^{\phi,\gamma}(\tau_m(S))}.
$$
If we choose sufficiently small $\varepsilon>0$, then we have
\begin{equation}
\label{23.09.29.00.30}
    \|v_m\|_{\cH_{p,q}^{\phi,\gamma}(\tau_m(S))} \leq N(S)\|u_0\|_{U_{p,q}^{\phi,\gamma}}+N(\varepsilon,S)=:N_S.
\end{equation}
By making use of Chebyshev's inequality, Theorem \ref{23.09.05.17.43} and \eqref{23.09.29.00.30},
\begin{align*}
    &\bP\left(\sup_{t\leq\tau}\sup_{x\in\bR^d}|u_m(t,x)|\geq R\right)\\
    &\leq \bP\left(\sup_{t\leq\tau_m(S)\wedge\tau}\sup_{x\in\bR^d}|v_m(t,x)|\geq R\right)+\bP(\tau_{m}(S)\leq\tau)\\
    &\leq \frac{N_S}{R}+\bP\left(\sup_{t\leq\tau}\|u_m(t,\cdot)\|_{L_1(\bR^d)}+\int_0^{\tau}\int_{\bR^d}|u_m(t,x)|^{1+\lambda_{\mathrm{s.d.}}}h_m(u_m(t,x))\mathrm{d}x\mathrm{d}t\geq S\right)\\
    &\leq \frac{N_S}{R}+\frac{N_{\alpha}}{S^{\alpha}}.
\end{align*}
Note that $N_{\alpha}$ is independent of $m$ and $S$.
Since $N_S$ and $N_{\alpha}$ are independent of $m$, the theorem is proved.
\end{proof}

\vspace{1cm}
\textbf{Declarations of interest.}
Declarations of interest: none

\textbf{Data Availability.}
Data sharing not applicable to this article as no datasets were generated or analysed during the current study.

\textbf{Acknowledgements.}
The authors are highly indebted to the anonymous reviewers for their careful reading and insightful comments, which substantially improved the presentation of this paper.
J.-H. Choi has been supported by a KIAS Individual Grant (MG102701) at Korea Institute for Advanced Study.
B.-S. Han has been supported by the National Research Foundation of Korea (NRF) grant funded by the Korea government (MSIT) (No. NRF-2021R1C1C2007792).
D. Park has been supported by the 2025 Research Grant from Kangwon National University (grant no. 202504300001).
The authors would also like to thank Dr. Junhee Ryu and Dr. Jinsol Seo for giving many helpful comments.

\bibliographystyle{plain}

\end{document}